\documentclass[12pt]{amsart}
\usepackage{bbm}
\usepackage{amscd, amssymb, dsfont, upgreek, mathrsfs} 
\usepackage{hyperref} 
\usepackage{xcolor}
\input{xy}
\xyoption{all}

\calclayout

\newtheorem{Thm}{Theorem}[section]
\newtheorem{Conj}[Thm]{Conjecture}
\newtheorem{Prop}[Thm]{Proposition}
\newtheorem{Def}[Thm]{Definition}
\newtheorem{Def/Thm}[Thm]{Definition/Theorem}
\newtheorem{Cor}[Thm]{Corollary}
\newtheorem{Lemma}[Thm]{Lemma}

\theoremstyle{remark}
\newtheorem{Rmk}[Thm]{Remark}

\numberwithin{equation}{subsection}
\newcommand{\ti }{\times}
\newcommand{\ot }{\otimes}
\newcommand{\ra }{\rightarrow}

\newcommand{\lra }{\longrightarrow}

\newcommand{\Hom }{{\mathrm{Hom}}}

\newcommand{\Spec}{{\mathrm{Spec}}}

\newcommand{\Pic}{{\mathrm{Pic}}}
\newcommand{\Sym}{{\mathrm{Sym}}}

\newcommand{\rank }{{\mathrm{rank}}}

\newcommand{\cO}{{\mathcal{O}}}

\newcommand{\cF}{{\mathcal{F}}}

\newcommand{\cY}{{\mathcal{Y}}}

\newcommand{\cX}{{\mathcal{X}}}

\newcommand{\fC}{{\mathfrak{C}}}

\newcommand{\G}{{\bf G}}

\newcommand{\NN}{{\mathbb N}}

\newcommand{\PP }{{\mathbb P}}

\newcommand{\QQ }{{\mathbb Q}}
\newcommand{\CC }{{\mathbb C}}
\newcommand{\ZZ }{{\mathbb Z}}

\newcommand{\one }{{\mathbbm 1}}

\newcommand{\ke }{{\varepsilon }}

\newcommand{\kG }{{\Gamma}}

\newcommand{\kl }{{\lambda}}

\newcommand{\Up}{\Upsilon}

\newcommand{\fBun}{\mathfrak{B}un_{\bf {G}}}

\newcommand{\vir}{\mathrm{vir}}

\newcommand{\X}{\mathfrak{X}}

\newcommand{\WmodtG}{W/\!\!/_{\!\theta}\G}

\newcommand{\BunG}{\mathfrak{B}un_\G}
\newcommand{\Bund}{\mathfrak{B}un_{\G, d}^{g,k}}

\newcommand{\fM}{\mathfrak{M}}

\newcommand{\lan}{\langle}
\newcommand{\ran}{\rangle}

\newcommand{\id}{\mathrm{id}}

\newcommand{\pr}{\mathrm{pr}}

\newcommand{\rk}{\mathrm{rk}\,}

\newcommand{\tP}{{\tt P}}
\newcommand{\tR}{{\tt R}}
\newcommand{\tF}{{\tt F}}

\begin{document}

\title[Quasimap Wall-crossings and Mirror Symmetry]{Quasimap Wall-crossings and Mirror Symmetry}
\begin{abstract} 
We state a wall-crossing formula for the virtual classes of $\ke$-stable quasimaps to GIT quotients and prove it for complete intersections in projective
space, with no positivity restrictions on their first Chern class. As a consequence, the wall-crossing formula relating the genus $g$ descendant Gromov-Witten potential and the genus $g$ $\ke$-quasimap
descendant potential is established. For the quintic threefold, our results may be interpreted as giving a
rigorous and geometric interpretation of the holomorphic limit of the BCOV $B$-model partition function
of the mirror family.
\end{abstract}

\author{Ionu\c t Ciocan-Fontanine}
\noindent\address{School of Mathematics, University of Minnesota, 206 Church St. SE,
Minneapolis MN, 55455, and\hfill
\newline \indent School of Mathematics, Korea Institute for Advanced Study,
85 Hoegiro, Dongdaemun-gu, Seoul, 02455, Korea}
\email{ciocan@math.umn.edu}

\author{Bumsig Kim}
\address{School of Mathematics, Korea Institute for Advanced Study,
85 Hoegiro, Dongdaemun-gu, Seoul, 02455, Korea}
\email{bumsig@kias.re.kr}

\date{\today}

\maketitle



\newcommand{\WtG}{W/\!\!/_{\!\theta}\G}
\newcommand{\uX}{\underline{X}}
\newcommand{\uC}{\underline{C}}
\newcommand{\ux}{\underline{x}}
\newcommand{\uL}{\underline{L}}
\newcommand{\tw}{\mathrm{tw}}
\newcommand{\fB}{\mathfrak{B}}
\newcommand{\fQmap}{\mathfrak{Qmap}}
\newcommand{\un}{\mathrm{un}}
\newcommand{\fpQmap}{\mathfrak{pQmap}}
\newcommand{\reg}{\mathrm{reg}}
\newcommand{\Qke}{Q_{g,k}^{\ke}(X, \beta )}
\newcommand{\Qked}{Q_{g,k}^{\ke}(X, d )}
\newcommand{\Eff}{\mathrm{Eff}}
\newcommand{\Kke}{\mathcal{K}_{g, k}^{\ke}((\X, X), \beta )}
\newcommand{\AAA}{\mathbb{A}}

\newcommand{\bx}{\mathbf{x}}
\newcommand{\Kked}{\mathcal{K}_{g, k}(X, d)}
\newcommand{\QGke}{QG^{\ke}_{0, k, \beta}(X)}
\newcommand{\bIX}{\bar{I}_{\mu}X}

\newcommand{\fU}{\mathfrak{U}}

\newcommand{\QV}{\fM_{g, k}([V/G], d)}
\newcommand{\QkeV}{Q^{\ke}_{g, k}(([V/G], [V^{ss}/G]), d)}
\newcommand{\bL}{\mathbb{L}}

\newcommand{\QMe}{Q^\ke_{g, k | A}(X, \beta _0 )}
\newcommand{\QM}{Q_{g, k | A}(X, \beta _0 )}
\newcommand{\MM}{\overline{M}_{g, k}(X, \beta )}

\newcommand{\Bun}{\fB un_{\G}^{g, k}}
\newcommand{\fV}{\mathfrak{V}}

\newcommand{\Qtw}{Q^{\ke, \tw}_{g, k}(\PP (V), \beta)}

\newcommand{\MMm}{\overline{M}_{g, k}^{\tw}(\PP(V), \beta)}
\newcommand{\QQm}{Q^{\ke, \tw}_{g, k}(\PP(V), \beta)}

\newcommand{\MMA}{\overline{M}_{g, k+A}^{\tw}(\PP(V), \beta _0)}
\newcommand{\QQA}{Q^{\ke, \tw}_{g, k+A}(\PP(V), \beta_0)}
\newcommand{\MMa}{\overline{M}_{0, a}(\mathds{V}, \beta _a)}

\newcommand{\MMaa}{\overline{M}_{0, 1+a}(\mathds{V}, \beta _a)}

\newcommand{\QQqm}{Q_{g, k}^{\tw}(\PP(V), \beta)}
\newcommand{\QQqA}{Q_{g, k|A}^{\tw} (\PP(V), \beta_0)}

\newcommand{\oG}{\overline{\Gamma}}
\newcommand{\cW}{\mathcal{W}}

\newcommand{\sA}{\mathscr{A}}
\newcommand{\sR}{\mathscr{R}}
\newcommand{\sQ}{\mathscr{Q}}
\newcommand{\sP}{\mathscr{P}}
 \newcommand{\sL}{\mathscr{L}}
 \newcommand{\sD}{\mathscr{D}}
 \newcommand{\sW}{\mathscr{W}}
 \newcommand{\sI}{\mathscr{I}}
 \newcommand{\sV}{\mathscr{V}}
 \newcommand{\ovG}{\overline{\Gamma}}

\newcommand{\Gr}{\mathbf{Gr}}
\newcommand{\bC}{\mathbf{C}}
\newcommand{\sM}{\mathscr{M}}
\newcommand{\mpr}{\mathrm{Pr}}

\newcommand{\PsA}{\PP ^{sep}_{j_A}}
\newcommand{\hatU}{\hat{\Upsilon}}

\newcommand{\hG}{\hat{\Gamma}}

\tableofcontents

\section{Introduction}

\subsection{Overview}
Let $W$ be a complex affine variety acted upon by a reductive algebraic group $\G$.
Fix a character $\theta$ of $\G$ for which the induced $\G$-action on
the $\theta$-semistable locus $W^{ss}$ is free. For the quasiprojective target $\WmodtG$ and a rational number $\ke >0$, or for $\ke =0+$,
the notion of $\ke$-stable quasimaps to $\WmodtG$ was introduced in \cite{CKM}, inspired by \cite{MOP, MM, CK}. 
They are in fact suitable maps from curves to the stack quotient 
$[W/\G]$. 
The Deligne-Mumford moduli stack $Q^{\ke}_{g, k}(\WmodtG, \beta )$ of $\ke$-stable quasimaps of type $(g, k, \beta)$
is proper over $\CC$ if $\WmodtG$ is projective. Here
 $g$, $k$, and $\beta$ are respectively the genus of the domain curve, the number of markings, 
and the numerical class  $\beta \in \Hom_\ZZ (\Pic ([W/\G], \ZZ )$) of the quasimaps.
If $W$ has at worst lci singularities and $W^{ss}$ is smooth (as always assumed in this paper), the moduli stacks carry canonical virtual fundamental classes. 
There are evaluation maps $ev_j$ to $\WmodtG$, as well as cotangent psi-classes $\psi _j$ at 
the $j$-th marking.
Hence, we may define
descendant $\ke$-quasimap invariants
\begin{equation}\label{invariants} 
\lan \gamma _1\psi ^{a_1}_1,\dots, \gamma _k \psi _k^{a_k} \ran _{g, k, \beta }^{\ke} =
\int_{[Q^{\ke}_{g, k}(\WmodtG, \beta ) ]^{\mathrm{vir}}} \prod_{j=1}^k\psi^{a_j}_j ev_j^*\gamma_j\end{equation}
for $\gamma _i \in A^*(\WmodtG )_{\QQ}$ and $a_i \in \ZZ _{\ge 0}$. Here and for the rest of the paper, the Chow cohomology $A^*(Y)_\QQ$ of a Deligne-Mumford stack $Y$ is the algebra $A^*(Y\stackrel{\id}{\ra} Y)_\QQ $ of bivariant classes, see \cite[\S17.3]{Fu} and \cite[\S5]{Vistoli}.

There is a wall-and-chamber structure on the space $\QQ_{>0}$ of stability parameters. Assuming for 
simplicity $(g,k)\neq(0,0)$, the walls are at $\ke=1/n$ with $n\in\NN$ and the moduli spaces stay constant in each chamber $(\frac{1}{n+1}, \frac{1}{n}]$. For $ \ke\in (1,\infty)$, they
parametrize exactly stable maps to $\WmodtG$. 
A conjectural wall-crossing formula for the invariants of {\it semi-positive} targets was stated in the paper \cite{CKg}, and was proved for semi-positive (quasiprojective) toric
quotients by localization techniques. 
In this paper we propose a geometric wall-crossing formula {\it at the level of virtual classes} and {\it without} any positivity restrictions (which, as we show,
implies the above mentioned semi-positive numerical wall-crossing, see Corollary \ref{numerical2}). The main result of the paper is a proof of the virtual class wall-crossing formula for complete intersections in
projective spaces.

The wall-crossing formula has important applications to Mirror Symmetry for Calabi-Yau threefolds at higher genus. This is explained in \S\ref{mirror symm}, the main point being that, assuming the Mirror Conjecture, the genus $g$ partition function of quasimap theory for the $\ke=0+$ stability of a Calabi-Yau threefold is precisely {\it equal} to (the holomorphic limit of) the $B$-model partition function of the mirror Calabi-Yau family, introduced in string theory by Bershadsky, Cecotti, Ooguri, and Vafa.

\subsection{Geometric wall-crossing} To state the wall-crossing formula, we recall some facts from quasimap theory and fix some notation.

The  monoid ${\mathrm {Eff}}(W,\G,\theta)$ of $\theta$-effective numerical classes is the submonoid of the additive group
$\Hom_\ZZ (\Pic ([W/\G], \ZZ )$ consisting of classes represented by $\theta$ quasimaps (possibly with disconnected domain curves). The Novikov ring of the theory is
$$\QQ[[q]]:=\left\{ \sum_{{\mathrm {Eff}}(W,\G,\theta)}a_\beta q^\beta \; \mid \; a_\beta\in \QQ\right\},$$
the $q$-adic completion of the semigroup ring $\QQ[{\mathrm {Eff}}(W,\G,\theta)]$.

The GIT set-up gives (see \cite[\S3.1]{CKg0} for details) a natural morphism $i:[W/\G] \ra [\CC ^{m+1}/\CC ^*]$ 
for some $m\in\ZZ_+$,
inducing a closed immersion $i:\WmodtG \hookrightarrow \PP ^m$ and also a morphism (denoted by the same letter)
\[ i : Q^{\ke}_{g, k } (\WmodtG , \beta  ) \ra Q^{\ke}_{g, k } (\PP ^m , d(\beta ) ), \]
where $d(\beta) := i_*(\beta ) \in \Hom (\Pic ([\CC ^{m+1}/\CC ^*]), \ZZ)\cong\ZZ$.

Fix a positive rational number $\ke_0$ such that $1/\ke_0 $ is an integer
and let $\ke _+>\ke_0\geq \ke _-$ be rational numbers in the two adjacent stability chambers separated by the wall $\ke_0$.
There is a morphism
\[ c : Q^{\ke_+}_{g, k} (\PP ^m, d(\beta ) ) \ra Q^{\ke _{-}}_{g, k } (\PP ^m , d(\beta) ) \]
which contracts rational tails of degree $1/\ke_0$, see \cite{Toda}.

Let $A$ denote a finite index set of cardinality $1,2,3,\dots $ Consider splittings 
$\beta=\beta_0+\sum_{a\in A}\beta_{a}$ into $\theta$-effective numerical classes such that 
$d(\beta _a) = 1/ \ke _0$ for all $a\in A$.
There is a natural morphism 
\[ b_{A} : Q^{\ke_-}_{g, k + A} (\PP ^m , d(\beta _0) ) \ra Q^{\ke_-}_{g, k } (\PP ^m, d(\beta ) ) \]
which trades the markings in $A$ for base points of length $1/\ke _0$ (\cite[\S3.2]{CKg0}).

Finally, recall from \cite[\S7]{CKM} and \cite[\S5]{CKg0} that for every triple $(W,G,\theta)$, with associated quotient $X=\WmodtG$, there is a corresponding {\it small $I$-function}, denoted
$I_{sm}(q,z)$. The precise definition we will use in this paper is Definition 5.1.1 in \cite{CKg0}, specialized at $\ke=0+$ and ${\bf t}=0$.

The small $I$-function lies in a certain completion
$A^*(X)_\QQ[[q]]\{\!\{1/z, z\}\!\}$ of Laurent series in $1/z$. (Here $z$ may be viewed as a formal variable of degree one, though it is more natural to interpret $z$ as the generator of the $\CC^*$-equivariant cohomology 
$A^*_{\CC^*}(\Spec(\CC))$.) It can be explicitly calculated for many targets.
For abelian quotients, that is, for toric varieties and for complete intersections in them, the small $I$-function is precisely the cohomology-valued hypergeometric 
series introduced by
Givental \cite{Giv} (up to exponential factors). Closed formulas for $I_{sm}$ in many examples with nonabelian 
$\G$ (e.g., complete intersections in flag varieties, but many others as well) can also be written down using the so-called {\it abelian/nonabelian correspondence}, see \cite{BCK1, BCK2, CKS, CKP}.

Consider the expansion 
\begin{equation*}
I_{sm}(q,z)= O(1/z^2)+ \frac{I_1(q)}{z}+I_0(q)+I_{-1}(q)z+I_{-2}(q)z^2+\dots
\end{equation*}
and set
\begin{equation*}
[zI_{sm}(q,z)-z]_+:=I_1(q)+(I_0(q)-1)z+I_{-1}(q)z^2+\dots 
\end{equation*}
In general $[zI_{sm}(q,z)-z]_+$ is a power series in 
$(q,z)$, but 
each $q$-coefficient is a polynomial in $z$. For each 
$0\neq\beta\in {\mathrm {Eff}}(W,\G,\theta)$, let
\[ \mu _{\beta}(z) \in A^* (X )_{\QQ}[z] \]
denote the coefficient of $q^\beta$ in $[zI_{sm}(q,z)-z]_+$. By easy dimension counting, 
$\mu_\beta(z)$ is homogeneous of degree
$1+ \beta (K_{[W/\G]})$.  Here $z$ has degree one, the Chow cohomology classes are given their usual degrees, and $K_{[W/\G]}=-\det(T_W)\in\Pic^{\G}(W) =\Pic([W/\G])$ is the canonical line bundle of the quotient stack.

We are now ready to state the wall-crossing for virtual classes.

\begin{Conj}\label{MainConj}
There is an equality
\begin{equation}\label{geometric}
\begin{split}& i_*
[Q^{\ke_{-}} _{g, k} (X, \beta ) ] ^{\vir}-c_*i_*[Q^{\ke_{+}}_{g, k}(X, \beta ) ]^{\vir}  =\\ 
&\sum _{|A|}  \sum _{\beta =\beta _0 + \sum _{a\in A}\beta _a} \frac{1}{|A|!} b_{A *} (c_A)_* i_*
 \left(\prod _{a\in A} ev_a^*\mu _{ \beta_a}(z)|_{z=-\psi_a}  \cap [Q^{\ke_{+}} _{g, k+A} (X, \beta _0) ] ^{\vir}\right)    \end{split}
 \end{equation}
in the Chow group $A_*(Q^{\ke _{-}}_{g, k} ( \PP ^m, d(\beta) ))_{\QQ}$.

More generally, let $\delta_1,\dots , \delta_k\in A^*(X)_\QQ$ be arbitrary homogeneous cohomology classes. Then 
\begin{equation}\label{geometric2}
\begin{split}& i_*\left(
\prod_{j=1}^k ev_j^*\delta_j\cap [Q^{\ke_{-}} _{g, k} (X, \beta ) ] ^{\vir}\right)-c_*i_*\left(\prod_{j=1}^k ev_j^*\delta_j\cap[Q^{\ke_{+}}_{g, k}(X, \beta ) ]^{\vir} \right) =\\ 
&\sum _{|A|}  \sum _{\beta =\beta _0 + \sum _{a\in A}\beta _a} \frac{1}{|A|!} b_{A *} (c_A)_* i_*
 \left(\prod_{j=1}^k ev_j^*\delta_j\prod _{a\in A} ev_a^*\mu _{ \beta_a}(z)|_{z=-\psi_a}  \cap [Q^{\ke_{+}} _{g, k+A} (X, \beta _0) ] ^{\vir}\right)    \end{split}
 \end{equation}
in $A_*(Q^{\ke _{-}}_{g, k} ( \PP ^m, d(\beta) ))_{\QQ}$.

\end{Conj}
\noindent In the above statement, 
$c_A: Q^{\ke_+}_{g, k+A} (\PP ^m, d(\beta_0 ) ) \ra Q^{\ke _{-}}_{g, k+A } (\PP ^m , d(\beta_0) )$
is the contraction of rational tails of degree $d(\beta_a)=1/\ke_0$.

\begin{Rmk} For $X$ a {\it semi-positive} quasi-projective toric manifold, Conjecture \ref{MainConj} coincides with Theorem 4.2.1 in \cite{CKg}, and the result is valid for any GIT presentation of $X$, see \cite[\S 5.9.2]{CKg}. In fact, the localization argument of \cite{CKg} extends with little effort to prove \eqref{geometric2} for {\it all} toric manifolds (i.e., no positivity restriction), offering the first evidence for
the validity of Conjecture \ref{MainConj}.
We will treat this extension elsewhere.
\end{Rmk}

\subsection{Numerical consequences} In this subsection we assume that $(W,\G,\theta)$ is a triple for which Conjecture \ref{MainConj} holds. We
work with arbitrary stability parameters
$\ke\in\QQ_{>0}\cup\{0+\}$ and will write $\ke=\infty$ for all parameters in the Gromov-Witten chamber
$(1,\infty)$.

Consider a formal power series in one variable $\psi$,
$${\bf t}(\psi):=t_0+t_1\psi+t_2\psi^2+t_3\psi^3+\dots,$$ 
with coefficients $t_j\in A^*(X)_\QQ$ general Chow cohomology classes.

The {\it genus $g$, $\ke$-descendent potential} of $X$ is 
\begin{equation*} F^\ke_g({q,\bf t}(\psi)):=\sum_{(\beta ,k)} \frac{q^\beta}{k!}\lan {\bf t}(\psi_1),{\bf t}(\psi_2),\dots {\bf t}(\psi_k) \ran^\ke_{g,k,\beta},\end{equation*}
the sum over all pairs $(\beta, k)$ for which the corresponding moduli spaces exist. 
If we choose a basis $\{\gamma_j\}$ in $A^*(X)_\QQ$ and write $t_i=\sum_j t_{ij}\gamma_j$, $i=0,1,2,
\dots $, then $F^\ke_g({q,\bf t}(\psi))$ is a formal power series in the infinitely many variables $t_{ij}$, whose Taylor coefficients are the
$\ke$-quasimap invariants \eqref{invariants}.
In particular, $F_g^\infty$ is the generating series for all descendent genus $g$ Gromov-Witten invariants of $X$.

\subsubsection{Wall-crossing from Gromov-Witten invariants to $\ke$-quasimap invariants}
Let $J^\ke_{sm}(q,z)$ be the small $J$-function of $X$ (\cite[Definition 5.1.1]{CKg0}, specialized at 
${\bf t}=0$). With this notation, $I_{sm}=J^{0+}_{sm}$. Let 
$$[zJ^\ke_{sm}-z]_+:=J_1^\ke(q)+(J^\ke_0(q)-1)z+J^\ke_{-1}(q)z^2+...$$
This is explicit for all $\ke$, since it is a $q$-truncation of the corresponding expression for the small $I$-function: 
\begin{equation*}
[zJ^\ke_{sm}(q,z)-z]_+=[zI_{sm}(q,z)-z]_+\; (\mathrm{mod}\;\mathfrak{a}_\ke),
\end{equation*}
with $\mathfrak{a}_\ke$ the ideal in the Novikov ring generated by $\{q^\beta\; |\: \beta(L_\theta) >\frac{1}{\ke}\}$.

\begin{Cor}\label{numerical1} For any $\ke\geq 0+$, and any $g\geq 1$,
$$F_g^\ke(q,{\bf t}(\psi))=F^\infty_g\left(q,{\bf t}(\psi)+[zJ^\ke_{sm}(q)-z]_+\big|_{z=-\psi}\right).$$
Further, in genus $g=0$ the same relation holds after discarding from $F^\infty_0(q,{\bf t}(\psi))$ the terms corresponding to pairs 
$(\beta,k)$ for
which $Q^\ke_{0,k}(X,\beta)$ is not defined.
\end{Cor}
\begin{proof} The $\psi$-classes at the markings $1,\dots ,k$ pull-back under the maps $b_A$, $c$, $c_A$, and $i$. Applying the virtual class wall-crossing \eqref{geometric2} in Conjecture \ref{MainConj} successively for the walls
from $1$ to $\ke$ (and using the projection formula) gives the equality of the Taylor coefficients 
of the two sides in the claimed equality.
\end{proof}
\begin{Rmk}
$(i)$ 
The formula in Corollary \ref{numerical1} is equivalent to
$$F_g^\ke\left(q,{\bf t}(\psi)-[zJ^\ke_{sm}(q)-z]_+\big|_{z=-\psi}\right)=F^\infty_g(q,{\bf t}(\psi)).
$$
$(ii)$ Assuming only the formula \eqref{geometric} from Conjecture \ref{MainConj} gives the weaker equality
$$F_g^\ke(q,\bar{\bf t}(\psi))=F^\infty_g\left(q,\bar{\bf t}(\psi)+[zJ^\ke_{sm}(q)-z]_+\big|_{z=-\psi}\right),$$
with $\bar{\bf t}(\psi)$ the restriction of ${\bf t}(\psi)$ to the subring $i^*A^*(\PP^m)_\QQ\subset A^*(X)_\QQ$.
\end{Rmk}
\subsubsection{Semi-positive targets}
Recall that a triple $(W,\G,\theta)$ is called semi-positive if 
$$\beta(\det T_W)=\beta(-K_{[W/\G]})\geq 0$$ 
for every $\beta\in{\mathrm {Eff}}(W,\G,\theta)$. For such targets we have 
$$[zJ^\ke_{sm}(q)-z]_+= J_1^\ke(q)+(J_0^\ke(q)-1)z,$$
since $\deg(\mu_\beta(z))\leq 1$ for all $\beta$.
The wall-crossing formula of Corollary \ref{numerical1} becomes
\begin{equation}\label{semipositive numerical}
F_g^\ke(q,{\bf t}(\psi))=F^\infty_g\left(q,{\bf t}(\psi)+J_1^\ke(q)-(J_0^\ke(q)-1)\psi\right).
\end{equation}
In fact, equation \eqref{semipositive numerical} is equivalent to the wall-crossing formula conjectured in  \cite[Conjecture 1.2.1]{CKg}:

\begin{Cor}\label{numerical2}
 For a semi-positive triple $(W,\G,\theta)$ we have
\begin{equation}\label{GW to epsilon}
 (J_0^\ke)^{2g-2}\left(\delta_g^1\left(\frac{\chi_{\mathrm{top}}(X)}{24}\log J_0^\ke(q)\right)
+ F_g^\ke(q,{\bf t}(\psi))\right)= F^\infty_g\left(q,\frac{{\bf t}(\psi)+J_1^\ke(q)}{J^\ke_0(q)}\right),
\end{equation}
where $\chi_{\mathrm{top}}(X)$ denotes the topological Euler characteristic and $\delta^1_g$ is the Kronecker delta.
(In genus $g=0$ we use the same convention as in Corollary \ref{numerical1}.)
\end{Cor}
\begin{proof} Using the dilaton equation for Gromov-Witten invariants in the right-hand side of 
\eqref{semipositive numerical} to remove the insertions $(J_0^\ke(q)-1)\psi$ produces exactly \eqref{GW to epsilon}. The additional term
$\delta_g^1\left(\frac{\chi_{\mathrm{top}}(X)}{24}\log J_0^\ke(q)\right)$ appears due to the failure
of the dilaton equation for $\overline{M}_{1,1}(X,0)=\overline{M}_{1,1}\times X$. Namely,  since the virtual class is
 $$[\overline{M}_{1,1}(X,0)]^{\mathrm{vir}}=\left(1\otimes c_{\dim X}(T_X)-\psi \otimes c_{\dim X-1}(T_X)\right)\cap [\overline{M}_{1,1}\times X],$$
 we have 
 \begin{equation*}
 \lan \psi\ran_{1,1,0}^\infty=\int_{\overline{M}_{1,1}\times X}\psi \otimes c_{\dim X}(T_X)=\frac{1}{24}\chi_{\mathrm{top}}(X),
 \end{equation*}
while the dilaton equation would formally predict $\lan \psi\ran_{1,1,0}^\infty=0$.
\end{proof}

\subsection{Complete intersections in projective space}
The main result of the paper is a proof of Conjecture \ref{MainConj} for projective complete intersections. In fact, we will prove the following slightly strengthened version.

Let $V$ be the affine space of dimension $n+1$ with the standard diagonal $\G:=\CC ^*$-action
and linearization $\theta={\mathrm{id}}_{\CC^*}$.
Let $W$ be a complete intersection of $r\leq n$ homogeneous hypersurfaces in $V$. Then 
$X:=\WmodtG$. is the corresponding projective complete intersection in $\PP(V)$ (and $W$ is the affine cone over $X$). Assume that the hypersurfaces are general, so that $X$ is smooth. We take
$X\hookrightarrow \PP(V)$ as our embedding $i$. In this case, the induced
$$i: Q^\ke_{g,k}(X,d)\lra  Q^\ke_{g,k}(\PP(V),d)$$
are also embeddings. The maps that replace markings by base-points, as well as the contraction maps,
respect these embeddings,  i.e., given a wall $\ke=1/{d_a}$ and $\ke_+>\ke\geq\ke_-$ nearby, we have restrictions
$$b_A: Q^{\ke_+}_{g,k+A}(X,d_0^A)\lra Q^{\ke_+}_{g,k}(X,d),$$
where $d_0^A=d-|A|d_a$, and
$$c: Q^{\ke_+}_{g, k}(X,d)\lra Q^{\ke_-}_{g, k}(X, d).$$

\begin{Thm}\label{Main}
There is an equality
\begin{equation*} 
\begin{split}& 
[Q^{\ke_{-}} _{g, k} (X, d ) ] ^{\vir}-c_*[Q^{\ke_{+}}_{g, k}(X, d ) ]^{\vir}  =\\ 
&\sum _{|A|}  \frac{1}{|A|!} (b_{A})_ * (c_A)_* 
 \left(\prod _{a\in A} ev_a^*\mu _{ d_a}(z)|_{z=-\psi_a}  \cap [Q^{\ke_{+}} _{g, k+A} (X, d _0^A) ] ^{\vir}\right)    \end{split}
 \end{equation*}
in the Chow group $A_*(Q^{\ke _{-}}_{g, k} ( X, d ))_{\QQ}$.
 \end{Thm}

Since Theorem \ref{Main} implies the formula \eqref{geometric2}, the relations between $\ke$-quasimap invariants and Gromov-Witten invariants in Corollaries \ref{numerical1} and \ref{numerical2} hold  
for nonsingular 
complete intersections $X\subset\PP^n$ of codimension $r\leq n$.

Let $l_1,l_2,\dots ,l_r$ be the degrees of the hypersurfaces whose intersection is $X$.
The small $I$-function of $X$ is given by the well-known formula (see \cite{Givental-equiv})
\begin{equation*}
I(q,z)=\one+\sum_{d\geq 1}q^d\frac{\prod_{i=1}^r\prod_{j=1}^{l_id}(l_iH+jz)}{\prod_{j=1}^d (H+jz)^{n+1}},
\end{equation*}
where $H$ denotes the restriction to $X$ of the hyperplane class on $\PP^n$. 

If $\sum_{i=1}^r l_i\geq n+2$, so that $X$ is a variety of general type, we do not know of any simplification of the wall-crossing formula in Corollary \ref{numerical1}. Note that even in genus $g=0$ our result is new.  

If $X$ is Fano or Calabi-Yau, more precise statements can be made.

The case $\sum_{i=1}^r l_i\leq n-1$ of complete intersections which are Fano of index at least two is the simplest, since $J_0^\ke(q)=1$ and
$J_1^\ke(q)=0$ for all $\ke\geq 0+$. We conclude the following $\ke$-independence result.
\begin{Cor} The quasimap invariants of a projective complete intersection with $\sum_il_i\leq n-1$ are independent of $\ke$.
\end{Cor}

In the Fano of index one case, $\sum_{i=1}^r l_i= n$, we have $J_0^\ke(q)=1$ and
$J_1^\ke(q)=q (\prod_{i=1}^rl_i!)\one$ for all $0+\leq \ke\leq 1$.
\begin{Cor}\label{index one}
For a projective complete intersection with $\sum_il_i=n$ and for $0+\leq\ke\leq 1$ we have
$$F_g^\ke({\bf t}(\psi))=F_g^\infty({\bf t}(\psi)+ q( \prod_{i=1}^rl_i!)\one).$$

In particular, if
$(g,n)\neq(0,1),(0,2)$, then the {\em primary} invariants are again $\ke$-independent:
\begin{equation*}
\lan\gamma_1,\dots\gamma_n\ran_{g,n,\beta}^\ke=\lan\gamma_1,\dots\gamma_n\ran_{g,n,\beta}^\infty.
\end{equation*}
\end{Cor}
The second equality in Corollary \ref{index one} is a consequence of the string equation in Gromov-Witten theory.

The most interesting is the Calabi-Yau case $\sum_{i=1}^r l_i= n+1$, for which
\begin{align*} 
& J_0^\ke(q)=\sum_{0\leq d\leq \frac{1}{\ke}}q^d\frac{\prod_{i=1}^r(l_id)!}{d!^{n+1}},\\ 
  &J_1^\ke(q)= H\sum _{1\leq d\leq \frac{1}{\ke}}
 q^d \frac{\prod _{i=1}^r (l_id) !}{d!^{n+1}} \left(\sum _{i=1}^r \sum_{k=1}^{l_id} \frac{l_i}{k}  - (n+1)\sum_{k=1}^{d} \frac{1}{k}\right).
\end{align*}

For every $\ke$ and every $d$, the virtual dimension of the moduli space $Q^{\ke}_{g, k } (X , d  )$ is equal to $(\dim X-3)(1-g)+k$. We split the discussion according to the genus.

\subsubsection{Genus zero}
The wall-crossing formula \eqref{GW to epsilon} at $g=0$ for a Calabi-Yau complete intersection is proved in \cite[\S3]{CKg} using Dubrovin-type reconstruction arguments and results from \cite{CKg0}. Here we just note that the new proof in this paper {\it does not} use the torus action on $\PP^n$.

\subsubsection{Genus one} 
When $g=1$, the virtual dimension is independent of the dimension of $X$. Consider the unpointed case $k=0$, i.e. the specialization of \eqref{GW to epsilon} at $g=1$, and ${\bf t}(\psi)=0$. Separating the $d=0$ contributions and applying the divisor equation in the Gromov-Witten side gives
\begin{Cor}\label{genus 1 BCOV} For a Calabi-Yau complete intersection $X\subset\PP^n$
\begin{align}\label{genus 1 wall-crossing}
&\frac{1}{24}\chi_{\mathrm{top}}(X) \log J_0^\ke+\sum_{d\geq 1} q^d\lan\;\ran_{1,0,d}^{\ke}=\\
\nonumber &-\frac{1}{24}\int_X \frac{J_1^\ke}{J_0^\ke}c_{\dim X-1}(T_X)+\sum_{d\geq 1} q^d\exp\left(\int_{d[line]}\frac{J_1^\ke}{J_0^\ke}\right) \lan\;\ran_{1,0,d}^\infty.
\end{align}
\end{Cor}
When $\ke=0+$, the
formula \eqref{genus 1 wall-crossing} answers a question raised first in \cite[\S10.2]{MOP}. 
Note that the unpointed genus one $(0+)$-invariants $\lan\ \ran_{1,0,d}^{0+}$ have been recently calculated by Kim and Lho  (\cite{KL}) in terms of the small $I$-function. Combining \cite[Theorem 1.1]{KL} with
Corollary \ref{genus 1 BCOV} gives new proofs for the main results on genus one Gromov-Witten invariants of $X$ from 
\cite{Zinger} and \cite{Popa}. 

\subsubsection{Higher genus} 
If $g\geq 2$ and $\dim X\geq 4$, the virtual classes (hence the invariants) vanish by dimension considerations. We restrict to the case of unpointed invariants of Calabi-Yau threefolds. The invariants for $d=0$ are the same for all stability conditions and are given by the formula
$$\lan\;\ran_{g,0,0}^{\ke}=
\frac{(-1)^g}{2}\chi_{\mathrm{top}}(X)\frac{|B_{2g}|}{2g}\frac{|B_{2g-2}|}{2g-2}\frac{1}{(2g-2)!},
$$
with $B_{2g}, B_{2g-2}$ the Bernoulli numbers, see \cite{getzler-p}, \cite{faber-p}.
\begin{Cor}\label{qmap BCOV} 
For a Calabi-Yau threefold complete intersection in $\PP^n$, $g\geq 2$ and $\ke\geq 0+$,
\begin{align*} 
&J_0^\ke(q)^{2g-2}\left(\frac{(-1)^g}{2}\chi_{\mathrm{top}}(X)\frac{|B_{2g}|}{2g}\frac{|B_{2g-2}|}{2g-2}\frac{1}{(2g-2)!}+ 
\sum_{d\geq 1}q^d\lan\;\ran_{g,0,d}^{\ke}\right)= \\
\nonumber & \frac{(-1)^g}{2}\chi_{\mathrm{top}}(X)\frac{|B_{2g}|}{2g}\frac{|B_{2g-2}|}{2g-2}\frac{1}{(2g-2)!}
+\sum_{d\geq 1} q^d\exp\left(\int_{d[line]}\frac{J_1^\ke}{J_0^\ke}\right) \lan\;\ran_{g,0,d}^\infty.\end{align*}
\end{Cor}

\subsection{Relation with Mirror Symmetry} \label{mirror symm}
In this subsection we let $X$ be the quintic hypersurface in $\PP^4$ and consider the asymptotic stability condition $\ke=0+$. (The same analysis will apply to the $(0+)$-theory of any Calabi-Yau threefold for which Conjecture \ref{MainConj} holds.) 

Fix a genus $g\geq 1$. In their landmark paper \cite{BCOV}, Bershadsky, Cecotti, Ooguri, and Vafa studied the string theory $B$-model of a Calabi-Yau threefold and in particular they proposed a method to calculate the
genus $g$ Gromov-Witten potential of the quintic (with no insertions)
via Mirror Symmetry. Namely, let $\cF^B_g(q)$ be the holomorphic limit of the genus $g$ partition function of the $B$-model associated to the {\it mirror family} of the quintic,
where $q$ is the coordinate around the large complex structure point. Let the mirror map be $Q=q\exp(\frac{1}{H}\frac{I_1(q)}{I_0(q)})$, where

\begin{equation*}
I_0(q)=1+\sum_{d\geq 1}q^d\frac{(5d)!}{d!^5},\;\;\;\; I_1(q)=H \sum_{d\geq 1}q^d \frac{(5d)!}{(d!)^5}\left(\sum_{j=d+1}^{5d}\frac{1}{j}\right).
\end{equation*}
Then the genus $g\geq 2$ Mirror Conjecture of \cite{BCOV} for the quintic threefold is the equality
\begin{equation}\label{BCOV}
I_0(q)^{2g-2}\cF^B_g(q)=\sum_{d\geq 0}Q^d\lan\;\ran_{g,0,d}^\infty.
\end{equation}
Hence Corollary \ref{qmap BCOV} says precisely that the quasimap partition function $F_g^{0+}|_{{\bf t}=0}(q)$ is {\it equal} to $\cF^B_g(q)$,
with no mirror map involved. Similarly, Corollary \ref{genus 1 BCOV} gives the same equality in genus $g=1$. In other words, our results in this paper can be viewed as giving a mathematically rigorous and geometrically meaningful definition of the holomorphic limit of the $B$-model partition function.

The $B$-model partition function of the mirror quintic has been studied extensively in the Physics literature. It is expected to have 
modular properties and to satisfy a recursion in $g$, determined up to a holomorphic function $f_g(q)$, the so-called ``holomorphic ambiguity".
The ambiguity has been fixed up to genus $g=51$ in \cite{Klemm et al} and this is by far the most efficient computational method 
for predicting (via the conjectural mirror formula \eqref{BCOV}) the higher genus Gromov-Witten invariants of the quintic.
We speculate that the holomorphic ambiguity $f_g(q)$ has an intrinsic meaning in quasimap theory. It would be very interesting to determine if this is indeed the case.

\subsection{Final remarks} While the proof of Theorem \ref{Main} we give here is quite involved, it turns out to be also robust. For example, it extends easily to the case of complete intersections in products of projective spaces. It also applies to proving a wall-crossing formula for the virtual classes of quasimap moduli spaces (with same stability parameter $\ke=0+$ and target a complete intersection 
$X\subset \prod\PP^{n_i}$) when one usual marking is changed to an infinitesimally weighted marking. To keep this paper from becoming excessively long, we defer the details of these developments to future writings.

\subsection{Acknowledgments}

I.C.-F. was partially supported by the NSF grants DMS-1305004 and DMS-1601771. B.K. is supported by 
the KIAS individual grant MG016403.
In addition, I.C-F. thanks KIAS for financial support, excellent
working conditions, and an inspiring research environment during visits when a large part of this project was completed.
We deeply thanks the anonymous referee for valuable suggestions to improve the readability of the paper.

\section{Virtual classes for moduli of quasimaps}\label{virtual classes}

\subsection{Overview}  In this section we give a concrete description of the virtual class of a moduli space of quasimaps to a complete intersection in projective space. This is accomplished by embedding the moduli space into a smooth stack and intersecting the normal cone for this embedding with the zero section of an appropriate vector bundle. This description will be crucially used in the proof of Theorem \ref{Main} given in section
 \ref{proof of Main Thm}.
The construction is uniform for all discrete parameters $g,k,d$ and $\ke$, but requires the existence of the moduli space of stable curves, so it doesn't apply directly to the unpointed elliptic case $(g,k)=(1,0)$. An appropriate modification, sufficient for completing the proof of Theorem \ref{Main} in this case as well, will be discussed in \S\ref{elliptic}.

\subsection{Set-up and conventions}
From now on we let $\G = \CC^*$. 
Let $V$ be an $n+1$-dimensional $\G$-representation ($n\geq 1$), with weight vector $(1, ..., 1)$. Let
$\CC^r_{\vec{l}}$ be an $r$-dimensional $\G$-representation with positive weight vector $\vec{l}:=(l_1, ..., l_r)$ ($l_j>0, \forall j$). Assume we are given 
a $\G$-equivariant map 
\[ \varphi = \oplus_{i=1}^r \varphi _i : V \ra \CC ^r_{\vec{l}} \]
such that  the  closed subscheme $W:=\varphi ^{-1}(0)$ is smooth away from $0\in V$
and of dimension
$\dim W = n+1-r>0$. We linearize the $\G$ action on $V$ by the character $\theta$ of weight $1$.
The GIT quotient $X:=\WmodtG$ is a nonsingular complete intersection of type $(l_1, ..., l_r)$ in
$\PP^n=V/\!\!/_{\theta}\G$, with $\varphi_i$ its homogeneous equations.

Recall that the inclusion $i:X\subset\PP(V)$ induces an embedding
$$i: Q^\ke_{g, k}(X, d)\hookrightarrow Q^{\ke} _{g, k}(\PP (V),d)$$
for all $\ke\geq 0+$.

We also make the following conventions:

\begin{itemize}

\item $\overline{M}_{g, k}$ denotes the Deligne-Mumford stack of $k$-pointed stable curves of genus $g$, while
$\fM _{g, k}$ denotes the Artin stack of prestable $k$-pointed curves of genus $g$.

\item $\Bun$  denotes the moduli stack of principal $\G$-bundles on $k$-pointed
prestable curves of genus $g$. It is a smooth Artin stack of pure dimension and decomposes as $\coprod_{d\in \ZZ}\Bund$, according to the degrees of the principal bundles.
There are natural forgetful morphisms $$Q^{\ke} _{g, k}(\PP (V),d)\lra \Bund\lra\fM _{g, k}.$$

\item The universal families of curves on various moduli stacks are denoted by $\fC$, usually with decorations recording the discrete data. For example,
\[ \xymatrix{ \fC ^{\ke}_{g, k, d} \ar[r]\ar[d]  & \fC _{g, k}  \ar[d] & \ar[l] \fC ^{\ke'}_{g, k, d'}\ar[d] \\
                     Q^\ke_{g, k}(X, d) \ar[r] & \overline{M}_{g, k}  & \ar[l] Q^{\ke'} _{g, k}(\PP (V\ot \CC ^N), d'). } \]
                    We will abuse notation and denote always by $\pi$ the projection from the universal curve to the base.

\end{itemize}

We will represent quasimaps to a projective space $\PP(V)$ as tuples
$$((C,p_1,\dots, p_k),L,u)$$
with $L$ a line bundle on $C$ and $u$ a section of $L\ot V$ (as in \cite{CK}). 
Quasimaps to $X\subset\PP(V)$ will then be such tuples for which the components 
$u_1,\dots, u_{\dim V}$ of $u$ (once a basis of $V$ is chosen) satisfy the homogeneous equations of $X$. The base-points of the quasimap are the points of $C$ where all the $u_i$'s vanish and the length $\ell(x)$  at a point $x\in C$ is the common order of vanishing. Given $ \ke\in \QQ_{>0}$, recall that the definition of $\ke$-stability requires the following conditions be satisfied:
\begin{enumerate}
\item[(1)] the base-points are away from nodes and markings;
\item[(2)]  $\ke\ell(x)\leq 1$ for all $x\in C$;
\item[(3)] the line bundle $\omega_C(p_1+\dots +p_k)\ot L^\ke$ is ample.
\end{enumerate}
For $\ke=0+$ condition $(2)$ is empty and is discarded, while condition $(3)$ translates into the absence of rational tails in $C$ and the strict positivity of $\deg L$ on rational bridges (rational components of $C$ containing exactly two special points).

Finally, recall that the theory of virtual classes was first developed by Li and Tian in \cite{LT}, and by Behrend and Fantechi in \cite{BF}. In this paper we use the formalism of \cite{BF}.

\subsection{Twisting line bundles}\label{twisting lb}

Fix $(g,k)\neq(1,0)$. 

For each $\ke\geq 0+$ we construct a line bundle $\sM_\ke$ on the universal curve
$$\fC ^\ke_{g, k, d}\lra Q^{\ke} _{g, k}(\PP (V),d)$$ as follows.

When $g=0$, we take the trivial line
bundle $\sM_{\ke}=\cO$.

When $g\geq 1$ and $g+k\geq 2$, the moduli stack $\overline{M}_{g, k}$ exists and we have the
diagram 
\[ \xymatrix{ \fC ^{\ke}_{g, k, d} \ar[r]^{\widetilde{ft_\ke}}\ar[d]_{\pi}  & \fC _{g, k}  \ar[d]^{\pi}\\
                    Q^\ke_{g, k}(X, d) \ar[r]^{ft_\ke} & \overline{M}_{g, k} \ar@/^1pc/[u]^{\Sigma_i}    } \]
with $ft_{\ke}, \widetilde{ft_\ke}$ the stabilisation morphisms and
$\Sigma_i$ the sections of $\pi$ corresponding to the $k$ markings. The logarithmic relative dualising sheaf 
$\omega_{log}:=\omega_{\pi}(\Sigma_1+\dots\Sigma_k)$ on $\fC _{g, k}$
is $\pi$-ample and we choose a positive integer $p$ such that $\omega_{log}^{ \ot p}$ is $\pi$-relatively very ample. We also choose a very ample line bundle on the (projective!) coarse moduli of 
$\overline{M}_{g, k}$ and denote by $\mathscr{H}$ its pull-back to the stack $\overline{M}_{g, k}$. Now set
$$\sM_\ke:=  \widetilde{ft_\ke}^*(\pi^*\mathscr{H}\ot \omega_{log}^{ \ot p}).$$

\begin{Lemma}\label{easy lemma} The line bundles $\sM_\ke$ satisfy the following properties:
\begin{enumerate}
\item[(i)] If $\ke>\ke'$, then $\sM_{\ke}=\tilde{c}^*\sM_{\ke'}$, where $\tilde{c}$ is the induced contraction morphism on universal curves in the diagram
\[ \xymatrix{ \fC ^{\ke}_{g, k, d} \ar[d]    \ar[r]^{\tilde{c}} &\fC ^{\ke '}_{g, k, d}\ar[d] \\
                     Q^\ke_{g, k}(\PP(V), d)  \ar[r]^c  & Q^{\ke'} _{g, k}(\PP (V), d) } \]

\item[(ii)] For every geometric fiber $C$ of $\fC ^\ke_{g, k, d}\ra Q^{\ke} _{g, k}(\PP (V),d)$ we have
$$H^1(C,\sL\ot\sM_\ke|_C)=0,$$
where $\sL$ denotes the universal line bundle associated to the universal principal $\G$-bundle
on the universal curve.

\end{enumerate}
\end{Lemma}

\begin{proof} Part $(i)$ is obvious from the definition, since $C$ and $\tilde{c}$ are compatible with the forgetful stabilisation maps. 
For part $(ii)$, notice that $\deg \sL$ is nonnegative on every component of every geometric fiber $C$ and by stability it is strictly positive on every rational component with at most two special points. On the other hand, by construction $\sM_\ke$ has vanishing $H^1$ on the stabilization of  $C$ and is trivial on rational tails and rational bridges.
The required vanishing follows.
\end{proof}

Choose once and for all global sections $\{\tau_1,\dots ,\tau_N\}$  giving a basis of $\Gamma(\fC _{g, k},\pi^*\mathscr{H}\ot \omega_{log}^{ \ot p})$, and hence an embedding
$$h:\fC _{g, k}\lra\PP(\CC^N).$$
Let $s_j^\ke:=\widetilde{ft_\ke}^*\tau_j$ of $\sM_\ke$ be the induced sections of $\sM_\ke$, determining the map $h_\ke :=h\circ\widetilde{ft_\ke}$, with $\sM_\ke=h_\ke^*\cO_{\PP(\CC^N)}(1)$.
When the parameter $\ke$ is understood we will drop it from the notation and write simply $\sM$ and $s_j$ for the twisting line bundle and its sections. Furthermore,
we will use the same notations when considering the restriction of the set-up in this subsection to the moduli spaces 
$Q^\ke_{g,k}(X,d)$ via the embedding $i$.

Note that the degree of $\sM$ on the fibers of the universal curve is a constant positive integer $d_{\sM}$ depending only on $(g,k)$, but not on
$d$, or on the dimension of $\PP(V)$.

\subsection{Perfect obstruction theory of \texorpdfstring{$Q^{\varepsilon}_{g,k}(X,d)$}{Lg}}\label{POT}

Fix $(g,k)\neq(1,0)$ and $\varepsilon \geq 0+$. 
Consider the line bundle $\sL ' := \sL \ot \sM$ on the universal curve $\fC ^{\varepsilon}_{g, k, d}$ over 
$Q^{\varepsilon}_{g,k}(X,d)$.
There is a commuting diagram with exact rows

\begin{equation}\label{comp} \xymatrix{ 0 \ar[r]   & \mathscr{L}\ot V \ar[r]^-{\oplus _j s_j}  \ar[d]_{\oplus_id\varphi_i} & 
 \oplus _{j=1}^N \mathscr{L}' \ot V  \ar[d]^{\oplus _{i, j} s_j ^{l_i-1} d\varphi _i} 
 \ar[r]^-{\alpha_0}     & \mathscr{P}  \ar[r]          \ar[d]^{f}           & 0 \\
                      0 \ar[r]  & \oplus _{i=1}^r \sL  ^{l_i}  \ar[r]^-{\oplus _{i, j} s_j^{l_i} }     
&  \oplus _{i, j}  ( \sL')  ^{l_i}  \ar[r]_-{\alpha_1}                              & \mathscr{Q}  \ar[r]              & 0 .} \end{equation}
The top row is obtained by puling-back the tautological sequence 
 \begin{equation}\label{tautological}
 0\lra \cO_{\PP(\CC^N)}(-1)\lra \cO_{\PP(\CC^N)}\ot \CC^N\lra Q\lra 0
 \end{equation} 
 via $h_\ke:\fC ^{\ke}_{g, k, d}\lra \PP(\CC^N)$ 
and tensoring with $\mathscr{L}'\ot V$. 
The bottom row comes from \eqref{tautological} similarly,
 by taking the direct sum of its pull-backs via $g_{l_i}\circ h_\ke$,  tensored with 
 $(\sL')^{l_i}$, where $g_{l_i}:\PP(\CC^N)\lra\PP(\CC^N)$ is the degree $l_i$ map
 $[t_1:\dots : t_N]\mapsto [t_1^{l_i}:\dots : t_N^{l_i}]$.
In particular, $\mathscr{P}$ and $\mathscr{Q}$ are vector bundles. 

The components $d\varphi_i$ of the vertical homomorphism on the left are given as follows. Let $\Delta\subset \fC ^{\ke}_{g, k, d}$ be an open substack. After choosing coordinates $(x_0,\dots x_n)$ on $V$, we may write $\varphi_i$ as a homogeneous polynomial of degree $l_i$  and a local section $v$ of $\mathscr{L}\ot V$ on $\Delta$ as $v=(v_0,\dots v_n)$. Then we put 
 $$d\varphi_i(v)=\nabla\varphi_i(u|_\Delta)\cdot v=\sum_{m=0}^n\frac{\partial \varphi_i}{\partial x_m}(u|_\Delta)v_m,$$
 where $u=(u_0,\dots ,u_n)$ is the universal section of $\sL\ot V$ on $\fC ^{\ke}_{g, k}$.
 Similarly, for fixed $i$ and $j$ and a local section $v'=(v_0',\dots v_n')$ of $\mathscr{L}'\ot V$,
 $$s_j ^{l_i-1} d\varphi _i(v')=\sum_{m=0}^n \frac{\partial \varphi_i}{\partial x_m}(u\ot s_j|_\Delta)v_m'
 =\sum_{m=0}^n s_j ^{l_i-1}|_\Delta\frac{\partial \varphi_i}{\partial x_m}(u|_\Delta)v_m'.$$

 Viewing \eqref{comp} as an exact sequence of two-term complexes, it follows that
 the two-term vertical complex on the left in \eqref{comp} is quasi-isomorphic 
to the shifted mapping cone $A^\bullet:=\mathrm{Cone}(\alpha)[-1]$ of the homomorphism $\alpha=(\alpha_0,\alpha_1)$.
Denote \[  \sR := \oplus _{i, j} (\sL ')  ^{l_i} .\]
 Define a coherent sheaf $\mathscr{E}$ (in fact, a vector bundle) by the exact sequence
\begin{align} \label{sW} 0 \ra\mathscr{E}   \ra \mathscr{P} \oplus \sR \ra\mathscr{Q}  \ra 0, \end{align}
where $\mathscr{P} \oplus \sR \ra\mathscr{Q}$ is given
by $(x,y)\mapsto f(x)-\alpha_1(y)$.
Then $A^\bullet$ is quasi-isomorphic to 
\begin{align}\label{preperfect}  \oplus _{j=1}^N \mathscr{L}'\ot V  \ra\mathscr{E}. \end{align}

On the other hand, if $$\rho : Prin(\sL) \times _\G W \ra  \fC^\ke_{g,k,d} $$ denotes the universal $W$-fiber bundle with $Prin (\sL)$ the principal $\G$-bundle associated to $\sL$ 
and we view $u$ as the universal section of $\rho$, 
then the pull-back $u^*\mathbb{T}_\rho$ of the relative tangent complex of $\rho$ coincides with the two-term complex $\sL\ot V\ra \oplus _{i=1}^r \sL  ^{l_i}$ on the left of \eqref{comp}. We conclude that $u^*\mathbb{T}_\rho$ is quasi-isomorphic to \eqref{preperfect}
at amplitude $[0,1]$.

Part $(ii)$ of Lemma \ref{easy lemma} gives the vanishing $R^1\pi _* \sL ' =0$. 
This in turn implies that $R^1\pi _* \mathscr{P}  = R^1 \pi _* \mathscr{Q}  =0$. 
 Since the derived push-forward of $u^*\mathbb{T}_\rho$ has   amplitude in $[0,1]$ by \cite[Theorem 4.5.2]{CKM}, the same is true for the derived push-forward of the shifted mapping cone $A^\bullet$. Hence the
map $\pi_* (\mathscr{P} \oplus \sR ) \to \pi_*\mathscr{Q} $ is surjective 
and then
$R^1\pi _*\mathscr{E} =0 $ from \eqref{sW}. It follows that 
\begin{equation}\label{E1 term}
{E}_d^\ke:= \pi _* \mathscr{E}
\end{equation} is a locally free sheaf on $\Qked$ and
we obtain
a perfect complex 
\begin{align}\label{perfect} \oplus _{j=1}^N  \pi _* \mathscr{L}' \ot V   \ra E_d^\ke, \end{align}
whose dual represents the canonical perfect obstruction theory 
$$(R^\bullet \pi _* u^* \mathbb{T}_\rho)^\vee$$ for $\Qked$
relative to $\fB un _{\G}^{g, k}$. We have proved the following result.

\begin{Prop}\label{virtual class} The virtual fundamental class of $\Qke$ is
\[       [\Qked ]^{\vir} = 0_{E_d^\ke}^! ([\mathbf{C}_\ke]) \] 
where $\mathbf{C}_\ke\subset E_d^\ke$ denotes the Behrend-Fantechi obstruction cone,  see \cite{BF}, associated to the relative perfect obstruction theory given by \eqref{perfect}.
\end{Prop}

\subsection{An embedding of \texorpdfstring{$Q^\ke_{g,k}(X,d)$}{Lg} into a smooth stack}\label{U-embedding}

Set $$d' :=d+d_{\sM}= d + \deg ( \sM|_C).$$ 
Consider the moduli stack $Q^{\ke}_{g, k} (\PP (V\ot \CC ^N), d')$, with universal curve
$\fC ^{\ke}_{g, k, d'} $.
By a slight abuse, denote also by $\sM $ the twisting line bundle on 
$\fC ^{\ke}_{g, k, d'}$ (defined by the construction in \S\ref{twisting lb}, as the pull-back of $\pi^*\mathscr{H} \ot\omega_{log}^{\ot p}$ on $\fC_{g,k}$ by the stabilization morphism).

\begin{Def} 
Define 
$ U_{d'}^\ke  \subset Q^\ke _{g, k}(\PP (V\ot \CC ^N), d') $
as the open substack consisting of the $\ke$-stable quasimaps 
$$((C, p_1, ..., p_k), L', u')$$ to $\PP (V\ot \CC ^N)$
such that  
$H^1(C, L')=0$.
\end{Def}

Note that $U_{d'}^{\ke}$ is the complement of the support 
of the coherent sheaf $R^1\pi _* \sL '$, so it is indeed an open substack.

\begin{Lemma}\label{Con}
The stack $U_{d'}^\ke$ is a separated DM-stack of finite type, smooth and of pure dimension over
 $\Bun$, and hence over $\fM _{g, k}$.
In particular, fixing a locally-closed substack of $\Bun$ parametrizing prestable curves with fixed topological type, together with line bundles of given degrees on the components, produces a corresponding locally-closed substack of $U_{d'}^\ke$ with the same codimension.
\end{Lemma}
\begin{proof}  The separatedness and finite type properties follow from the corresponding ones for
$Q^{\ke}_{g, k} (\PP (V\ot \CC ^N), d')$. By definition, the quasimaps in $U_{d'}^\ke$ are unobstructed, which gives the smoothness and the pure dimensionality. (In fact, $U_{d'}^\ke$ is also irreducible, since it is
the smooth locus in the ``main component" of $Q^\ke _{g, k}(\PP (V\ot \CC ^N), d')$. Irreducibility of 
the ``main component" follows from the connectedness of $\overline{M}_{g, k}(\PP (V\ot \CC^N), d')$, proven in \cite{KP}.)
\end{proof}

Let $\pi :\fC ^{\ke}_{g, k, d'} \ra U_{d'}^\ke$ be the universal curve and let $\sL'$  be the universal line bundle of $\pi$-relative degree $d'$ 
on $\fC ^{\ke}_{g, k, d'}$. By the very definition of $U^\ke_{d'}$, the sheaf $\pi_*\sL'$ is locally free.
Put $$\sL := \sL ' \ot \sM ^{-1},$$  
and consider the diagram of vector bundles on $\fC ^{\ke}_{g, k, d'}$
\begin{equation}\label{twisting} \xymatrix{ 0 \ar[r]   &\mathscr{L} \ot V \ar[r]^{\oplus _j s_j\;\;\;\;}   &  \oplus _{j=1}^N  \mathscr{L}' \ot V \ar[r]\ar[d]^{\oplus_j(\oplus_id\varphi_i)}        & \mathscr{P}^{\ke}_{d'}  \ar[r]                 & 0 \\
                      0 \ar[r]  & \oplus _{i=1}^r  \sL  ^{l_i}  \ar[r]^{\oplus _{i, j} s_j^{l_i}}       
& \oplus _{i, j}   (\sL ') ^{l_i} \ar[r]                              & \mathscr{Q}^{\ke}_{d'}  \ar[r]              & 0 . } \end{equation}
As before, the exact
rows are obtained from the tautological exact sequence \eqref{tautological} on $\PP(\CC^N)$ 
via pull-backs, tensoring with appropriate line bundles, and taking direct sums.
The components of the map between the middle terms (for fixed $i$ and $j$) are given by
 $$d\varphi _i(v'_{j0},\dots v'_{jn})=\sum_{m=0}^n \frac{\partial \varphi_i}{\partial x_m}((u'_{j0},\dots u'_{jn})|_\Delta)v_{jm}',$$
 where 
  \begin{equation}\label{eqn:def u'}
  u'=(u'_{10},\dots , u'_{1n},u'_{20},\dots , u'_{2n},\dots, u'_{N0},\dots, u'_{Nn})
  \end{equation} is the universal global section of $\oplus _{j=1}^N  \mathscr{L}' \ot V$ on $\fC ^{\ke}_{g, k, d'}$ and
 $(v'_{10},\dots , v'_{1n},\dots, v'_{N0},\dots, v'_{Nn})$ is a local section
 of $\oplus _{j=1}^N  \mathscr{L}' \ot V$ over an open $\Delta\subset\fC ^{\ke}_{g, k, d'}$.

Let us denote \[ \sA^{\ke}_{d'} := \oplus _{j=1}^N \sL ' \ot V, \ \ \sR^{\ke}_{i, d'} :=  \oplus _{j=1}^N  (\sL ' )^{l_i},   \ \ \sR^{\ke}_{d'} :=  \oplus _{i=1}^r \sR^{\ke}_{i, d'}  .\]
The tautological section  $\tau^\ke$ of $\pi_*\sA^{\ke}_{d'}$ induces a natural section 
$\sigma^\ke_P$ of the vector bundle
$$P^\ke_{d'}:= \pi_*\sP ^{\ke}_{d'}$$ on $U_{d'}^{\ke}$.
On the other hand, we also have the section $\sigma^\ke_R$ of the vector bundle 
$$R^\ke_{d'}:= \pi _* \sR^{\ke}_{d'} $$ whose $(i,j)$-component is given by 
$\varphi_i(u'_{j0},\dots,u'_{jn})$. Set
\begin{equation}\label{sigma epsilon}
\sigma^\ke:=(\sigma^\ke_P,\sigma^\ke_R)\in H^0(U_{d'}^{\ke}, P^\ke_{d'}\oplus R^\ke_{d'}).
\end{equation}
Because the exactness of the rows of \eqref{twisting} is preserved for any base change, it follows immediately that the zero locus of 
the section $\sigma^{\ke}$ is identified with the stack  $Q^\ke_{g, k}(X, d)$. 
Thus, we have an explicit embedding of $Q^\ke_{g, k}(X, d)$ in the smooth stack $U_{d'}^{\ke}$, summarized in the diagram

\begin{equation}\label{embedding diagram}
 \xymatrix{                                                          &  P^\ke_{d'}\oplus R^\ke_{d'} \ar[d]                \\
Q^\ke_{g, k}(X, d) \cong (\sigma^{\ke})^{-1}(0)  \ar@{^{(}->}[r]^-{\text{closed} } \ar[rd]  &  U_{d'}^{\ke} \ar[d]^{\text{smooth}} \ar@/_1pc/[u]_{\sigma^\ke}  \\
                                                                                  &           \Bun . } \end{equation}

Over $Q^\ke_{g, k}(X, d) $, the diagram \eqref{twisting} restricts to the diagram \eqref{comp}. 
Denoting by $\mathscr{I}$ the ideal sheaf of the closed substack $Q^\ke_{g, k}(X, d)$ in 
$U^{\ke}_{d'}$ and setting 
\begin{equation}\label{E hat}
F_{d'}^\ke:=  P^\ke_{d'}\oplus R^\ke_{d'} =\pi _*\sP ^{\ke}_{d'} \oplus  \pi _* \sR^{\ke}_{d'},
\end{equation}
we obtain the commuting diagram of coherent sheaves

\begin{equation}\label{embedded obstruction}
\xymatrix { \left( { F}_{d'}^\ke  |_{Q^\ke_{g, k}(X, d) }\right)^\vee \ar@{->>}[rd]_-{(\sigma^\ke)^\vee} \ar[r] & (E _d^\ke )^\vee \ar[d]  \ar[r] 
& \left(    \pi_* \sA^{\ke}_{d'} |_{Q^\ke_{g, k}(X, d) }\right)^\vee   \ar[d]_{=} \\
 &  \sI /\sI ^2\ar[r] &  \Omega _{U_{d'}^{\ke}/\Bun } |_{Q^\ke_{g, k}(X, d) }\;\;\;\;\; ,}
\end{equation}
where the existence of the surjection $(E _d^\ke )^\vee \twoheadrightarrow  \sI /\sI ^2$ follows from a standard deformation theory calculation.

The square in the diagram \eqref{embedded obstruction} is precisely the map of complexes 
from the obstruction theory
\eqref{perfect} to the two-term truncation of the relative cotangent complex 
$\mathbb{L}_{Q^\ke_{g, k}(X, d)/\Bun }$. 
The indicated equality $ (\pi_* \sA ^{\ke}_{d'} )^\vee=\Omega _{U_{d'}^{\ke}/\Bun }$
follows from the definition of $U_{d'}^{\ke}$ and the identification of 
$(R^\bullet \pi_*\sA ^{\ke}_{d'})^\vee$ with the relative obstruction theory
over $\Bun$ for $Q^\ke _{g, k}(\PP (V\ot \CC ^N), d')$, see \cite[\S5.3]{CK}.
Here $\sL '$ denotes, by abusing notation, also the universal line bundle on the universal curve on $Q^\ke _{g, k}(\PP (V\ot \CC ^N), d')$.

\begin{Lemma}\label{normal cone}
 The relative normal cone $\mathbf{C}_{Q^\ke_{g, k}(X, d) / U^{\ke}_{d'}}$ for the embedding in \eqref{embedding diagram} coincides with the obstruction cone $\mathbf{C}_\ke\subset E^\ke _d$.
\end{Lemma}

\begin{proof} First, we have by definition 
$$\mathbf{C}_\ke=\bC_{in}\times_{[E_d^\ke/T_{U_{d'}^{\ke}/\Bun }]}E_d^\ke,$$
where $\bC_{in}$ is the relative intrinsic normal cone 
of $Q^\ke_{g, k}(X, d)$ over $\Bun$ (see \cite{BF}) and $[E_d^\ke/T_{U_{d'}^{\ke}/\Bun }]$ denotes the stack quotient.
Since
$\bC_{in}=[\mathbf{C}_{Q^\ke_{g, k}(X, d) / U^{\ke}_{d'}}/T_{U_{d'}^{\ke}/\Bun }]$, the Lemma follows.
\end{proof}

Proposition \ref{virtual class} and Lemma \ref{normal cone} imply the following concrete
description of the virtual classes of moduli spaces of $\ke$-stable quasimaps to $X$.
\begin{Cor}\label{concrete virtual class}
$$ [\Qked ]^{\vir} = 0_{E_d^\ke}^! ([\mathbf{C}_{Q^\ke_{g, k}(X, d) / U^{\ke}_{d'}}]).$$
\end{Cor}

\begin{Rmk} Recall that in genus zero we take a trivial twisting line bundle $\sM$, so in this
 case $U_{d'}^{\ke}=Q^\ke_{0,k}(\PP(V),d)$ and the construction reduces to the known realization of 
  $Q^\ke_{g, k}(X, d) $ as the zero locus of a section of the bundle $\oplus_i\pi_*(\sL)^{l_i}$ on 
  $Q^\ke_{0,k}(\PP(V),d)$. This bundle has ``correct" rank $d\sum_il_i +r$, hence its refined top Chern class gives  $[Q^\ke_{g, k}(X, d) ]^{\mathrm{vir}}$.                                                                              
However, for $g\geq 1$ the rank of the bundle 
${F}_{d'}^\ke=\pi _*\sP ^{\ke}_{d'} \oplus  \pi _* \sR^{\ke}_{d'} $ is larger than the virtual codimension 
of $Q^\ke_{g, k}(X, d) $ in $U_{d'}^{\ke} $, so the virtual class is {\it not} the refined top Chern class.
\end{Rmk}

\section{Proof of Theorem \ref{Main}}\label{proof of Main Thm}

\subsection{Overview} Adapting an idea of Bertram from \cite{Be}, we consider a one-parameter degeneration of the diagram
\eqref{embedding diagram} which is obtained
 via a refinement of MacPherson's Graph Construction. The proof of Theorem \ref{Main} will then follow by
 analyzing the central fiber limit of the virtual cycle 
 $[Q^+_{g,k}(X,d)]^{\mathrm{vir}}$ in this degeneration.

\subsection{Boundary strata}\label{sec:boundary}

Let $\ke_0$ be a wall, so that  $m:=1/\ke_0$ is a positive integer. 
Let $\ke_+ >\ke_0\geq \ke_{-}$ be stability parameters separated only by the single wall $\ke_0$.
Fix the numerical data $(g,k,d)$.
We will denote by $Q^{\pm}_{g,k}(X,d)$, $U^{\pm}_{d'}$ etc. the moduli spaces corresponding to the stability parameters $\ke_{\pm}$. The contraction morphisms
with the abused notation
$$c:Q^{+}_{g,k}(X,d)\lra Q^{-}_{g,k}(X,d),\;\;\; c:U^{+}_{d'}\lra U^{-}_{d'}$$
contract precisely the rational tails of degree $m$.

The evaluation maps at the markings will be denoted by $\hat{ev}_j$ for $Q^\ke_{g,k}(\PP(V\ot \CC^N),d')$ and for its open substack $U^{\ke}_{d'}$, while we reserve the notation $ev_j$ for the evaluation maps on $Q^{\ke}_{g,k}(\PP(V),d)$ and on $Q^{\ke}_{g,k}(X,d)$.

For a finite index set $A$, with $|A|= 1, 2, ..., [\frac{d}{m}]$ we associate to each $a\in A$ the integer
$d_a=m$ 
and set \begin{equation}\label{eqn:d_0} d_0=d_0^A := d - \sum _{a\in A} d_a =d-|A|m\ge 0. \end{equation}

Denote
\begin{align*}
        D_A & :=  U^{+}_{k+A,  d_0'}  \times _{\PP(V\ot \CC ^N)^A} 
\prod_{a\in A}  Q^+_{0,a}(\PP(V \ot \CC ^N),  d_a) ,  \\
      \tilde{D}_A & :=  U^{+}_{k+A, d_0'}  
\times _{\PP(V \ot \CC ^N)^A} \prod_{a\in A} \fC^+_{0,a,d_a}  ,
        \end{align*}
where $ \fC^+_{0,a,d_a} \ra Q^+_{0,a}(\PP(V \ot \CC ^N),  d_a)$ is the universal curve, the notations $U^{\pm}_{k+A,  d_0'} $ are the obvious ones, and the fiber products are made via $(\hat{ev}_a)_{a\in A}$ on the left and 
$\prod_{a\in A}\hat{ev}_a$ on the right. The easiest way to describe the evaluation map 
$\hat{ev}_a: \fC^+_{0,a,d_a} \ra \PP(V \ot \CC ^N)$ is by identifying $ \fC^+_{0,a,d_a}$ with the moduli stack $Q^+_{0,a|1}(\PP(V \ot \CC ^N),  d_a)$ which parametrizes $\ke_+$-stable quasimaps of degree $d_a$ from rational curves with one marking $a$ of weight $1$ and one additional marking of weight $0+$, see \cite{bigI} for more on these moduli stacks.

We will need an alternative description of these boundary strata which takes into account the twisting line bundles $\sM$.

Consider the diagram of universal curves
\begin{equation}\label{univ curves}
\xymatrix{\fC ^{+}_{g, k, d'} \ar[rd]_{\pi} \ar[r]\ar@/^1pc/[rr]^{\tilde{c}}  \ar@/^3pc/[rrr]^{h_{+}} 
& c^*\fC ^{-}_{g, k, d'}\ar[d]\ar[r]
&\fC ^{-}_{g, k, d'}\ar[d]^{\pi} \ar[r]^{h_{-}} &\PP(\CC^N)\\
& U^{+}_{d'}  \ar[r]^{c}    & U^{- }_{d'}  &
}
\end{equation}
with cartesian square and the maps $h_{\pm}$ given by the sections $s_1,\dots s_N\in \Gamma(\fC ^{-}_{g, k, d'},\sM_-)$, so that $\sM_{\pm}=(h_{\pm})^*(\cO_{\PP(\CC^N)}(1))$. For each $a\in A$ we obtain maps
\begin{equation} \label{eqn:h_a} h_a^{\pm}: U^{\pm}_{k+A,d_0'}\lra\PP(\CC^N) \end{equation}
as the compositions
$$\xymatrix{ h_a^- : U^{-}_{k+A,d_0'}\ar^ -{\Sigma_a} [r] & \fC^-_{g,k+A, d_0'}\ar^-{\tilde{b}_A}[r] &\fC ^{-}_{g, k, d'}\ar^{h_-}[r] & \PP(\CC^N),
}
$$
$$
\xymatrix{ h_a^+ : U^{+}_{k+A,d_0'} \ar^ -{{c}_A } [r]& U^{-}_{k+A,d_0'} \ar^ -{h^-_a}[r] & \PP(\CC^N) .}
$$
Here $\Sigma_a$ is the section corresponding to the marking $a\in A$, ${\tilde{b}_A}$ is the map that trades each marking in $A$ for a base-point of length $d_a$, and ${c}_A$ is the contraction of rational tails of degree $d_a$.
There is a natural identification
\begin{equation}\label{new D_A}
 D_A \cong U^{+}_{k+A, d_0'}  \times _{(\PP(V\ot \CC ^N)\times\PP(\CC^N))^A} 
\prod_{a\in A}  (Q^+_{0,a}(\PP(V \ot \CC ^N),  d_a) \times\PP(\CC^N)),
\end{equation}
where the fiber product is now done using $( (\hat{ev}_a, h^+_a))_{a\in A}$ on the left and $\prod_{a\in A} (\hat{ev}_a\times {\mathrm {id}}_{\PP(\CC^N)}) $ on the right. Similarly,
\begin{equation*} 
 \tilde{D}_A \cong  U^{+}_{k+A, d_0'}  
\times _{(\PP(V\ot \CC ^N)\times\PP(\CC^N))^A} \prod_{a\in A} (\fC^+_{0, a,d_a}\times\PP(\CC^N)).
\end{equation*}

We have the following commuting diagram of canonical morphisms:
\begin{equation}\label{bdry starta} 
\xymatrix{   U^{+}_{d'}  \ar[rr]^{{c}}  &  & U^{- }_{d'}    \\
\ar[d]_{\pr _{A}} { D}_A  \ar[r]^{\mpr  _{A} \ \ \ \ \ \ \  } \ar[u]^{\nu_A}  &   U^{+}_{k+A, d_0'}  \ar[r]^{c_A}  & U^{- }_{k + A , d_0'}  \ar[u]_{b_A}   \\
 \prod_{a\in A} (Q^{+}_{0, a}(\PP(V \ot \CC ^N), d_a)\times\PP(\CC^N)) , &    &  }  
 \end{equation}
where ${b}_A$ denotes the morphism which trades the markings $A$ for base points of length  $d_a$. The two projections ${\pr}_A$ and ${\mpr}_A$ are those coming from the fiber product description \eqref{new D_A} of ${D}_A$. 
The map $\nu_A$ 
has degree $|A|!$ and 
sends  ${D}_A$ onto the boundary
stratum of $U^{+}_{d'}$ generically parametrizing (unobstructed) $\ke_+$-stable quasimaps to $\PP(V\ot \CC^N)$ whose domain curves have exactly $|A|$ unordered rational tails of degree $d_a$. In particular, for $A=\{a\}$ the map $\nu_{\{a\}}$ is an
embedding of ${D}_{\{a\}}$ as a boundary divisor.

  The contractions ${c}$, ${c}_A$ are isomorphisms over the (nonempty) loci of quasimaps with irreducible domain curves. By Lemma \ref{Con}, the complements of these loci have positive codimension and we conclude that ${c}$, ${c}_A$  
are birational morphisms and hence degree 1 maps. 

We finally introduce one more piece of notation.
Let $p_a$ denote the Cartier divisor on the universal curve $\fC _{g, k+\{a\}, d'_0}^{\pm}$ of the moduli spaces  $U^{\pm}_{k+\{a\},d'_0}$ which is the image of the section $\Sigma_a$ corresponding to the marking $a$. Similarly, we have the Cartier divisor $p_a^{tail}$ on the universal curve $\fC^+_{0,a, d_a} \ti \PP (\CC ^N)$
of $Q^+_{0,a}(\PP(V \ot \CC ^N),  d_a) \ti \PP (\CC ^N)$ defined by the image of the section $\Sigma _{tail, a}$  corresponding to the marking $a$.
As usual, $\cO(p_a)$, respectively $\cO(p_a^{tail})$, will stand for the associated line bundles; and $\cO _{p_a}$, respectively $\cO_{p_a^{tail}}$ will stand
for the coherent sheaves $\Sigma _{a *} \Sigma _a ^* \cO$, $\Sigma _{tail, a *}\Sigma _{tail, a}^*\cO$ on the universal curves.
Then $\Sigma _a^*\cO (-p_a)$, respectively $\Sigma _{tail, a}^*\cO (-p_a^{tail})$, is identified with the 
line bundle with first Chern class $\psi_a$ on $U^{\pm}_{k+A, d'_0}$, respectively $\psi_a^{tail}$ on 
$Q^+_{0,a}(\PP(V \ot \CC ^N),  d_a)\ti \PP(\CC^N)$. 
Abusing notation, we will write $\cO(\psi_a)$ and $\cO(\psi_a^{tail})$ for these line bundles, and $\cO(-\psi_a)$, $\cO(-\psi_a^{tail})$ for their duals.

\subsection{MacPherson's Graph Construction}

For easy notation, for $A = \{ a\}$ in \eqref{bdry starta} we write $D_a$, $\mathrm{Pr} _a$, $c_a$, $b_a$, etc instead 
of $D_{\{a\}} $, $\mathrm{Pr} _{\{a\}} $ $c_{\{a\}}$ $b_{\{a\}}$, etc.
Let $\pi: \fC _{g, k, d'}^{\pm} \ra U^{\pm}_{d'} $ be the universal curve and
 denote  by 
 $\tilde{c}$ the contraction morphism from $\fC _{g, k, d'}^+$ to $\fC_{g,k, d'}^{-}$,
                          which is an isomorphism outside the divisor  $\tilde{D}_{a}$.
 Hence $\sL'_{+} \cong \tilde{c}^*\sL '_{-} (-d_a \tilde{D}_{a})$.
 Here the coefficient $-d_a$ is obtained by the consideration of $\deg \sL '_{+} |_{C_a} = d_a$,
 $\deg \cO_{C_a} (C_a) = -1$ for the contracted rational tail $C_a$ on the fiber curve of $\pi$ over a general closed point 
 of $D_a$.
It follows that for every $l\geq 1$
there are homomorphisms
\[ (\sL'_{+})^l     \cong \tilde{ c}^* (\sL' _{- })^l (- l d_a\tilde{{D}}_{a})  \ra \tilde{c}^*(\sL' _{- })^l \]
of line bundles on $\fC _{g, k, d'}^+$.

In particular, taking $l=1$ and using the top line of the diagram \eqref{twisting} gives a map 
 $\sP^{+}_{d'}\lra\tilde{c}^* (\sP^{-}_{d'})$. Applying $\pi_*$  we obtain homomorphisms
 \begin{align*}
 \Phi_P: &  P^+_{d'}\lra{c}^*P^-_{d'}, \;\;\;\;  \Phi_R:  R^+_{d'}\lra{c}^*R^-_{d'},\\
& \Phi=(\Phi_P,\Phi_R): {F}^+_{d'}\lra  c^*{F}^-_{d'}
 \end{align*}
of vector bundles on $U^+_{d'}$, which are isomorphisms outside $ D_{a}$.
We have used here the canonical isomorphisms  $\pi_*\tilde{c}^* \sR^{-}_{d'} 
\cong  c^*\pi_*\sR^-_{d'}$ and
$ \pi_*\tilde{c}^*\sP^{-}_{d'}\cong   c^*\pi_*\sP^-_{d'}$ obtained by applying to \eqref{univ curves} 
the base-change followed by the projection formula.

Consider the Grassmann bundle over $ U^{+}_{d'}$
$$\Gr :=\Gr ({F}^+_{d'}\oplus  c^*{F}^-_{d'}):=
Grass(r_d, {F}^+_{d'}\oplus  c^*{F}^-_{d'}),$$
with $r_d=\rank({F}^+_{d'})$. 
Let $\eta: \Gr  \ra U^{+}_{d'}$ be the projection and denote by 
$\zeta$ the tautological subbundle of rank $r_d$ in $\eta^*({F}^+_{d'}\oplus  c^*{F}^-_{d'})$.

The map $\eta\times \id$ has a section 
\begin{equation*} 
v: U^{+}_{d'}\times \AAA ^1\lra \Gr \times \AAA ^1,\;\;
v(y,\lambda)=(y,\mathrm{graph}(\lambda (\Phi)_y), \lambda ).\end{equation*} 
Define the closed substack
\begin{equation*} 
\Gamma :=\overline{\mathrm{Im}(v)}\subset \Gr \times \PP^1
\end{equation*}
as the stack-theoretic closure of the image of $v$. As $ U^{+}_{d'}$ is nonsingular and irreducible, ${\Gamma}$ is also irreducible, of dimension equal to $1+\dim U^{+}_{d'}$.

In fact, if we consider the ``component" Grassmann bundles
$$\Gr _P:=\Gr (P^+_{d'}\oplus  c^*P^-_{d'}):=
Grass(r_P, P^+_{d'}\oplus  c^*P^-_{d'}),$$
$$\Gr _R:=\Gr (R^+_{d'}\oplus  c^*R^-_{d'}):=
Grass(r_R, R^+_{d'}\oplus  c^*R^-_{d'}),$$
with projections $\eta_P,\eta_R$ and tautological subbundles ${\zeta}_P,{\zeta}_R$,
then there is a natural inclusion
$$\Gr_P\times_{U^{+}_{d'}}\Gr_R\subset \Gr$$
such that $ \zeta$ restricts to $\zeta_P\boxplus\zeta_R$ and the inclusion of $\Gamma$ in $\Gr\times\PP^1$ factors through
$(\Gr_P\times_{U^{+}_{d'}}\Gr_R)\times\PP^1$.

For $\lambda\in \PP^1=\AAA^1\cup\{\lambda=\infty\}$ denote by $\Gamma_\lambda$ the fiber of the projection $\Gamma\ra\PP^1$. 
When $\lambda\in \AAA^1$, under the identifications 
$v_\lambda: U^+_{d'}\stackrel{\cong}{\lra}\Gamma_\lambda$, we have 
$$v_\lambda^* \zeta = \mathrm{Im}({F}^+_{d'}\stackrel{(\id,\lambda\Phi)}{\lra}{F}^+_{d'}\oplus  c^*{F}^-_{d'}).$$
In particular, at $\lambda=0$ we have
$v_0^* {\zeta}={F}^+_{d'}\oplus\{0\}$.

At $\lambda=\infty$ the fiber breaks into components encoding the degeneracy of the map $\Phi$, as in \cite[Example 18.1.6]{Fu}. First of all, there is a distinguished component 
$\Gamma_{\infty, dist}$ which has multiplicity one and projects 
birationally to $U^+_{d'}$, while $\zeta|_{\Gamma_{\infty,dist}}=\{0\}\oplus c^*{F}^-_{d'}$. All other components 
of $\Gamma_{\infty}$ come with some multiplicities and project into $ D_{a}$ under $\eta$. 
Their description is our next task. The analysis is similar to the one in the proof of \cite[Lemma 4.4]{Be}, where a related genus zero case is treated. In our situation there are complications due to the twisting by $\sM$, but also slight simplifications, due to the fact that 
 $ c$ only contracts rational tails of fixed degree $d_a$, which therefore do not interfere 
with each other.

\subsubsection{Description of $\Gamma _{\infty}$}
For each $j_a\geq 1$ consider the $\PP^1$-bundle over ${D}_{a}$
$$\PP_{j_a}:= \PP\left( {\pr}_{a}^*\cO(j_a\psi_a^{tail})\oplus {\mathrm{Pr}}_{a}^*\cO(-j_a\psi_a) \right)$$
and their fiber product
$$ \PP _{j_A} := \prod _{a \in A} \PP _{j_a}|_{D_A} $$
over $D_A$.

\begin{Thm}\label{thm:Gamma des} 
Let $j_A$ be the multi-index $(j_a)_{a\in A}$ with each $j_a$ in the range $1 \le  j_a \le \mathrm{max} \{d_a, d_a l _i \, | \, i=1, ..., r\}$ and 
let $m_{j_A}:=\prod_{a\in A}j_a$.
For each $j_A$, there exists a map $\alpha_{j_A}: \PP _{j_A}  \to {\bf Gr}$, described below,  satisfying that
\begin{equation}\label{special fiber 1}
[{\Gamma}_\infty]=[{\Gamma}_{\infty, dist}]
+\sum_{ (A,j_A)} m_{j_A} [\Gamma_{\infty,j_A}] =[{\Gamma}_{\infty, dist}]+\sum_{(A,j_A)} \frac{m_{j_A}}{|A|!}
(\alpha_{j_A})_*[\PP_{j_A}]
\end{equation} 
in the Chow group $A_*(\Gr)_\QQ$. Here 
$\Gamma_{\infty,j_A}$ is the image stack of $\alpha _{j_A}$. Furthermore $\Gamma _{\infty, j_A}$ projects to $D_A$ under the projection
map $\eta : \Gr \to U^+_{d'}$.
\end{Thm}

Defining $\alpha _{j_A}$ amounts to finding a subbundle $\xi^{j_A}$ 
of $\pi_{\PP}^*\nu_{A}^*({F}^+_{d'}\oplus c^*{F}^-_{d'})$ with its rank equal to the rank of $F^{+}_{d'}$.
Denote by $\pi_{\PP}: \PP _{j_A} \to D_A$ the projection map.
Then the vector bundle  $\xi^{j_A}$
will be constructed as an extension of 
$ \boxplus _{a \in A} \cO_{\PP_{j_a}} (-1) \ot \pi _{\PP}^* {\tt F}^{j_a} $
by 
$ \pi_{\PP}^*({\pr}_{A} ^*{F}^{+, j_A+1}_{tail, d_a} \oplus  {\mpr}_{A} ^*{c}_{A}^*{F}^{-, j_A-1}_{ d'_0} ) $
for some vector bundles 
$${\tt F}^{j_a}, F^{+, j_A +1}_{tail, d_a}, {F}^{-, j_A-1}_{d'_0} \text{ on } D_a, \prod _{a\in A} Q^{+}_{0, a}(\PP(V \ot \CC ^N), d_a)\times\PP(\CC^N),
U^{-}_{k+A, d_0'} \text{ respectively. } $$
The bundles $\pr _a^* F^{+, j_a}_{tail, d_a}$ (resp. $\mpr_a^*c_a^* {F}^{-, j_a}_{d'_0}$) for $j_a$ will
form a decreasing (resp. increasing) filtration of the kernel sheaf of $\nu _a^*\Phi$
(resp. of the sheaf $\nu_a^* c^*F^-_{d'}$).

\subsubsection{Description of the vector bundle $F^{+, j_a +1}_{tail, d_a}$ on $Q^{+}_{0, a}(\PP(V \ot \CC ^N), d_a)\times\PP(\CC^N)$}

Consider first the case $A=\{ a\}$ of the boundary divisor ${D}_{a }$. 
On the universal curve $$\pi:\fC_{0,a,d_a}^+\times\PP(\CC^N)\ra Q^{+}_{0, a}(\PP(V \ot \CC ^N), d_a)\times\PP(\CC^N),$$ put
$\sL_+ := \sL'_+\boxtimes\cO_{\PP(\CC^N)}(-1)$.
We have the diagram
\begin{equation*} 
\xymatrix{ 0 \ar[r]   &\mathscr{L}_+ \ot V \ar[r]^{\oplus _j s_j}   &  \oplus _{j=1}^N  \mathscr{L}' _+\ot V \ar[r]\ar[d]^{\oplus_j(\oplus_id\varphi_i)}        & \mathscr{P}^+_{tail,d_a}  \ar[r]                 & 0 \\
                      0 \ar[r]  & \oplus _{i=1}^r  \sL_+  ^{l_i}  \ar[r]^{\oplus _{i, j} s_j^{l_i}}       
& \oplus _{i, j}   (\sL '_+) ^{l_i} \ar[r]                              & \mathscr{Q}^+_{tail,d_a}  \ar[r]              & 0 ,} \end{equation*}
whose rows are obtained from the exact sequence
$$0 \ra \cO_{\PP(\CC^N)}(-1) \ra \oplus _{j=1}^N \cO_{\PP(\CC^N)}\ra Q\ra 0$$
via pull-backs, tensoring with appropriate line bundles, and taking direct sums, as explained in \S 3.
Now define the vector bundles
\begin{equation*} 
P^+_{tail, d_a}:= \pi_*\sP ^+_{tail, d_a},\; \sR^{+}_{tail, d_a} := \oplus _{i,j} (\sL ' _{+})^{l_i}, \; R^+_{tail, d_a}:=\pi_*\sR^{+}_{tail, d_a},\;
{F}^+_{tail,d_a}:=P^+_{tail, d_a} \oplus R^+_{tail, d_a}.
\end{equation*}
For integers $j_a=1,\dots, \mathrm{max} \{d_a, d_al_i \, | \, i=1, ..., r \}$, we have the subbundles
\begin{equation}\label{P filtration}
 P^{+, j_a}_{tail, d_a} : =  \pi_* ( \sP^+_{tail, d_a}(- j_ap_a^{tail}) ),
 \end{equation}
\begin{equation} \label{R filtration}
 R^{+, j_a}_{tail, d_a} : =  \pi_*    (\sR^+_{tail, d_a}  (- j_ap_a^{tail}))     
\end{equation} of vector bundles $P^+_{tail, d_a}$, $R^+_{tail, d_a}$ respectively.
They are vector bundles on $Q^+_{0,a}(\PP (V\ot\CC^N), d_a)\times\PP(\CC^N)$.
We also put $$ P^{+, 0}_{tail, d_a} : =P_{tail,d_a}^+,\;  R^{+, 0}_{tail, d_a} : =R_{tail,d_a}^+, \;
{F}^{+, 0}_{tail, d_a} : ={F}_{tail,d_a}^+.$$

 Note that
$ P^{+, j_a}_{tail, d_a}=0$ if $j_a >d_a$, and that $(\sL'_+)^{l_i}$ does not contribute to $ R^{+, j_a}_{tail, d_a}$ if 
$j_a >l_id_a$. Hence the quotients of the decreasing filtrations given by \eqref{P filtration} and \eqref{R filtration} are
$$0\ra P^{+, j_a+1}_{tail, d_a}\ra P^{+, j_a}_{tail, d_a}\ra 
 {\tt P} _{tail}^{j_a}\ot \cO(j_a\psi_a^{tail}) \ra 0,$$
$$0\ra R^{+, j_a+1}_{tail, d_a}\ra R^{+, j_a}_{tail, d_a}\ra 
 {\tt R}_{tail}^{ j_a}\ot \cO(j_a\psi_a^{tail}) \ra 0 ,$$
where we put for each $0\leq j_a \le \mathrm{max} \{d_a, l_i d_a \, | \, i=1, ..., r \}$
$$ {\tt P} _{tail}^{ j_a}  :=  \left\{\begin{array}{rr}  ({ev}_a\times \id_{\PP(\CC^N)})^* \left( (\cO_{\PP(V\ot\CC^N)}(1)\ot V)\boxtimes Q \right),
& \text{if } j_a \le d_a \\
0, & \text{if } j_a > d_a \end{array}\right. 
$$
and
$$ {\tt R}_{tail}^{ j_a}:=\oplus_{i=1}^r {\tt R}_{i, tail}^{j_a},$$
$$ {\tt R}_{i, tail}^{j_a}  :=  \left\{\begin{array}{rr}  ({ev}_a\times\id_{\PP(\CC^N)})^*
\left(  \cO_{\PP(V\ot\CC^N)}(l_i)\boxtimes \oplus_{j=1}^N\cO_{\PP(\CC ^N)}  \right),
& \text{if } j_a \le l_id_a \\
0, & \text{if } j_a > l_i d_a \end{array}\right. .$$ 
Alternatively, when they are not set to zero, 
\begin{align*}
{\tt P}_{tail}^{j_a}=  \pi_*(\mathscr{P}^{+}_{tail, d_a} \ot \cO_{p_a^{tail}}) ,\; \;  {\tt R}_{tail}^{j_a}=   \pi_*(  \sR^{+}_{tail, d_a} 
\ot  \cO_{p_a^{tail}}).
\end{align*}

Taking the direct sums 
$$ {F}^{+, j_a}_{tail, d_a} : =  P^{+, j_a}_{tail, d_a}\oplus R^{+, j_a}_{tail, d_a},\;\;\;  {\tt F}^{j_a}_{tail} :={\tt P}^{j_a}_{tail}\oplus  {\tt R}^{j_a}_{tail}$$
gives a filtration of the vector bundle ${F}_{tail,d_a}^+$ on $Q^{+}_{0, a}(\PP(V \ot \CC ^N), d_a)\times\PP(\CC^N)$, with quotients 
${\tt F}^{j_a}_{tail} \ot \cO(j_a\psi_a^{tail})$.
 The pull-back $\nu_{a}^*{F}^+_{d'}$ can be written 
as the extension   
\begin{align}\label{res}
\xymatrix{ 0\ar[r] & {\mathrm{pr}}_{a}^*{F}^{+, 1}_{tail, d_a}\ar[r] &\nu_{a}^*{F}^+_{d'}
\ar[r]^{res\;}   & {\mpr} _{a} ^*{F}^+_{d'_0} \ar[r] & 0} .
\end{align}

\subsubsection{Description of the vector bundle $F^{-, j_a-1}_{ d_0'}$ on $U^{-}_{k+\{a\}, d_0'}$} 
Let ${F}^\pm_{b_A, d'_0}$ denote the vector bundles on $U^{\pm}_{k+A,  d_0'}$ defined as in
\eqref{E hat}, but using the twisting line bundles $\sM^\pm$ {\it induced from} $\fC ^{-}_{g, k, d'}$ (and hence from $\overline{M}_{g,k}$) via
pull-back by
$${\tilde{b}_A}: \fC^-_{g,k+A, d_0'}\lra\fC ^{-}_{g, k, d'}.
$$

The homomorphism $\Phi$ factors when pulled-back to ${D}_A$ as 
\begin{equation*} 
\xymatrix{ 
\nu_A ^* {F}^+_{d'} \ar[r]^{res\ \ \;}   & {\mpr} _A ^*{F}^+_{b_A, d'_0} \ar[rr]^{\text{generic. isom}\ } &   
&  {\mpr} _A ^*{c}_A^* {F}^-_{b_A, d'_0} \ar@{^(->}[r] 
&  {\mpr}_A ^*{c}_A^*{b}_A ^* {F}^-_{d'} =  \nu_A ^* c^*{F}^-_{d'}  } .  \end{equation*}
Here the first map $res$ is given by the restriction of sections to the non-contracted parts of the universal 
curve.
The middle arrow is the pull-back by $ {\mpr} _A$ of the map $\Phi$ on $U^+_{k+A, d'_0}$ and is therefore an isomorphism generically on ${D}_A$. 
The third map is induced from the canonical injections on the universal curve
$\sL'_{-,d'_0}\ra\sL'_{-,d'_0}(\sum_a d_a p_a)=\tilde b_A^*\sL'_{-,d'}$ and
$(\sL'_{-,d'_0})^{l_i}\ra(\sL'_{-,d'_0})^{l_i}(\sum_a l_i d_a p_a)=\tilde b_A^*(\sL'_{-,d'})^{l_i}$.

Consider  the codomain $  {\mpr}_{a} ^*{c}_{a}^*{b}_{a} ^*  F^-_{d'} $ of
$\Phi|_{ D_{a}}$ and the square diagram of universal curves
\[ \xymatrix{ \fC^{-}_{g, k+\{a \}, d_0'} \ar[r] \ar[d]_{\pi} &  \fC^{-}_{g, k, d'} \ar[d] \\
 U^-_{k+\{a \}, d_0'} \ar[r]_{b_a} & U^-_{k, d'} . } \]
 In the bundle ${b}_{a} ^*  F^-_{d'} $ on $U^-_{k+\{a\}, d_0'}$
we have the increasing filtrations 
\begin{align*}
& P^{-, 0}_{d'_0}\subset 
 P^{-, 1}_{d'_0}\subset \dots
\subset  P^{-, d_a}_{d'_0}=
{b}_{a} ^* P^-_{d'},\\
&  R^{-, 0}_{d'_0}\subset 
 R^{-, 1}_{d'_0}\subset \dots
\subset  R^{-, \mathrm{max}_i \{d_al_i\}}_{d'_0}=
{b}_{a} ^* R^-_{d'}
\end{align*}
induced via the subbundles
$$
P^{-, j_a}_{d'_0} : =\pi_*\left(  \sP^{-}_{d'_0}  ( j_a p_a)
\right)  \cap   b_a^*P^-_{d'}, \; j_a=0,1,\dots, d_a,
$$
$$
R^{-, j_a}_{d'_0} : =\pi_*\left(        \sR^-_{d'_0}  ( j_a p_a) 
\right)  \cap   {b}_{a} ^* R^-_{d'}, \; j_a=0,1,\dots ,\mathrm{max}_i \{ l_id_a\}.
$$
Here we use the natural injections $ \sP^{-}_{d'_0}  ( j_a p_a) \to \sP^{-}_{d'_0}  ( d_a p_a) \cong \tilde{b}_a^* \sP^{-}_{d'}$ for $j_a \le d_a$
and   $\sR^-_{i, d'_0}  ( j_a p_a)  \to  \sR^-_{i, d'_0}  ( l_i d_a p_a) \cong \tilde{b}_a^*\sR^-_{i, d'}$ for $j_a \le l_id_a$.
The  quotients are
$$0\ra  P^{-, j_a-1}_{d'_0}\ra  P^{-, j_a}_{d'_0}\ra
{\tt P}^{-, j_a}\ot \cO(-j_a\psi_a) \ra 0,
$$
$$0\ra  R^{-, j_a-1}_{d'_0}\ra  R^{-, j_a}_{d'_0}\ra
 {\tt R}^{-, j_a}\ot \cO(-j_a\psi_a) \ra 0.
$$
where we put for each $0\leq j_a \le \mathrm{max} \{d_a, l_i d_a \, | \, i=1, ..., r \}$
\begin{eqnarray}\label{eqn:def tP-} {\tt P} ^{-, j_a}  :=  \left\{\begin{array}{rr}   \pi_*( \mathscr{P}^{-}_{d'_0} \ot \cO_{p_a}),
& \text{if } j_a \le d_a \\
0, & \text{if } j_a > d_a \end{array}\right. ,
\end{eqnarray}
and
$$ {\tt R}^{-, j_a}:=\oplus_{i=1}^r {\tt R}^{-,j_a}_{i} ,$$
\begin{eqnarray}\label{eqn:def tR-i} {\tt R}^{- , j_a}_i  :=  \left\{\begin{array}{rr} 
 \pi_*(  \oplus _{j=1}^N  (\sL '_{-})^{l_i} \ot  \cO_{p_a}),
& \text{if } j_a \le l_id_a \\
0, & \text{if } j_a > l_i d_a \end{array}\right. . \end{eqnarray}

Setting
$$ F_{d'_0}^{-,j_a}:=P^{-, j_a}_{d'_0}\oplus R^{-, j_a}_{d'_0}$$
gives an increasing filtration of the vector bundle $b_{a} ^*  F^-_{d'} $ on $U^-_{k+\{a\}, d_0'}$ with quotients 
${\tt F}^{-, j_a} \ot \cO(-j_a\psi_a)$ and ${\tt F}^{-, j_a}  := {\tt P} ^{-, j_a} \oplus {\tt R}^{-, j_a}$.

\subsubsection{Description of $\alpha _{j_a} : \PP_{j_a} \to \Gr$}
For each $j_a\geq 1$ recall the $\PP^1$-bundle over ${D}_{a}$
$$\PP_{j_a}:= \PP\left( {\pr}_{a}^*\cO(j_a\psi_a^{tail})\oplus {\mathrm{Pr}}_{a}^*\cO(-j_a\psi_a) \right),$$
with projection $\pi_{\PP}:\PP_{j_a}\lra {D}_{a}$. Consider the tautological sequence
$$0\lra \cO_{\PP_{j_a}}(-1)\lra \pi_{\PP}^*\left( {\pr}_{a}^*\cO(j_a\psi_a^{tail})\oplus {\mathrm{Pr}}_{a}^*\cO(-j_a\psi_a) \right)\lra 
\cO_{\PP_{j_a}}(1)\lra 0.$$

Now
define the extension ${ \xi}_P^{j_a}$ as the vector bundle uniquely fitting in the commuting diagram with exact columns

\[ \xymatrix{    0  &  0     \\
  \cO _{\PP_{j_a}}(-1) \ot \pi_{\PP}^*{\tt P} ^{j_a} \ar@{^{(}->}[r] \ar[u] 
& \pi_{\PP}^*\left(  ({\pr}_{a} ^*\cO (j_a\psi_a^{tail}) \oplus {\mpr}_{a}^*\cO (-j_a\psi_a) )\ot {\tt P}^{j_a} \right)   \ar[u] \\
     {\xi}_P^{j_a} \ar[u]\ar@{^{(}->}[r]   & \pi_{\PP}^*\left({\pr}_{a} ^*P^{+, j_a}_{tail, d_a} \oplus  {\mpr}_{a} ^*{c}_{a}^*P^{-, j_a}_{d'_0} \right) \ar[u]  \\
    \pi_{\PP}^*\left({\pr}_{a} ^*P^{+, j_a+1}_{tail, d_a} \oplus  {\mpr}_{a} ^*{c}_{a}^*P^{-, j_a-1}_{d'_0} \right)\ar[r] \ar[r]_{=} \ar[u] 
& \pi_{\PP}^*\left( {\pr}_{a} ^*P^{+, j_a+1}_{tail, d_a} \oplus  {\mpr}_{a} ^*{c}_{a}^*P^{-, j_a-1}_{d'_0} \right)\ar[u]  \\
  0 \ar[u] & 0 \ar[u]  } \]
  
  where the horizontal arrows are injective as {\it maps of vector bundles}
  and $${\tt P}^{j_a}  := \pr ^*_a {\tt P}_{tail}^{ j_a}  \cong \mpr ^*_a c_a^* {\tt P}^{-, j_a} . $$

Similarly, we define $\xi_R^{j_a}$ as an extension, via
\[ \xymatrix{    0  &  0     \\
  \cO _{\PP_{j_a}}(-1) \ot \pi_{\PP}^* {\tt R}^{j_a} \ar@{^{(}->}[r] \ar[u] 
& \pi_{\PP}^*\left(  ({\pr}_{a} ^*\cO (j_a\psi_a^{tail}) \oplus {\mpr}_{a}^*\cO (-j_a\psi_a) )\ot  {\tt R}^{j_a} \right)   \ar[u] \\
     {\xi}_R^{j_a} \ar[u]\ar@{^{(}->}[r]   & \pi_{\PP}^*\left({\pr}_{a} ^*R^{+, j_a}_{tail, d_a} \oplus  {\mpr}_{a} ^*{c}_{a}^*R^{-, j_a}_{d'_0} \right) \ar[u]  \\
    \pi_{\PP}^*\left({\pr}_{a} ^*R^{+, j_a+1}_{tail, d_a} \oplus  {\mpr}_{a} ^*{c}_{a}^*R^{-, j_a-1}_{d'_0} \right)\ar[r] \ar[r]_{=} \ar[u] 
& \pi_{\PP}^*\left( {\pr}_{a} ^*R^{+, j_a+1}_{tail, d_a} \oplus  {\mpr}_{a} ^*{c}_{a}^*R^{-, j_a-1}_{d'_0} \right)\ar[u]  \\
  0 \ar[u] & 0 \ar[u]  } \]
  where
   $${\tt R}^{j_a}  := \pr ^*_a {\tt R}_{tail}^{j_a}  \cong \mpr ^*_a c_a^*{\tt R}^{-, j_a} . $$

Since
\begin{equation*} 
\xi^{j_a}:= \xi_P^{j_a}\oplus \xi_R^{j_a}
\end{equation*}
is canonically a subbundle of
$\pi_{\PP}^*\nu_{a}^*({F}^+_{d'}\oplus c^*{F}^-_{d'})$ whose rank is equal to the rank of ${F}^+_{d'}$,
it 
gives rise to a morphism $$\alpha_{j_a}:\PP_{j_a}\lra\Gr ({F}^+_{d'}\oplus c^*{F}^-_{d'})$$
which is birational onto its image and such that ${\xi}^{j_a} =\alpha_{j_a}^* {\zeta}$ (respecting the decompositions into $P$ and $R$ components).
We will show in \S \ref{sec:Gamma des pf} that  the image is a component of the limit fiber ${\Gamma}_\infty$ which we
denote by ${\Gamma}_{\infty, j_a}$ and which has multiplicity $j_a$ in the fiber.

\subsubsection{Description of $\alpha _{j_A}: \PP_{j_A} \to \Gr$ and the vector bundle ${\tt F}^{j_a}$ on $D_a$}
For general $A$ the above analysis extends immediately, as the various rational tails may be treated independently. Specifically, this
means that we now consider a collection $j_A:=\{j_a| a\in A\}$ of positive integers and define

$$ P^{+, j_A+1}_{tail, d_a} : = \boxplus_{a\in A} \pi_*  \left( \sP^{+}_{tail, d_a} (-(j_a+1)p_a^{tail})
\right) , $$ 
$$ R^{+, j_A+1}_{tail, d_a} : = \boxplus_{a\in A} \pi_*  \left( \sR^{+}_{tail, d_a} (- (j_a+1)p_a^{tail})
\right) $$ on $\prod _{a \in A} (Q^+_{0, a} (\PP (V\to \CC ^N), d_a ) \ti \PP (\CC ^N))$ 
and 
$$  P^{-, j_A-1}_{d'_0} : = \pi_* (   \sP^{-}_{ d'_0}  (\sum_{a\in A} (j_a-1) p_a ))  \cap   {b}_A ^* {P}^{-}_d , $$
$$  R^{-, j_A-1}_{d'_0} : = \pi_* (   \sR ^{-}_{ d'_0} (\sum_{a\in A} (j_a-1) p_a ))  \cap   {b}_A ^* {R}^{-}_d $$
on $U_{k+A, d'_0}^{-}$.
Further, we put
\begin{equation*}        
{F}^{+, j_A+1}_{tail, d_a}:=P^{+, j_A+1}_{tail, d_a}\oplus R^{+, j_A+1}_{tail, d_a},\;
{F}^{-, j_A-1}_{d'_0}:= P^{-, j_A-1}_{d'_0}\oplus R^{-, j_A-1}_{d'_0} .
\end{equation*}
Setting 
\begin{equation*} 
\PP_{j_A}:=\prod_{a\in A} \PP_{j_a}|_{{D}_A},\end{equation*}
where the product is fiber product over ${D}_A$,
we have the projection $\pi_{\PP}: \PP_{j_A}\ra {D}_A$ and extensions
\begin{align}\label{hat extension P}
&0\ra \pi_{\PP}^*({\pr}_{A} ^*P^{+, j_A+1}_{tail, d_a}\oplus {\mpr}_{A} ^*{c}_{A}^*P^{-, j_A-1}_{d'_0})\ra
{\xi}_P^{j_A}\ra \boxplus_{a\in A}(\cO_{\PP_{j_a}}(-1)\ot \pi_{\PP}^*{\tt P}^{j_a} )  \ra 0,\\
\label{hat extension R}
&0\ra \pi_{\PP}^*({\pr}_{A} ^*R^{+, j_A+1}_{tail, d_a}\oplus {\mpr}_{A} ^*{c}_{A}^*R^{-, j_A-1}_{d'_0})\ra
{\xi}_R^{j_A}\ra  \boxplus_{a\in A}(\cO_{\PP_{j_a}}(-1)\ot \pi_{\PP}^* {\tt R}^{j_a}) \ra 0,\\
\label{hat extension E}
&0\ra  \pi_{\PP}^*({\pr}_{A} ^*{F}^{+, j_A+1}_{tail, d_a} \oplus  {\mpr}_{A} ^*{c}_{A}^*{F}^{-, j_A-1}_{d'_0} )\ra {\xi}^{j_A}\ra 
\boxplus_{a\in A}(\cO_{\PP_{j_a}}(-1)\ot \pi_{\PP}^* {\tt F}^{j_a} )  \ra 0,
\end{align}
with
\begin{align}\label{eqn:def tF}
 {\xi}^{j_A}:={\xi}_P^{j_A}\oplus {\xi}_R^{j_A},  \ \ {\tt F}^{j_a} := {\tt P}^{j_a} \oplus  {\tt R}^{j_a}   .
\end{align}
As before, this gives a morphism $\alpha_{j_A}:\PP_{j_A}\ra\Gr$
such that ${\xi}^{j_A} =\alpha_{j_A}^* {\zeta}$. 
We will show in \S \ref{sec:Gamma des pf} that the image of $\alpha_{j_A}$, denoted $\Gamma_{\infty,j_A}$, is a component of
the limit fiber, with multiplicity $m_{j_A}:=\prod_{a\in A}j_a$.

\subsubsection{Proof of Theorem \ref{thm:Gamma des}}\label{sec:Gamma des pf}
The description of the components $\Gamma_{\infty, j_A}$ of $\Gamma _{\infty}$ supported over $D_{A}$, 
with their multiplicities, as well as the fact that they exhaust the special fiber, all follow from writing explicitly the map $\Phi$ in local coordinates in an analytic (or \'etale)  neighborhood of a general point $p$ of the boundary stratum $ D_{A}$. An explicit proof is as follows.

Choose an \'etale open neighborhood $\tt U$ of $U^+_{d'}$ such that $p$ is a closed point in the scheme ${\tt U}$.
Let $\hat{\cO}_p$ be the completion of
$\cO_{{\tt U}, p}$ and let $C$ be the fiber curve of $\pi$ over $p$. The curve $C$ has exactly $|A|$-many nodal points $q$.
Let $C_{tail, q}$ be the rational tail component of $C$ which meets $q$
and let $C_{main}$ be the remained component of $C$ so that $C = \cup_{q} C_{tail, q} \cup C_{main}$. We may express the completion $\hat{\cO}_q$ at the node as
$$\hat{\cO}_{q} \cong \hat{\cO}_{p} [[ x_q, y_q ]] / (x_qy_q - t_q) $$ with local defining equations $x_q\in \hat{\cO}_q$, 
$t_q \in \hat{\cO_p}$ of the divisors $\tilde{D}_a$, $D_{a}$ 
respectively.

Consider a commuting diagram of natural $\hat{\cO}_p$-module homomorphisms
\[ \xymatrix{   (\pi _* \left(\sP^+_{d'} \oplus \bigoplus _{i} \sR ^+_{i, d'} \right) )_p \ot \hat{\cO}_{p} \ar@{^{(}->}[r]_{\phi_1} \ar[d]_{\Phi _p \ot \mathrm{id}}^{=: \Phi _{\hat{p}}}  
                                     &   \oplus _q \left(\sP^+_{d'} \oplus \bigoplus _{i} \sR ^+_{i, d'} \right)_{q}  \ot \hat{\cO}_{q}    \ar[d]^{\oplus _q \Psi_q}   \\
  (\pi _*  \left( \sP^+_{d'} (d_a\tilde{D}_a)  \oplus \bigoplus _{i}  \sR^+_{i, d'} (l_id_a \tilde{D}_a  ) \right) )_p \ot \hat{\cO}_{p} \ar@{^{(}->}[r]_{\phi_2} 
                    &  \oplus _q \left(\sP^+_{d'} (d_a\tilde{D}_a) \oplus \bigoplus _{i} \sR ^+_{i, d'} (l_id_a\tilde{D}_a) \right)_{q} \ot \hat{\cO}_{q}  
                    } \]
where $\sP^+_{d'}:= \sL '_+\ot V \ot Q$, $\sR^+_{i, d'} := \oplus _{j=1}^N (\sL '_+)^{l_i} $ as in \eqref{twisting},
the horizontal maps $\phi_i$ are the restriction maps, and $\Psi _q$ are the natural maps. 

Since the horizontal restriction maps $\phi_i$ are injections, we will use $\oplus _q \Psi_q$ to express $\Phi _{\hat{p}}$ explicitly. 
For this, let us choose
a $\hat{\cO}_q$-basis $\{ e ^q_{0,j} \}_{j=1}^{(N-1)\dim V}$ of $\sP^+_{d', q} \ot \hat{\cO}_q$ 
and a $\hat{\cO}_q$-basis $\{ e ^q_{i, j} \}_{i=1, j=1}^{r, N}$ of $\oplus _i \sR^+_{i, q} \ot \hat{\cO}_q$. With respect to this basis, we have also a basis 
$\{ e ^q_{0,j} \ot x_q^{-d_a} \}_{j=1}^{(N-1)\dim V}$ of $\sP^+_{d', q}  \ot \hat{\cO}_q (d_a\tilde{D}_a) \cong \sP^+_{d', q} (d_a\tilde{D}_a) \ot \hat{\cO}_q$ 
and a basis $\{ e ^q_{i, j} \ot x_q^{-l_id_a}  \}_{i=1, j=1}^{r, N}$ of 
$\oplus _i \sR^+_{i, q}   \ot \hat{\cO}_q (l_id_a\tilde{D}_a) \cong \oplus _i \sR^+_{i, q} (l_id_a\tilde{D}_a)  \ot \hat{\cO}_q$.
With respect to these bases, the right vertical map $\Psi _q$ is the component-wise multiplication by $x_q^{d_a}, x_q^{l_1d_1}, ..., x_q^{l_rd_r}$.

Let $\mathbf{k}(p)$ be the residue field of $\cO_p$ and let $\bar{e}^q_{0, j}$, $\bar{e}^q_{i, j}$ be the restrictions in 
$(\sP^+_{d'} \oplus \bigoplus _{i} \sR ^+_{i, d'})_q \ot \hat{\cO}_q \ot \mathbf{k}(p)$ of $e^q_{0,j}$, $e^q_{i, j}$ respectively.
Choose also a $\mathbf{k}(p)$-basis $\mathcal{B}_{main}$ of $H^0 (C_{main}, (\sP^+_{d'} \oplus \bigoplus _{i} \sR ^+_{i, d'} )   |_{C_{main}} (-\sum_q q ) )$
by taking the union of some bases of $H^0 (C_{main}, \sP^+ _{d'}   |_{C_{main}} (-\sum_q q ) )$, $H^0 (C_{main},  \sR ^+_{i, d'}    |_{C_{main}} (-\sum_q q ) )$, $\forall i$.
Consider the following subset 
\begin{equation}\label{eqn:basis}
\{ \oplus _q s_q \} _{s\in \mathcal{B}_{main}} \cup \bigcup _q \{ \bar{e}^q_{0, j}, y_q \bar{e}^q_{0, j}, ..., y_q^{d_a}\bar{e}^q_{0, j}\}_{j=1}^{(N-1)\dim V} 
\cup \bigcup _q \{ \bar{e}^q_{i, j}, y_q \bar{e}^q_{i, j}, ..., y_q^{l_id_a} \bar{e}^q_{i, j}\} _{i=1,j=1}^{r, N}
\end{equation}  of  $\oplus _q (\sP^+_{d'} \oplus \bigoplus _{i} \sR ^+_{i, d'} )_q  \ot \hat{\cO}_q \ot \mathbf{k}(p)$.
Here $s_q$ denotes the stalk of $s$ at $q \in C_{main}$.
Note that \eqref{eqn:basis} is a $\mathbf{k}(p)$-basis of the subspace $H^0 (C,  (\sP^+_{d'} \oplus \bigoplus _{i} \sR ^+_{i, d'} ) |_{C} )$.
Extend this $\mathbf{k}(p)$-basis  \eqref{eqn:basis}
 to a basis of  $ (\pi _* (    \sP^+_{d'} \oplus \bigoplus _{i} \sR ^+_{i, d'}    ) )_p \ot \hat{\cO}_{p}$ as a $\hat{\cO}_p$-module, 
\begin{equation}\label{eqn:basis1} \{ \oplus _q \tilde{s}_q \} _{s\in \mathcal{B}_{main}} 
\cup \bigcup _q \{ e^q_{0,j},  y_q e^q_{0,j}, ..., y_q^{d_a} e^q_{0,j}\}_{j=1}^{(N-1)\dim V} \cup 
\bigcup _q \{ e^q_{i, j},  y_q e^q_{i, j}, ..., y_q^{l_id_a}e^q_{i, j}\} _{i=1,j=1}^{r, N} \end{equation}
where $\tilde{s} \in \pi_* (\sP ^+_{d'} \oplus \bigoplus _i \sR ^+_{i, d'}) \ot \hat{\cO}_p$ is an extension of $s$.

Let $l_0 = 1$ and let $l(s) = l_0$ for $s\in \mathcal{B}_{main}$ if $s$ comes from $\sP^+_{d'} |_{C_{main}} (-\sum_q q ) $, 
$l(s) = l_i$ if $s$ comes from $\sR^+_{i,d'} |_{C_{main}} (-\sum_q q ) $.
Choose also a basis of $ (\pi _* (  \sP^+_{d'} (d_a\tilde{D}_a)  \oplus \bigoplus _{i}  \sR^+_{i, d'} (l_id_a \tilde{D}_a )))_p \ot \hat{\cO}_{p}$ which is expressed 
via $\phi_2$ as 
\begin{eqnarray}\label{eqn:basis2}   & \{ \oplus _q \tilde{s}_q \}_{s\in \mathcal{B}_{main}}  \\ &   \cup \bigcup_q 
 \{x^{d_a}_q (e^q_{0,j} \ot x_q^{-d_a}),  x_q^{d_a-1} (e^q_{0,j} \ot x_q^{-d_a}), ...,  e^q_{0,j} \ot x_q^{-d_a}) \}_{j=1}^{(N-1)\dim V} \nonumber \\
&  \cup \bigcup_q
\{x_q^{l_id_a} (e^q_{i,j} \ot x_q^{-d_a}), x_q^{l_id_a -1} (e^q_{i,j} \ot x_q^{-l_id_a}), ...,  e^q_{i,j} \ot x_q^{-l_id_a})\}_{i=1, j=1}^{r, N} . \nonumber \end{eqnarray}
The map $\lambda \Phi _{\hat{p}}$  sends    
 \[ \oplus _q \tilde{s}_q \mapsto   \oplus_q \lambda  \tilde{s}_q, \text{ and }  y_q^k  e_{i, j}^q \mapsto \lambda t_q^k x_q^{l_id_a -k} (e_{i, j}^q \ot x_q^{-l_id_a}) , 
  \ i=0, 1, ..., r; k =0, 1, ..., l_id_a ; \forall j\]
so that with respect to the $\hat{\cO}_p$-bases \eqref{eqn:basis1} and \eqref{eqn:basis2}, $\lambda\Phi _{\hat{p}}$ is a diagonal matrix with
entries $\lambda$'s, $\lambda t_q^k$, $k=0, 1, ..., \mathrm{max}_{i=0, ..., r} \{ l_id_a \}$. 

Now according to the fate of $\lambda t_q^k$, $k=0, 1, ... $, as $\lambda \to \infty$ and $t_q\to 0$ $\forall q$, 
the cycle class $[\Gamma_{\infty}]$ can be easily identified yielding
the decomposition \eqref{special fiber 1} for each $A$. Namely, for the node $q$ corresponding to $a$, if $\lambda t_q^{j_a}$ goes to a nonzero number $w_a \in \CC$
for some $j_a$, then the limit of  $\text{graph} (\lambda \Phi )$ in the region is the  point 
$Point (j_a, w_a )_{a\in A}$ in $\Gr |_{p}$ corresponding to the direct sum of the following three subspaces $(i), (ii), (iii)$
\[ \begin{array}{lll}  (i) \, F^{+, j_A+1} _{tail, d_a} |_{\pr_A (p)} 
& =  \oplus _{a\in A} \oplus _{i, j} \lan y_q^{j_a+1} \bar{e}^q_{i, j}, ..., y_q^{l_id_a} \bar{e}^q_{i, j} \ran  & \subset F^+_{d'} |_p  ; \\
  (ii)\,  F^{-, j_A-1} _{d_0'} |_{\mpr_A\circ c_A (p)}   & = H^0(C_{main},   \sP^{-}_{d'}|_{C_{main}} (-\sum_{a} d_a q)   ) \oplus & \\
   & \   \oplus _{i\ge 1} H^0(C_{main},   \sR^{-}_{i, d'}|_{C_{main}} (-\sum_{a} l_id_a q)   ) \oplus 
                                              &   \\
             & \  \oplus _{a\in A} \oplus _{i, j}   \lan x_q^{l_i d_a} (\bar{e}^q_{i, j} \ot x_q^{-l_id_a}) , ..., x_q^{l_id_a - (j_a -1)} (\bar{e}^q_{i, j} \ot x_q^{-l_id_a})  \ran 
                                                  & \subset  c^*F^{-}_{d'}|_p ; \\
       (iii)\,    \oplus _{a\in A} \oplus _{i, j} y_q^{j_a} \bar{e}^q_{i, j} &  \oplus w_a x_q^{l_i d_a - j_a} (\bar{e}^q_{i, j} \ot x_q^{-l_id_a}) 
           \subset \pr_A^* F^{+, j_A}_{tail, d_a} |_{p} \oplus \mpr_A^* c_A^*F^{-, j_A}_{d_0'} |_{p} . &
\end{array} \]
It is clear that there is a natural correspondence between the irreducible components of $\Gamma _{\infty}$ and $Point (j_a, 1 )_{a\in A}$ $\forall j_A$. 
Denote by $\Gamma_{\infty, j_A}$ the component corresponding to $Point (j_a, 1 )_{a\in A}$.
 The intersection multiplicity of $\Gamma _{\infty} \cap \{ \lambda = \infty \}$ at $\Gamma_{\infty, j_A}$ 
  is $m_{j_A}:= \prod_{a\in A} j_a$ according to the equations $t_q^{j_a} =0$, $\forall a \in A$ in the open affine coordinate ring of $\Gr$ 
  around $Point (j_a, 1 )_{a\in A}$.

\subsubsection{Remark} 
Denoting by ${\bf e}$ the Euler class, \cite[Example 18.1.6]{Fu} gives
\begin{equation}\label{degeneration 1}
{\bf e}({F}^+_{d'})\cap [U^+_{d'}]-{\bf e}( c^*{F}^-_{d'})\cap [U^+_{d'}]=
\sum_{(A,j_A)} \frac{m_{j_A}}{|A|!}(\eta |_{\Gamma_{\infty,j_A}})_*({\bf e}(\zeta)\cap[\Gamma_{\infty,j_A}]).
\end{equation}
For $g=0$, when no twisting occurs, $U^{\pm}_{d'}$ reduces to $Q^\pm_{0,k}(\PP(V),d)$, while
$ F^{\pm}_{d'}=\pi_*(\oplus_{i=1}^r\sL_{\pm}^{l_i})$. After applying $c_*$, the left-hand side of \eqref{degeneration 1} becomes
precisely
$$c_*i_*[Q^{+}_{0,k}(X,d)]^{\mathrm{vir}}-i_*[Q^{-}_{0,k}(X,d)]^{\mathrm{vir}}.$$
On the other hand,
it is not too difficult to show\footnote{The argument is a considerably simplified version of the proof  of Theorem \ref{Degeneration2} in \S\ref{correcting class} below.} that 
the right-hand side can be written in the form 
$$\sum_A \frac{1}{|A|!}(b_A)_*(c_A)_*i_*\left(\prod_{a\in A} ev_a^*\mu_{d_a}(z)|_{z=-\psi_a} \cap [Q^{+}_{0,k+A}(X,d_0^A)]^{\mathrm{vir}}\right ) ,$$
for {\it some} polynomial Chow cohomology class $\mu_{d_a}(z)\in A^*(X)_\QQ [z]$. 
Combined with the identification of $\mu_{d_a}$ in \S\ref{identification} below, this proves 
for $X$ 
the weaker equality \eqref{geometric} in Conjecture \ref{MainConj} in genus zero.

\subsection{A refinement of the graph construction} 
The equality
\eqref{degeneration 1} may be viewed as a degeneration formula for the top Chern class of 
the vector bundle ${F}^+_{d'}$ on $U^{\pm}_{d'}$.
As a main step in our proof of Theorem \ref{Main}, we establish in this subsection a refined degeneration formula
which relates the Gysin pull-backs 
$0_{E_d^\ke}^! ([\mathbf{C}_{Q^\ke_{g, k}(X, d) / U^{\ke}_{d'}}])$ of the normal cones from Corollary 
\ref{concrete virtual class}.

\subsubsection{Deformation of the embedding \eqref{embedding diagram}}\label{sec:bar sigma}
The map $\Phi$ fits in the following commuting diagram
\[ \xymatrix{ {F}^+_{d'}   \ar[r]^{ \Phi}  
&  c^*{F}^-_{d'} \ar[r]   & {F}^-_{d'}  \\ 
                          U^+_{d'} \ar[r]_{=}  \ar[u]^{\sigma^{+}}         &    
 U^+_{d'} \ar[r]_{ c} \ar[u]_{ c^*(\sigma^{-})}    &  U^{- }_{d'}  \ar[u]_{\sigma^{- }}  } \] 
with $\sigma^{\pm}$ the canonical sections \eqref{sigma epsilon}. 
Recall that the zero locus of $\sigma^{\pm}$, call it $Y^{\pm}$, is identified with $Q^{\pm}_{g,k}(X,d)$. 
Denote by
$Z=Z_{g,k,d}$ the zero locus of $ c^*(\sigma^{-})=\Phi\circ \sigma^+$; in other words, $Z= c^{-1} (Q^-_{g,k}(X,d))$.
Observe that there is a closed embedding $Y^+\hookrightarrow Z$. 
\begin{Rmk}\label{compatible contractions} 
If we restrict ${c}$ further to $Y^+\subset Z$, the resulting map coincides with the contraction $c:Y^+\ra Y^-$ induced from
the natural embedding $X\subset\PP(V)$ and the contraction $c:Q^+_{g,k}(\PP(V),d)\ra Q^-_{g,k}(\PP(V),d)$. This follows from the fact that the
twisting line bundle $\sM$ is trivial on the rational tails.
\end{Rmk}

It turns out that it is better to consider the deformation of $Z$ induced by the family $\Gamma\ra\PP^1$.
To this end, consider the {\it universal quotient bundle} $\Upsilon$ on $\Gr$, so that
$$0\ra \zeta\ra \eta^*({F}^+_{d'}\oplus c^* {F}^-_{d'})\ra\Upsilon\ra 0.$$
is exact. 
As before, we also consider the universal quotient bundles $\Upsilon_P$ on $\Gr_P$ and $\Upsilon_R$ on $\Gr_R$.
{\em We will use the same notations for the induced vector bundles on $\Gamma$.}

 The section $\eta^*(\sigma^+, c^*\sigma^-)$ of $\eta^*({F}^+_{d'}\oplus c^*{F}^-_{d'})$ induces a section 
  \begin{equation*} 
 \overline {\sigma}\in H^0(\Gamma, \Upsilon)
 \end{equation*}
 of $\Upsilon$ on $ \Gamma$, via composition with the projection.
 
 Let
\begin{equation*} 
\Gamma ^0 := \overline{\sigma}^{-1}(0)\subset \Gamma \subset \Gr\times\PP^1.
\end{equation*}
be the zero locus of $\overline{\sigma}$.

As before, let $\Gamma^0_\lambda$ denote the fiber of $\Gamma^0$ over $\lambda\in\PP^1$. For $\lambda\neq 1, \infty$, under the isomorphism
$v_\lambda: U^+_{d'}\times \{ \lambda \} \ra \Gamma_\lambda$, the section $\overline{\sigma}$ corresponds to the section 
$(1-\lambda) c^*\sigma^-$ of $F^-_{d'}$. Hence, for $\lambda\notin \{1,\infty\}$, we get that $\Gamma^0_\lambda$ 
is isomorphic to $Z$. 

The fiber over $1\in\AAA^1$ is the entire $U^+_{d'}$, so from now on we will consider the families $\Gamma$ and $\Gamma^{0}$ {\it only over}
$\PP^1\setminus\{1\}$ (but will keep the same notation).

The fiber over $\infty\in\PP^1$ decomposes in the Chow group as
$$[\Gamma^0_{ \infty}]=[\Gamma^0_{ \infty, dist}] + \sum_{(A, j_A)} m_{j_A}[\Gamma^0_{\infty,j_A}],$$
with $\Gamma^0 _{\infty, dist} :=\Gamma_{\infty, dist}\times_{\Gamma} \Gamma ^0$ and 
$\Gamma ^0_{\infty,j_A} :=\Gamma_{\infty,j_A}\times_{\Gamma} \Gamma ^0$.

Note that on $\Gamma_{\infty, dist}= U^+_{d'}$ the quotient bundle $\Upsilon$ is equal to $\eta^*{F}^+_{d'}\oplus\{0\}$ and $\overline{\sigma}=(\sigma^+,0)$, 
hence $\Gamma^0_{\infty, dist}$ is identified with $Q^+_{g,k}(X,d)$, embedded as in \eqref{embedding diagram}.

\subsubsection{Deformation of the obstruction theory}\label{def of obs}

The normal cone $\bC_{\Gamma^0 /\Gamma}$ is a subcone of $\Upsilon|_{\Gamma^0}$.
 We claim that, possibly after a birational modification of the fiber $\Gamma_{\infty}$, 
 it actually sits inside a subbundle $\Upsilon^0$ of the ``correct" rank.

Recall the twisting line bundle $\sM$  on the universal curve $\fC^{\pm}_{g, k, d'}$ of $U^{\pm}_{d'}$ introduced in the beginning 
of \S \ref{U-embedding} and recall $s_j$ the sections $\tilde{ft}^*_{\pm}  \tau_j$ of $\sM$ where
$\tilde{ft}_{\pm} : \fC^{\pm}_{g, k, d'} \to \fC_{g, k}$ is the stabilization map; see \S \ref{twisting lb} for the definition of $\tau _j$. 
Here $\fC_{g, k}$ is the universal curve over $\overline{M}_{g, k}$.

On  the universal curve $\fC^+_{g,k,d'}$ over $U^+_{d'}$, 
there is a vector bundle monomorphism
\[ \sP^+_{d'}  \hookrightarrow  \sP ^{+ }_{d', big} := \sL '_+\ot \sM \ot V \ot \CC ^{\binom{N}{2}} \]
induced from the homomorphism 
\begin{align*} \oplus _j \sL ' _+ \ot V   \ra  \sP ^{+ }_{d', big} ,  \quad
                         (v_j)_{j=1}^N   \mapsto  \oplus _{j_1> j_2} (s_{j_1} v_{j_2} - s_{j_2} v_{j_1} ) . \end{align*}
Similarly there are vector bundle monomorphisms
\begin{align*} & \sP ^{-}_{d'} \hookrightarrow\sP ^{-}_{d' , big} := \sL '_-\ot \sM \ot V \ot \CC ^{\binom{N}{2}}  ; \\
        &  \sQ ^{\pm}_{d'} \hookrightarrow \sQ ^{\pm}_{d', big} := \oplus _i (\sL' _{\pm}\ot \sM )^{l_i} \ot \CC ^{\binom{N}{2}}  .  \end{align*}

We replace the stack ${\Gamma}$  by the closed substack ${\Gamma}^{new}$ of the product $\Gr^{new}\times\PP^1$ defined via the MacPherson graph construction, where 
$\Gr^{new}$ is now the fibered product over $U^+_{d'}$
of the various Grassmann bundles:
\begin{align}\label{new Gr}
\Gr^{new}=& \mathbf{Gr} ( \pi _* \sP^+_{d'} \oplus  {c}^*\pi _* \sP ^-_{d'}) \times_{U^+_{d'}} \mathbf{Gr} (\pi_*\sR ^+_{d'}
 \oplus  {c}^*\pi_* \sR ^-_{d'})
\\ \nonumber \times_{U^+_{d'}} &
\mathbf{Gr} ( \oplus _j \pi _* \sL ' _{+}\ot V  \bigoplus \oplus _j {c}^* \pi _* \sL '_-  \ot V  )  
\times_{U^+_{d'}} \mathbf{Gr}  (\pi_*\sQ^{+}_{d'} \oplus  {c}^*\pi _*  \sQ^{-}_{d'} ) \\ \nonumber
\times_{U^+_{d'}} & \mathbf{Gr} ( \pi _* \sP^+_{d', big} \oplus {c}^* \pi _* \sP^-_{d', big}  )\times_{U^+_{d'}} 
\mathbf{Gr} (\pi _* {\sQ ^+_{d', big}}\oplus {c}^* \pi _* \sQ ^-_{d', big})  .
\end{align}
The projection onto the first two factors induces a birational morphism $p_{12}:{\Gamma}^{new}\ra {\Gamma}$, which is an isomorphism outside $\infty\in\PP^1$.

Denote by  $\Upsilon_{\oplus_j \sL ' \ot V }, \Upsilon _{\sP_{big}} ,  \Upsilon_{\sR} ,
 \Upsilon_{\sQ_{big}}, ... $ 
 the universal quotient bundles on ${\Gamma}^{new}\subset\mathbf{Gr}^{new}\times \PP ^1$ obtained via pull-back from the third, the fifth, the second, the sixth,
... factor of $\Gr^{new}$ respectively. Similarly, 
denote by $\zeta_{\oplus _j \sL ' \ot V }, ...$, the universal subbundles on ${\Gamma}^{new}$.
Recall that $\Upsilon _{\sP }$ and $\Upsilon_{\sR }$ come with the sections $\overline{\sigma}_P$ and $\overline{\sigma}_R$, the components of the section $\overline{\sigma}$ of 
$${\Up}={\Upsilon} _{\sP }\oplus {\Upsilon} _{\sR}$$ (see \S \ref{sec:bar sigma}).  We set 
$$\Gamma^{new, 0} = \overline{\sigma} ^{-1}(0).$$

As in the case when we had only the fibered product of the first two relative Grassmannians, 
for each $j_A$ 
there is a natural morphism $$\alpha_{j_A}^{new}:\PP _{j_A}\lra\mathbf{Gr}^{new}\times\{\infty\},$$ which has generic degree $|A|!$ to the image and such that the relation
\eqref{special fiber 1} still holds for the new special fiber (in other words, the birational modification $p_{12}:{\Gamma}^{new}\ra {\Gamma}$ does {\it not} introduce additional components over $\infty\in\PP^1$). These morphisms are obtained by constructing extensions analogous to \eqref{hat extension P}
and \eqref{hat extension R} 
for the remaining four factors in \eqref{new Gr}. We have $\alpha_{j_A}=p_{12}\circ\alpha_{j_A}^{new}$.
Our proof of Theorem \ref{Main} will eventually reduce to intersection-theoretic computations performed after transfering everything to the $\PP _{j_A}$'s.
Hence it is harmless to drop from now on the superscript ``new" from the notations for $\Gr$, $\Gamma$, $\Gamma$ etc.

 We are now ready to construct the required vector bundle $\Up$.
 Define two homomorphisms 
\begin{align*} d \varphi _{\pm, big} : \sP^{\pm}_{d', big}  \ra \sQ^{\pm}_{d', big},  \quad
                                         (v_{j_1,j_2}) \mapsto \oplus _{i}\oplus _{j_1>j_2} \nabla \varphi _i (s_{j_1} u_{j_2}' )\cdot v_{j_1,j_2}  .
\end{align*}
where $\oplus _j u'_{ j}$ is the universal sections of $\oplus _j \pi_* \sL'_{\pm} \ot V$ as in \eqref{eqn:def u'}.
 
On ${\Gamma}$, there is a natural diagram
\begin{align}\label{for_corr} \xymatrix{ \Upsilon_{\oplus_j \sL ' \ot V} \ar[r]  \ar[d] &  \Upsilon _{\sP_{big} } \ar[d]^{\overline{\pi_* d \varphi_{big} }} \\
               \Upsilon _{\sR } \ar[r]  &  \Upsilon _{\sQ_{big} }  }\end{align}
               which is not necessarily commutative.  
               Here $\overline{\pi_* d \varphi_{big} }$ is the homomorphism induced from $d\varphi _{\pm , big}$ via push-forward. 
               The remaining three arrows are all constructed by the same procedure. For example, the top horizontal homomorphism is obtained as follows. The composition
               of natural maps 
               $${\zeta}_{\oplus _j \sL ' \ot V } \ra \eta ^* (\pi _* \sP^+_{d'}  \oplus c^*\pi _* \sP^-_{d'})
               \ra \eta ^* (\pi _* \sP^+_{d', big} \oplus c^*\pi _* \sP^{-}_{d', big} ) \ra {\Upsilon}_{\sP _{big}}$$ vanishes 
               on ${\Gamma}\setminus {\Gamma}_{\infty}$ and hence vanishes on the closure ${\Gamma}$.

Let $\tilde{\eta}$ denote the composition 
of natural maps $\Gamma \to \Gr \ti (\PP ^1\setminus \{1\}) \to U^+_{d'}$.

\begin{Lemma}\label{lem:barsigma_P sigma_P} The following hold.
\begin{enumerate}
\item The zero locus of the $P$-component $\overline{\sigma} _P$ of $\overline{\sigma}$ is contained in the zero locus of $\tilde{\eta}^* \sigma ^{-}_{P}$ (see
\eqref{sigma epsilon} for the definition of $\sigma ^{-}_{P}$).

\item $(\sigma ^+_P)^{-1} (0) = Q^+_{g, k}(\PP (V), d) = (c^* \sigma ^-_P)^{-1} (0)$

\item  The diagram \eqref{for_corr} becomes commutative
when it is restricted to $\overline{\sigma} _P^{-1}(0)$.
\end{enumerate}
\end{Lemma}

\begin{proof}
(1) Consider the homomorphism of locally free sheaves 
\[ (\eta ^* P^+_{d'} \oplus \eta ^*c^* P^-_{d'} ) \boxplus \cO_{\PP ^1\setminus \{1\}} \to \eta ^*c^* P^-_{d'} \boxplus (\cO_{\PP ^1}(1))|_{\PP^1 \setminus \{1\}} , \ 
(v^+, v^-) \mapsto \lambda _0 \Phi (v^+) - \lambda _1 v^- ,\]
where $\lambda _0, \lambda _1$ denote homogeneous coordinates of $\PP ^1$.
Since $\zeta |_{\Gamma}$ is contained in the kernel of the above homomorphism, there is a map 
$\Upsilon _P \to  \eta ^*c^* P^-_{d'} \boxplus (\cO_{\PP ^1}(1) ) |_{\PP^1 \setminus \{1\}} $, under which the section $\overline{\sigma}_P$ goes to
$(\lambda _0 - \lambda _1) c^*\sigma ^-_P$.
Therefore the zero locus of $\overline{\sigma} _P$ is contained in the zero locus of $\tilde{\eta}^* \sigma ^-_P$.

(2) The first equality is clear. The second equality is the claim $$Q^+_{g, k}(\PP (V), d)  = c^{-1} (Q^{-}_{g, k}(\PP (V), d)) .$$
The claim is obvious since for a $T$-family of $\ke _+$-stable quasimaps to $\PP(V \ot \CC ^N)$, it is a $T$-family of 
$\ke _+$-stable quasimaps to $\PP(V)$ if and only if the family restricted to every geometric point of the test scheme $T$ is a $\ke _+$-stable quasimaps to $\PP(V)$.

(3) The diagram \eqref{for_corr} is by definition induced, by the pullback $\tilde{\eta}^*$,  
from the diagram of homomorphisms of locally free sheaves on $U^+_{d'}$ 
\begin{align}\label{for_corr orig} \xymatrix{ \pi_* \oplus_j \sL ' _{+} \ot V \oplus c^* \pi_* (\oplus_j \sL '_{-} \ot V)  \ar[r]  \ar[d] &  \pi_* \sP ^+_{big} 
  \oplus c^* \pi_* \sP ^-_{big}  \ar[d]^{\pi_* d \varphi_{+, big} \oplus c^*\pi_* d \varphi_{-, big} } \\
              \pi_* \sR ^+_{d'}  \oplus c^*\pi_* \sR ^-_{d'} \ar[r]  &  \pi_* \sQ ^+_{big}  \oplus  c^* \pi_* \sQ^-_{big}  . } \end{align}
The diagram \eqref{for_corr orig} is commutative on the zero locus $Q^+_{g, k}(\PP(V), d)$ of the section $\sigma ^+_P$ since
the difference of the clockwise path and the counterclockwise path in each $\pm$-component
\begin{eqnarray*}  & & \oplus _i \left( \nabla \varphi_i (s_{j_1} u_{j_2}') \cdot (s_{j_1} v_{j_2} - s_{j_2} v_{j_1} )-  
(s_{j_1}^{l_i} \nabla \varphi _i (u_{j_2}') \cdot  v_{j_2} -  s_{j_2}^{l_i} \nabla \varphi _i (u_{j_1}')  \cdot  v_{j_1} ) \right) \\
& = &  \oplus _i \left(- \nabla \varphi_i (s_{j_1} u_{j_2}') \cdot s_{j_2} v_{j_1} +  \nabla \varphi _i (s_{j_2} u_{j_1}')   \cdot s_{j_2} v_{j_1} \right) \end{eqnarray*}
vanishes for the universal section $(u'_j)_j$ of $\oplus _j \sL '_{\pm} \ot V$ with the vanishing condition $s_{j_1}u_{j_2}' - s_{j_2}u_{j_1}' = 0$. 
Hence it is enough to show that the zero locus of $\overline{\sigma} _P$ contained in $\Gamma \ti _{U^+_{d'}} (\sigma ^+_P)^{-1} (0)$.
This follows from (1) and (2) above. 
\end{proof}

 In particular, the diagram \eqref{for_corr} commutes when restricted to $\Gamma^0$.
Since the horizontal maps factor through ${\Upsilon}_{{\sP}}$ and ${\Upsilon}_{{\sQ}}$, it follows that on $\Gamma^0$ we have 
the commuting diagram
\[ \xymatrix{ \Upsilon_{\oplus _j \sL ' \ot V } |_{\Gamma ^0 } \ar[r]  \ar[d] &  \Upsilon _{\sP} |_{\Gamma ^0 } \ar[d]^-{f_{\Upsilon}} \\
               \Upsilon _{\sR} |_{\Gamma ^0 } \ar[r]^-{\alpha_{1, \Upsilon}}  &  \Upsilon _{\sQ } |_{\Gamma ^0 } ,  } \]
where $f_{\Up} =\overline{\pi_* d \varphi_{big} }|_{\Upsilon _{\sP}|_{\Gamma ^0}}$.
The map of vector bundles
$$\gamma
: (\Up={\Upsilon} _{\sP}   \oplus  {\Upsilon} _{\sR})|_{\Gamma ^0 }  \ra  {\Upsilon} _{\sQ} |_{\Gamma ^0 } ,\;\;\; \gamma(x,y)=f_{\Upsilon}(x)-\alpha_{1, \Upsilon} (y) $$
is surjective since it is so at each closed point of $\Gamma$ (this needs to be checked at points on the special fiber $\Gamma_\infty$,  where it follows by pulling-back to the appropriate $\PP_{j_A}$ and using the description of the three universal quotient bundles as extensions, as in e.g. \eqref{univ quotient} below).
Define the required vector bundle on $\Gamma$ to be $$\Up ^0:=\ker \gamma .$$

\begin{Lemma}\label{subcone} The 
normal cone $\bC_{\Gamma^0/{\Gamma}}$ is a subcone of $\Upsilon ^0$.
\end{Lemma}
\begin{proof} Let $\sI_{\Gamma ^0}$ denote the defining ideal sheaf of the closed substack $\Gamma ^0$ of ${\Gamma}$.
We will check that the induced homomorphism $ ({\Upsilon}_{{\sQ}})|_{\Gamma ^0}^\vee \ra \sI _{\Gamma^0} /\sI _{\Gamma^0}^2 $ 
is identically zero.
For this consider the commuting diagram
\[ \xymatrix{  ({\Upsilon}_{{\sQ _{big}}})^\vee  \ar[r] \ar[d] & \Up ^{\vee} \ar[d] \ar[r] &  \sI _{\Gamma^0} \ar@{^{(}->}[d] \\
                    \tilde{\eta} ^*(  \pi_*\sQ^{+}_{d', big} \oplus c^*\pi_*\sQ^{-}_{d', big} )^{\vee}  \ar[r] &  \tilde{\eta}^* (F^+_{d'}\oplus c^*F^{-}_{d'} )^{\vee}   \ar[r]      & \cO_{\Gamma} ,  } \]
                    where $\tilde{\eta}$ denotes the composition $\Gamma \to \Gr \ti (\PP ^1\setminus \{1\}) \to U^+_{d'}$.
By the above commuting diagram and the surjection 
$ ({\Upsilon}_{{\sQ _{big}}})|_{\Gamma ^0}^\vee \twoheadrightarrow  ({\Upsilon}_{{\sQ}})|_{\Gamma ^0}^\vee $, 
 it is enough to show that the composition of the bottom arrows lands in $\sI _{\Gamma^0}^2$.
On the other hand $\tilde{\eta}^*\mathrm{Im}(\sigma _{P}^{+ \vee}) \subset \sI _{\Gamma^0}$ by Lemma \ref{lem:barsigma_P sigma_P} (1).
Here we view the dual $\sigma _P^{+ \vee}$ of $\sigma _P^{+}$ as the cosection $\sigma_P^{+ \vee} :     (P_{d'}^{+})^{\vee} \to   \cO_{U^+_{d'}}$.
Hence by Lemma \ref{lem:barsigma_P sigma_P} (2) 
 it is enough to check that the composition $comp$ of $(\pi_*\sQ^{\pm}_{d', big})^{\vee} \to (F^{\pm}_{d'} )^{\vee} \to \cO_{U^{\pm}_{d'}}$ 
lands in $(\mathrm{Im} \sigma _P^{\pm \vee} )^2$.
This is easy to check as follows. Recalling the definition of $\sigma ^{\pm}_R$, $\sigma ^{\pm}_P$ in \eqref{sigma epsilon}, note that, for $\delta \in  (\pi_*\sQ^{\pm}_{d', big})^{\vee}$
\begin{eqnarray*}  comp (\delta) 
& = & \lan \delta, \oplus_i \oplus _{j_1 > j_2}  \nabla \varphi_i (s_{j_1}  u'_{j_2} ) \cdot (s_{j_1} u'_{j_2} - s_{j_2} u'_{j_1} )  
-   (\varphi _i ( s_{j_1} u'_{j_2})  -  \varphi _i ( s_{j_2} u'_{j_1}))   \ran  \\
& \in & (\mathrm{Im} \sigma _P^{\pm \vee} )^2 .
\end{eqnarray*}
Here the last line is due to the Taylor expansion of the last term $\varphi _i ( s_{j_2} u'_{j_1})$ in the first line:
$$\varphi _i ( s_{j_2} u'_{j_1}) =  \varphi _i ( s_{j_1} u'_{j_2}) + \nabla \varphi _i ( s_{j_1} u'_{j_2}) \cdot (s_{j_2} u'_{j_1} - s_{j_1} u'_{j_2} ) 
$$ modulo the square of the ideal $\mathrm{Im}\sigma _P^{\pm \vee}$ generated by $s_{j_2} u'_{j_1} - s_{j_1} u'_{j_2} $.
\end{proof}

By construction, on the fiber $\Gamma^0_0 := \Gamma^0 \ti_{\Gamma} \Gamma _0$ we have 
$$\Up ^0 |_{\Gamma^0_0}=c^*E_d^-,$$ 
while on the distinguished component $\Gamma^0_{\infty,dist}:=  \Gamma^0 \ti _{\Gamma} \Gamma _{\infty, dist}$ over $\lambda=\infty$,
$$\Up ^0 |_{\Gamma^0_{\infty,dist}}=E_d^+,$$
with $E_d^\pm$ as defined in \eqref{E1 term}.
 
\subsubsection{Refined degeneration formula}

Consider the diagram, whose squares are all cartesian,

\[ \xymatrix{        \lambda \ar[d]_{\lambda}     &  \ar[l]  \Gr_Z  \ar[d]  & \ar[l]_{\ \iota _\lambda}   \Gamma^0_\lambda \ar[r]\ar[d] & 
\bC_{\Gamma ^0/ {\Gamma}} |_\lambda \ar[d]         \ar[r]  & \lambda   \ar[d]^{\lambda}  \\
     \PP ^1 \setminus \{1\} &  \ar[l]   \Gr_Z\times ( \PP ^1  \setminus \{1\})    & \ar[l]^-{\iota}  \Gamma^0  \ar[r]  \ar[d] & 
     \bC_{\Gamma ^0 / {\Gamma}}  \ar[d] \ar[r] & \PP ^1  \setminus \{1\}  \\
                            &  &  \Gamma ^0  \ar[r]^-{0}  &   \Upsilon ^0 & } \]
 where $\Gr_Z$ denotes the relative Grassmannian $\Gr$ restricted to $Z$, with projection $\eta|_Z:\Gr_Z\ra Z$.

 \begin{Lemma}\label{degeneration for cones} In $A_*(Z)_{\QQ}$ we have the equality
 \begin{align}\label{refinedGraph}
 &(\eta |_{Z}) _* (\iota _0)_* 0^!_{\Up^0 |_{\Gamma ^0_0}}([\bC_{\Gamma^0_0 / \Gamma_0}  ]) - 
 (\eta |_{Z}) _* (\iota _{\infty})_*( 0^!_{\Up^0 |_{\Gamma ^0_{\infty, dist}}}([ \bC_{dist} ])) =\\ 
\nonumber & \sum _{(A,j_A)} m_{j_A}
 (\eta |_{Z}) _*  (\iota _{\infty})_* ( 0^!_{\Up^0 |_{\Gamma ^0_{\infty, j_A}}}([\bC_{ j_A}] )) , \end{align}
where $\bC_{dist}$ is the normal cone $\bC_{\Gamma ^0_{\infty, dist}/\Gamma_{\infty, dist}}$ and
$\bC_{j_A}$ is the normal cone $\bC_{\Gamma ^0_{\infty, j_A} / \Gamma _{\infty,j_A}}$. 

 \end{Lemma}
 
 \begin{proof}
By Theorem 6.2.(a) and Theorem 6.4 in \cite{Fu} (as extended to DM-stacks in \cite{Vistoli}), we have
\begin{align}\label{LHS}   \lambda^! \iota _* 0^! [\bC_ {\Gamma^0 / \Gamma}]  = (\iota _\lambda) _* \lambda^! 0^! [\bC_{\Gamma^0 / \Gamma}  ]   
= (\iota _\lambda)_* 0^! \lambda^! [\bC_{\Gamma^0 / \Gamma} ] .     \end{align}
When $\lambda=0$, $$0^! \lambda^! [\bC_{\Gamma^0 / \Gamma} ] = 0^!_{\Up^0 |_{\Gamma^0_0}}([\bC_{\Gamma^0_0 / \Gamma_0} ] ).$$ 
By Lemma \ref{limiting} below, when $\lambda=\infty$,  
$$0^! \lambda^! [\bC_{\Gamma^0 / \Gamma }  ] 
= 0^!_{\Up^0 |_{\Gamma^0_{\infty, dist}}}([\bC_{dist} ]) +  \sum _{(A,j_A)} m_{j_A}0^!_{\Up^0 |_{\Gamma^0_{\infty,j_A}}} ([\bC_{ j_A}]) .$$
The first term in \eqref{LHS} is independent of $\lambda$. Hence
\begin{align*}
(\iota _0)_* 0^!_{\Up^0 |_{\Gamma^0_0}}([\bC_{\Gamma^0_0 / \Gamma_0} ]) =& (\iota _{\infty})_*( 0^!_{\Up^0 |_{\Gamma^0_{\infty, dist}}}([\bC_{dist} ]))\\ 
&+\sum _{(A,j_A)} m_{j_A} (\iota _{\infty})_* (0^!_{\Up^0 |_{\Gamma^0_{\infty, j_A}}}([\bC_{ j_A}] )) 
\end{align*}
in $A_*(\Gr_Z)_{\QQ}.$
Pushing forward to $Z$ we get \eqref{refinedGraph}.
\end{proof}

To state Lemma \ref{limiting} used in the above proof, we set up some notation first. Recall from \cite[p. 489]{Kresch} that for a local embedding 
$\cX\ra \cY$ of 
 algebraic stacks of finite type over the base field, one has the normal cone $\bC _{\cX/\cY}$ to $\cX$ in $\cY$ and also the deformation of normal cone,
 denoted $M^{\circ}_{\cX}(\cY)$. This is a stack with a morphism to $\PP ^1$ such that the general 
 fiber is isomorphic to $\cY$ and the special fiber at $t=0 \in \PP ^1$ is isomorphic to $\bC _{\cX/\cY}$.
 If $\cX$ is a closed substack in $\cY$, the deformation can be obtained as in \cite[Chapter 5]{Fu}, by constructing 
 $$M_\cX(\cY):=Bl_{\cX\times\{0\}}\cY\times \PP^1 $$
 and setting $$M^{\circ}_{\cX}(\cY):=M_\cX(\cY) \setminus Bl_{\cX\times\{0\}}\cY\times\{0\}.$$

Now form the commuting diagram, whose squares are all cartesian
\[ \xymatrix{ 
\bC_{\Gamma ^0_{\infty}/\Gamma _{\infty}} \ar[d]\ar@{^{(}->}[r]_{j} & \bC_{\Gamma^0 /\Gamma } |_{\lambda = \infty} \ar[d]\ar[r] & \bC_{\Gamma^0 / \Gamma} \ar[d]\ar[r] 
& t=0 \ar[d]^{v_0}  \\
           M^{\circ}_{\kG ^0_{\infty}}(\Gamma_{\infty} ) \ar@{^{(}->}[r]^{i}_{\text{ closed } } &         M^{\circ}_{\kG ^0}(\Gamma )|_{\kl = \infty} \ar[r]\ar[d] 
                                   & M^{\circ}_{\kG ^0}(\Gamma ) \ar[d] \ar[r] & \mathbb{P}^1 \\
              &  \lambda =\infty \ar[r] & \mathbb{P}^1 \setminus \{ 1\} . & } \]
             
\begin{Lemma}\label{limiting} The equalities
 $$\infty ^{!} [\bC_{\Gamma ^0/\Gamma}] = j_*[\bC_{\kG ^0_{\infty}/\Gamma _{\infty}}]
=[\bC_{dist}] + \sum _{A, j_A} m_{j_A}[\bC_{j_A}] $$ hold in $A_*(\bC_{\Gamma^0 /\Gamma } |_{\lambda = \infty} )_{\QQ}$.
\end{Lemma}

\begin{proof} The equality $\infty ^{!} [\bC_{\kG^0/\Gamma}] = j_*[\bC_{\kG^0_{\infty}/\Gamma_{\infty}}]$ is a consequence of the definition of Gysin maps, their commutativity, and their compatibility with proper push-forward,
as follows:
\begin{align*}  & \infty ^{!} [\bC_{\kG^0 /\Gamma}] =  \infty ^{!} v_0^{!}  [M^{\circ}_{\kG^0 }(\Gamma) ]  =  v_0^{!}\infty ^{!}   [M^{\circ}_{\kG^0}(\Gamma)]  =  
v_0^{!} [M^{\circ}_{\kG^0}(\Gamma) |_{\infty}]  
\\ = &  v_0^{!} [ \overline{\Gamma_{\infty} \times (\mathbb{P}^1 - \{t=0\}) } ]  =
v_0^{!} i_* [ M^{\circ}_{\kG ^0_{\infty}} (\kG _{\infty})] = j_* v_0^{!} [   M^{\circ}_{\kG^0 _{\infty}} (\kG _{\infty})    ]  \\ = & j_*[\bC_{\kG^0 _{\infty}/\Gamma _{\infty}}]  . \end{align*}
Here some explanation is in order.  For the third equality in the above chain, note that $M^{\circ} _{\kG^0}(\kG)$ is irreducible and dominant over $\PP ^1 \setminus \{1\}$.
 The closure is taken in $M^{\circ}_{\kG^0}(\Gamma) |_{\infty}$. The fifth equality follows by the very definition of proper push-forward.
 
 The decomposition  $$j_*[\bC_{\kG ^0_{\infty}/\kG _{\infty}}]
=[\bC_{dist}] + \sum _{A, j_A} m_{j_A}[\bC_{j_A}] $$ is a consequence of the decomposition 
$[\kG _{\infty}] = [\kG_{\infty, dist}] +  \sum _{A, j_A} m_{j_A}[\kG_{\infty, j_A}] $ in $A_*(\kG _{\infty})_{\QQ}$ (Theorem \ref{thm:Gamma des}), via
the specialization to the normal cone homomorphism $ A_*(\kG _{\infty})_{\QQ} \ra A_*(\bC _{\Gamma ^0_{\infty}/\kG _{\infty}})_{\QQ}$.
\end{proof}

We finish this subsection by recording a basic intersection-theoretic Lemma which will be used several times in the sequel.

\begin{Lemma}\label{birational} Let $f:\cY'\lra \cY$ be a proper morphism between finite type 
Deligne-Mumford stacks of the same pure dimension.  
Let $i:\cX\hookrightarrow \cY$ be a closed embedding and form the fiber square
$$\xymatrix{\cX'\ar[d] \ar[r] &\cY'\ar[d]^f\\
\cX \ar[r]^{i} & \cY .}
$$
Let $\tilde{f}:\bC_{\cX'/\cY'}\lra \bC_{\cX/\cY}$ be the induced map between normal cones. 
If $f_*[\cY']=m[\cY]$ for a nonnegative rational number $m$, then $\tilde f_*[\bC_{\cX'/\cY'}]= m [\bC_{\cX/\cY}]$.
\end{Lemma}
\begin{proof} When $\cY,\cX$, and $\cY'$ are schemes, this is \cite[Lemma 3.15]{Vistoli}. 
For the convenience of the reader, we give a short argument.
Consider the deformations to the normal cone
$$M_\cX\cY=Bl_{\cX\times\{0\}}\cY\times\PP^1,\;\;\;\;\; M_{\cX'}\cY'=Bl_{\cX'\times\{0\}}\cY'\times\PP^1.$$ The map 
$\phi:M_{\cX'}\cY'\lra M_\cX\cY$ induced by $f$ is proper and $\phi_*[M_{\cX'}\cY']=m[M_\cX\cY]$.
Let $v_0:\{0\}\hookrightarrow\PP^1$ be
the inclusion. Denoting by ${\bf 1}$ the trivial rank one vector bundle, we have
\begin{equation}\label{deformation}
m [\PP(\bC_{\cX/\cY}\oplus{\bf 1})]+ m[Bl_\cX\cY]=mv_0^![M_\cX\cY]=v_0^!\phi_*[M_{\cX'}\cY']=(\phi|_{t=0})_*v_0^![M_{\cX'}\cY'],
\end{equation}
where we have used the commutativity of Gysin maps with proper push-forward for the last equality. Since 
$$v_0^![M_{\cX'}\cY']=[\PP(\bC_{\cX'/\cY'}\oplus{\bf 1})]+[Bl_{\cX'}\cY']$$ and $(\phi|_{t=0})_*[Bl_{\cX'}\cY']=m[Bl_{\cX}\cY]$, we conclude from \eqref{deformation} that
$$(\phi|_{t=0})_*[\PP(\bC_{\cX'/\cY'}\oplus{\bf 1})]=m[\PP(\bC_{\cX/\cY}\oplus{\bf 1})].$$
The Lemma follows, since $\tilde f$ is the restriction to $\bC_{\cX'/\cY'}$ of $\phi|_{t=0}$.
\end{proof}

\subsection{The correcting classes \texorpdfstring{$\mu_{d_a}^N(z)$}{Lg}}\label{correcting class}
Consider the Segre embedding
\begin{equation}\label{Segre0}
Seg :\PP(V)\times\PP(\CC^N)\lra \PP(V\ot \CC^N).
\end{equation}
Recall the map $h_a^+ : U_{k+A,d'_0}^+\lra \PP(\CC^N)$ given by the twisting line bundle $\sM_+$ and its sections $s_1,\dots, s_N$; see  \eqref{eqn:h_a}.
Viewing
$Q^+_{g, k+A}(X,d_0)$ as a substack of $U_{k+A,d'_0}^+$ 
via the embedding \eqref{embedding diagram} for the bundle $F^+_{d'_0}$, we have the restriction $h_a^+ : Q^+_{g, k+A}(X,d_0)\lra \PP(\CC^N)$;
see \eqref{eqn:d_0}  for notation $d_0=d_0^A$.
The two evaluation maps on $Q^+_{g, k+A}(X,d_0)$ at markings in $A$ are related by $$\hat{ev}_a|_{Q^+_{g, k+A}(X,d_0)}=Seg\circ(ev_a,h_a^+);$$ 
see \S \ref{sec:boundary} for notations $\hat{ev}_a$ and $ev_a$.

In this subsection we prove the following weaker version of the main theorem.

\begin{Thm}\label{Degeneration2} Let $z$ be a formal variable. There exists a Chow cohomology class $\mu_{d_a}^N(z)\in A^*(X\times\PP(\CC^N))_\QQ[z]$, 
dependent on $g$ and $k$ only through the dependence on $N$,
such that after push-forward to $A_*(Q^-_{g,k}(X,d))_\QQ$ by $c|_Z$, the equality of Lemma \ref{degeneration for cones} becomes
\begin{align}\label{weak formula}
&[Q^{-}_{g,k}(X,d)]^{\mathrm{vir}} - c_*[Q^{+}_{g,k}(X,d)]^{\mathrm{vir}} =\\
& \nonumber\sum_A \frac{1}{|A|!}
(b_A)_* (c_A)_* \left(\prod_{a\in A} (ev_a, h_a^+)^*\mu_{d_a}^N(z)|_{z=-\psi_a} \cap [Q^{+}_{g,k+A}(X,d_0^A)]^{\mathrm{vir}}\right ) .
\end{align}
\end{Thm}

\begin{proof}
We analyze the push-forward to $A_*(Q^-_{g,k}(X,d))_\QQ$ of each term in \eqref{refinedGraph} by $c|_Z$ which will be also denoted 
by $c$ for easy notation. We have also induced maps $$\bC_{\Gamma ^0_0/\Gamma_0}  \to c^*\bC_{Q^-_{g,k}(X,d)/U^-_{d'}} \to  \bC_{Q^-_{g,k}(X,d)/U^-_{d'}} , $$
whose composition will be denoted by $c_{\bf c}$.

The terms on the left-hand side are very easy.
First, by the identifications $(\Gamma^0_0\subset\Gamma_0)=(Z \subset U^+_{d'})$ and 
$\Up^0 |_{\Gamma^0_0}=c^*E^-_d$ we have
\begin{equation*} 
\begin{split}
c_* (\eta |_{Z}) _* (\iota _{0})_*( 0^!_{\Up^0 |_{\Gamma^0_0}}([ \bC_{\Gamma ^0_0/\Gamma_0} ]))
& =0^!_{E^-_d}(c_{{\bf c} *} [ \bC_{\Gamma ^0_0/\Gamma_0} ])\\
&=0^!_{E^-_d}([\bC_{Q^-_{g,k}(X,d)/U^-_{d'}}])\\
& = [Q^-_{g,k}(X,d)]^{\mathrm{vir}},
\end{split}
\end{equation*}
where we have used standard properties of the Gysin map for the first equality, Lemma \ref{birational} for the second equality, and Corollary \ref{concrete virtual class} for the third equality.

Second, 
\begin{equation*} 
c_*(\eta |_{Z}) _* (\iota _\infty)_* 0^!_{\Up^0 |_{\Gamma_{dist}}}([\bC_{dist}  ])=c_*([Q^+_{g,k}(X,d)]^{\mathrm{vir}}),
\end{equation*}
again by the identifications $(\Gamma^0_{\infty, dist}\subset \Gamma_{\infty, dist})=(Q^+_{g,k}(X,d) \subset U^+_{d'})$ and 
$\Up^0 |_{\Gamma_{\infty, dist}}=E^+_d$, together with Corollary \ref{concrete virtual class}.

The analysis of the right-hand side of  \eqref{refinedGraph} is significantly more subtle, so we divide it into several steps for clarity.

{\it Step 1: Transferring the computation to $\PP_{j_A}$.}
The Segre embedding \eqref{Segre0}, together with the inclusion $i:X\hookrightarrow \PP(V)$, induces the embedding 
\begin{align}\label{Segre}
i_{Seg}& :X\times \PP(\CC^N)\hookrightarrow \PP(V\ot \CC^N)\times\PP(\CC^N),  \\
\nonumber & (x,y)\mapsto( Seg(i(x),y),y).
\end{align}
We identify $X\times \PP(\CC^N)$ with its image under $i_{Seg}$.
Set
\begin{align*}\label{eqn:Qtaila}
Q_{tail,a}^+:=(\hat{ev}_a\times\id_{\PP(\CC^N)})^{-1}(X\times\PP(\CC^N)),
\end{align*}
a closed substack in $Q_{0,a}^+(\PP(V\ot\CC^N),d_a)\times\PP(\CC^N)$, and 
$$Q_{tail,A}^+:=\prod_{a\in A}Q_{tail,a}^+,$$
so that we have the cartesian square
\begin{equation*}
\xymatrix{Q_{tail,A}^+ \ar[r] \ar[d] &
\prod_{a\in A}(Q_{0,a}^+(\PP(V\ot\CC^N),d_a)\times\PP(\CC^N))\ar[d]^{\prod_a(\hat{ev}_a\times \id_{\PP(\CC^N)})}\\
(X\times\PP(\CC^N))^A \ar[r]^-{\prod_a i_{Seg}} & (\PP(V\ot \CC^N)\times\PP(\CC^N))^A.
}
\end{equation*}
Further, define the closed substack $D_{X, A} \subset D_A$ by the cartesian square
\begin{equation}\label{cartesian diag}
\xymatrixcolsep{7pc}
\xymatrix{D_{X, A}  \ar[r]^-{\mpr_{A}}\ar[d]_-{\pr_{A}} & Q^+_{g, k+A}(X,d_0)\ar[d]^-{((ev_a,h_a^+))_{a\in A}}\\
Q_{tail,A}^+ \ar[r]_-{\prod_a(\hat{ev}_a\times \id_{\PP(\CC^N)})} & (X\times\PP(\CC^N))^A.
}
\end{equation}
where by abusing notation $\mpr_{A}$, $\pr _A$ denote $\mpr_{A} |_{D_{X, A}}$, $\pr |_{D_{X, A}}$ respectively.
Note that $\prod_a(\hat{ev}_a\times \id_{\PP(\CC^N)})$ is a flat map (in fact, smooth) and therefore
so is $\mpr_A$.

Now fix the pair $(A,j_A)$ and
define $Z_{j_A}\subset \PP _{j_A}$ by the cartesian square
\begin{equation} \label{diag:alpha_j_a} \xymatrix{ Z_{j_A}  \ar[d]_{\alpha_{j_A}} \ar[r]      & \PP _{j_A} \ar[d]^{\alpha_{j_A}} \\
 \Gamma^0_{\infty, j_A} \ar[r] & {\Gamma}_{\infty, j_A} . }  
 \end{equation}
$Z_{j_A}$ is the zero locus of the section $\alpha_{j_A}^*\overline{\sigma}\in H^0(\PP_{j_A}, \alpha_{j_A}^* \Up)$. The restriction to $Z_{j_A}$
of the projection $\pi_{\PP}: \PP _{j_A}\lra D_A$ factors through $D_{X, A}$.

We assemble everything in the commuting diagram
\begin{equation}\label{Main Diagram}
\xymatrixcolsep{4pc}
\xymatrix{ \Gamma ^0_{\infty,j_A}\ar[r]^-{(\eta|_Z)\circ\iota_{\infty}} & Z  \ar[rr]^{c}  &  & Q^{-}_{g, k}(X, d)    \\
Z_{j_A}\ar[r]^-{\pi_{\PP}}\ar[u]^-{\alpha_{j_A}} & \ar[d]_{\pr _{A}} D_{X, A} \ar[r]^-{\mpr  _{A} } \ar[u]^{\nu_A}  
 &   \ar[d]^-{(({ev}_a,h_a^+))_{a\in A}} Q^+_{g, k+A}(X,d_0)  \ar[r]^{c_A\ \ \  }  
& Q^{-}_{g, k + A}(X, d_0 )  \ar[u]_{b_A}     \\
 & Q_{tail, A}^+
 \ar[r]_-{\prod  (\hat{ev}_a\times\id)}  & (X\times\PP(\CC^N)) ^A   &   } 
\end{equation}
 with abusing notation again $c=c|_Z$, $c_A=c_A|_{Q^+_{g, k+A}(X,d_0)}$ (this notation is justified by Remark \ref{compatible contractions} in \S \ref{sec:bar sigma})
 and $\nu _A = \nu _A |_{D_{X, A}}$ etc.
 
Let  \begin{align}\label{eqn:tildeC_j_A} \widetilde{\bC_{ j_A}}:=\bC_{Z_{j_A}/\PP_{j_A}} . \end{align}
By Lemma \ref{birational} applied to \eqref{diag:alpha_j_a} 
and the commutativity of the Gysin map with push-forward,
$$
0^!_{\Up^0 |_{\Gamma^0_{\infty, j_A}}}([\bC_{ j_A}] )= 
\frac{1}{|A|!}(\alpha_{j_A})_*0^!_{\alpha_{j_A}^*(\Up^0 |_{\Gamma ^0_{\infty,j_A}})}([\widetilde{\bC_{j_A}}]),
$$ where $\bC_{ j_A} := \bC_{\Gamma ^0_{\infty, j_A} /\Gamma _{\infty, j_A}}$ as defined in Lemma \ref{degeneration for cones}.
From the diagram \eqref{Main Diagram},
\begin{align*}
& \frac{1}{|A|!} c_* (\eta |_{Z}) _* (\iota _{\infty})_*(\alpha_{j_A})_*(0^!_{\alpha_{j_A}^*(\Up^0 |_{\Gamma ^0_{\infty,j_A}})}([\widetilde{\bC_{j_A}}]))=\\
& \frac{1}{|A|!} (b_A)_*(c_A)_*(\mpr_A)_*(\pi_\PP)_*(0^!_{\alpha_{j_A}^*(\Up^0 |_{\Gamma ^0_{\infty,j_A}})}([\widetilde{\bC_{j_A}}])).
\end{align*}

Letting $\Upsilon ^0_{j_A}$ denote $\alpha_{j_A}^*\Up^0 |_{\Gamma ^0_{\infty,j_A}} $, it remains to
show that 
\begin{equation}\label{main correction}
\sum_{j_A}m_{j_A}(\mpr_A)_*(\pi_\PP)_*(0^!_{\Up ^0_{j_A}}([\widetilde{\bC_{j_A}}]))
\end{equation}
has the form 
$$\left(\prod_{a\in A} (ev_a, h_a^+)^*\mu_{d_a}(z)|_{z=-\psi_a}\right) \cap [Q^{+}_{g,k+A}(X,d_0^A)]^{\mathrm{vir}},$$
as claimed in Theorem \ref{Degeneration2}.

\newcommand{\FF}{G}

{\it Step 2: Description of $\Up_{j_A}$.} We start by describing first 
\begin{align}\label{eqn:def Up_j_A} \Up_{j_A}:=\alpha_{j_A}^*\Upsilon |_{\Gamma_{\infty, j_A}} \end{align} on $\PP_{ j_A}$. 
Define vector bundles $\FF_{d'}^{+,j_A}$ and $\FF_{d'_0}^{-,j_A}$ on $D_A$
via exact sequences
\begin{align}\label{F tail}
0\ra \pr_{A} ^* F_{tail, d_a}^{+,j_A}\ra \nu_A^* F^+_{d'}\ra \FF_{d'}^{+,j_A}\ra 0,
\end{align}

\begin{align*} 
0\ra {\mpr}_A^*  c_A^* F_{d'_0}^{-,j_A}\ra \nu_A^* c^* F^-_{d'} \ra \FF_{d'_0}^{-,j_A}\ra 0.
\end{align*}
By \eqref{hat extension E}, we have an extension
\begin{align}\label{univ quotient}
0\ra \boxplus_{a\in A}(\cO_{\PP_{j_a}}(1)\ot \pi_{\PP}^* {\tt F}^{j_a}) \ra \Upsilon_{j_A}     
 \ra \pi_{\PP}^*(\FF_{d'}^{+,j_A}\oplus \FF_{d'_0}^{-,j_A})\ra 0.
\end{align}
Further, if we let
\begin{align*}
\FF^{+,j_A}_{tail, d_a}:= 
(\boxplus_{a\in A} {\pr}_a^*F^{+,1}_{tail,d_a})/ {\pr}_A^*F^{+,j_A}_{tail, d_a},
\end{align*}
then from \eqref{res} and \eqref{F tail} it follows that $\FF_{d'}^{+,j_A}$ fits into an extension
\begin{align}\label{F tail extension}
0\ra  \FF^{+,j_A}_{tail, d_a} \ra \FF_{d'}^{+,j_A} \ra {\mpr}_A^* F^+_{d'_0}\ra 0.
\end{align}

Note that we may write alternatively
\begin{align*} 
\FF_{d'_0}^{-,j_A}=& {\mpr}_A^* c_A^*(\oplus_{a\in A: j_a\leq d_a}\pi_*(\sP^{-}_{d'_0}\ot \cO_{(d_a-j_a)p_a}(d_ap_a))) \oplus \\
&\nonumber  {\mpr}_A^* c_A^*(\oplus_{a\in A} \pi_*(\oplus _{i : j_a \le l_i d_a  }  \sR^{-}_{i, d'_0}          
 \ot \cO_{(l_id_a-j_a)p_a}(l_id_ap_a))).
\end{align*}
and
\begin{equation*}
\FF^{+,j_A}_{tail, d_a}=
\left\{\begin{array}{ll}  {\pr}_A^*(\boxplus_{a\in A}\pi_*( (\sP^{+}_{d_a}\oplus \oplus_i       \sR^{+}_{i, d_a}         )         
\ot \cO_{(j_a-1)p_a^{tail}}(-p_a^{tail}))),
& \text{if } j_a \le d_a, \\
 {\pr}_A^*(\boxplus_{a\in A}\pi_*( \oplus_{i:j_a\leq l_id_a}  \sR^{+}_{i, d_a}                  
 \ot \cO_{(j_a-1)p_a^{tail}}(-p_a^{tail}))) , & \text{if } j_a > d_a, \end{array}\right.
\end{equation*}
from which it follows that in the $K$-group of vector bundles on $D_A$ 
\begin{align*}
\FF_{d'_0}^{-,j_A}\sim & 
\left(\boxplus_{a\in A} 
\oplus _{m=j_a+1}^{d_a}( {\mpr}_a^* c_a^*\cO(-m\psi_a)
\ot {\tt P}^{d_a-m})\right)\oplus \\
\nonumber &\left( \boxplus_{a\in A} \oplus_{i=1}^r \oplus _{m=j_a+1}^{l_id_a}
( {\mpr}_a^*  c_a^*\cO(-m\psi_a)\ot {\tt R}_{i}^{l_id_a-m})\right),
\end{align*}
and
\begin{align*}
 \FF^{+,j_A}_{tail, d_a}\sim \boxplus_{a\in A}\left ( \oplus_{m=1}^{j_a-1} ( \pr_a^*\cO(m\psi_a^{tail})\ot {\tt F}^{m})\right)\footnotemark
\end{align*}\footnotetext{The notation ${\tt F}^{m}$ is a little ambiguous, since the dependence on the marking $a$ is not apparent anymore.
The same will happen later, e.g., with the bundles ${\tt F}^{0}$ in \eqref{unglued} below. Hopefully this will not cause any confusion. }
where ${\tt P}^{d_a-m} := \mpr_a^*c_a^* {\tt P}^{d_a-m}$,   
${\tt R}_{i}^{l_id_a-m} := \mpr_a^* c_a^* {\tt R}_{i}^{-, l_id_a-m}$ 
(see \eqref{eqn:def tP-}, \eqref{eqn:def tR-i}, \eqref{eqn:def tF} for the definition of ${\tt P}^{-, d_a-m}$, ${\tt R}_{i}^{-, l_id_a-m}$, ${\tt F}^m$
respectively).

To summarize, the outer terms of the exact sequences \eqref{univ quotient} and \eqref{F tail extension} give four pieces that combine to make $\Upsilon_{j_A}$. 

We now move to the description of the subbundle $\Up ^0_{j_A}  \subset \Up_{j_A}|_{Z_{j_A}}$
(see \eqref{main correction} for the notation $\Up ^0_{j_A}$). For
each $1\leq i\leq r$ and $0\leq j_a$, introduce the bundles
\begin{equation*}
\tR_{i, small}^{j_a} :=\left\{\begin{array}{rr} {\pr}_{a} ^* (\hat{ev}_a\times\id_{\PP(\CC^N)})^*
\left(  \cO_{\PP(V\ot\CC^N)}(l_i)\boxtimes \cO_{\PP(\CC^N)}(-l_i)\right ),
& \text{if } j_a \le l_id_a, \\
0, & \text{if } j_a > l_i d_a , \end{array}\right. 
\end{equation*}
on $ D_{a}$. 
We use the same notation for the restrictions of $\tR_{i, small}^{j_a}$ to the substacks $D_A$ and 
$D_{X, A}$ of $D_{ a }$. Further, we set
$$\tR_{small}^{j_a} :=\oplus_{i=1}^r \tR_{i, small}^{j_a} .$$
Note that, alternatively, we may write on $D_{X, A}$
\begin{align*}
\tR_{small}^{j_a} & ={\pr}_{A} ^*(\pi_*(\oplus_i(\mathscr{L}_{+, d_a})^{l_i} \ot  \cO_{p_a^{tail}}))\\
& =\mpr_A ^*(\pi_*(\oplus_i(\mathscr{L}_{+, d_0})^{l_i} \ot  \cO_{p_a})),
\end{align*}
for $j_a\leq l_id_a$. Finally, put
$$
{\tF}^{j_a}_{small} :=\tP^{j_a}\oplus \tR^{j_a}_{small}.
$$

The surjection
$\Up_{j_A}\twoheadrightarrow \pi_{\PP}^* \mpr_A^* F^+_{d'_0}$ on $\PP_{j_A}$ (coming from \eqref{univ quotient} and \eqref{F tail extension}) induces a surjection $\Up ^0_{j_A}\twoheadrightarrow \pi_{\PP}^*\mpr_A^*E^+_{d_0}$ on $Z_{j_A}$. Define the {\it excess bundles} 
$\Up_{j_A,ex}$ and $\Up^0_{j_A,ex}$ as the corresponding kernels:
\begin{align*} 
0\lra \Up_{j_A,ex} \lra \Up_{j_A} \lra \pi_{\PP}^* \mpr_A^*F^+_{d'_0}\lra 0,
\end{align*}
\begin{align} \label{excess}
0\lra \Up^0_{j_A,ex} \lra \Up^0_{j_A} \lra \pi_{\PP}^*\mpr_A^* E^+_{d_0}\lra 0 .
\end{align}

To complete the description of $\Up^0_{j_A}$, we note that the excess bundle in turn fits into an extension
\begin{align}\label{excess 2}
0\ra \boxplus_{a\in A}(\cO_{\PP_{j_a}}(1)|_{Z_{j_A}}\ot \pi_{\PP}^* {\tF}^{j_a}_{small} ) \ra  \Up^0_{j_A,ex} \ra \pi_{\PP}^*(\FF^{+,j_A}_{tail, d_a, small}
\oplus \FF_{d_0, small}^{-,j_A} )\ra 0,
\end{align}
with
\begin{align*}
\FF^{+,j_A}_{tail, d_a, small}\sim \oplus_{a\in A}\left ( \oplus_{m=1}^{j_a-1} (\pr_A^*\cO(m\psi_a^{tail})\ot \tF^{m}_{small} )\right),
\end{align*}
\begin{align*}
\FF_{d_0, small}^{-,j_A}\sim & 
\left(\oplus_{a\in A} 
\oplus _{m=j_a+1}^{d_a} ({\mpr}_A^* c_A^*\cO(-m\psi_a)
\ot \tP^{d_a-m})\right)\oplus \\
\nonumber &\left( \oplus_{a\in A}  \oplus_{i=1}^r \oplus _{m=j_a+1}^{l_id_a}
({\mpr}_A^* c_A^*\cO(-m\psi_a)\ot \tR_{i, small}^{l_id_a-m})\right) 
\end{align*}
in the $K$-group of $D_A$.
For later use, we note that from the above $K$-group expressions it follows that the Euler classes of these bundles have the form
\begin{align}\label{F plus}
\mathbf{e}( \FF^{+,j_A}_{tail, d_a, small})=\pr_A^*\prod_{a\in A}  (\hat{ev}_a\times\id_{\PP(\CC^N)})^*f^{+,j_a}_{d_a}(z)|_{z=\psi_a^{tail}} ,
\end{align}
\begin{align}\label{F minus}
\mathbf{e}( \FF^{-,j_A}_{d_0, small})=\mpr_A^*\prod_{a\in A}  (ev_a,h_a^+)^*f^{-,j_a}_{d_a}(z)|_{z=-\psi_a} ,
\end{align}
where the Chow cohomology classes 
$$f^{+,j_a}_{d_a}(z), f^{-,j_a}_{d_a}(z)\in A^*(X\times\PP(\CC^N))_\QQ [z]=
(A^*(X)_\QQ\ot A^*(\PP(\CC^N))_\QQ)[z]$$
are polynomials in $z$ with coefficients which are {\it universal} expressions in Chern classes of various tautological bundles $\cO_X(l)$ on $X$, and $\cO_{ \PP(\CC^N ) }(m)$ 
and the tautological quotient bundle $Q$ on $\PP(\CC^N)$.

In the formula \eqref{F minus} we have used that the $\psi$-classes at markings in $A$ on $Q^-_{g,k+A}(X,d_0)$ and
$Q^+_{g,k+A}(X,d_0)$ pull-back under $c_A$, that is, $c_A^*\psi_a=\psi_a$.

{\it Step 3: Deformation.} The idea for computing \eqref{main correction} is to
deform the bundle $\Upsilon^0_{j_A}$, together with its closed 
subcone $\widetilde{\bC_{j_A}}$ (see \eqref{eqn:tildeC_j_A} for the notation $\widetilde{\bC_{j_A}}$), to 
the bundle $\Up^0_{j_A,ex}\oplus \pi_{\PP}^*\mpr_A^*E^+_{d_0}$
with the closed cone 
$\pi_{\PP}^*\mpr_A^*\bC_{Q^+_{g, k + A}(X, d_0)/U^+_{k+A, (d_0,d'_0)}}$ (see \eqref{excess} for the notation $\Up^0_{j_A,ex}$).

To begin with, consider on $ D_A$ the vector bundle homomorphisms
\begin{align*}\xymatrixcolsep{4pc}\xymatrix{
\pr_A^*(\oplus_{a\in A}F^+_{tail, d_a})\ar[r]^-{\oplus_a r_a^{tail}} &
\oplus_{a\in A}\tF ^0,}
\end{align*}
\begin{align*}\xymatrixcolsep{4pc}\xymatrix{
\mpr_A^*F^+_{d'_0}\ar[r]^-{\oplus_a r_a} &
\oplus_{a\in A}\tF ^0,}
\end{align*}
where $r_a^{tail}$ and $r_a$ are given by ``restricting sections at the marking $a$". The resulting surjective gluing map
\begin{align*}\xymatrixcolsep{4pc}\xymatrix{
\mpr_A^*F^+_{d'_0}\oplus \pr_A^*(\oplus_{a\in A}F^+_{tail, d_a})\ar[r]^-{\oplus_a (r_a-r_a^{tail})} & \oplus_{a\in A}\tF ^0\ar[r] & 0
}
\end{align*}
has kernel $\nu_A^*F^+_{d'}$.

Via its embedding in $\pi_\PP ^* (\nu_A^*F^+_{d'}\oplus \mpr_A^*c_A^* F^-_{d'})$,
we may view $\alpha _{j_A}^* (\zeta_{\sP\oplus \sR} |_{\Gamma_{\infty, j_A}})$ as a subbundle 
$$  \alpha _{j_A}^* (\zeta_{\sP\oplus \sR} |_{\Gamma_{\infty, j_A}})    \subset \pi_\PP ^* (
\mpr_A^*F^+_{d'_0}\oplus \pr_A^*(\oplus_{a\in A}F^+_{tail, d_a})
\oplus \mpr_A^*c_A^* F^-_{d'}).
$$
The quotient is an ``unglued" version of $\Up_{j_A}$. Precisely, it splits as
$\pi_\PP ^* \mpr_A^*(F^+_{d'_0})\oplus \Up_{j_A, ex, \hat{0}}$, and there are 
exact sequences 
\begin{align*}
\xymatrixcolsep{2pc}\xymatrix{
0\ar[r] & \Up_{j_A, ex}\ar[r] &  \Up_{j_A, ex, \hat{0}} \ar[r]^-{\oplus_a r_a^{tail}} & \pi_{\PP}^*(\oplus_a\tF ^0) \ar[r] &0
}
\end{align*}
and
\begin{align}\label{unglued}
\xymatrixcolsep{3pc}\xymatrix{
0\ar[r] & \Up_{j_A}\ar[r] & \pi_\PP ^* \mpr_A^*F^+_{d'_0}\oplus \Up_{j_A, ex, \hat{0}}
\ar[r]^-{\oplus_a (r_a-r_a^{tail})} & \pi_{\PP}^*(\oplus_a\tF ^0) \ar[r] &0
}
\end{align}
on $\PP_{j_A}\stackrel{\pi_{\PP}}{\lra} D_A$.
Composing the section $\overline{\sigma} : \cO_{\PP_{j_A}}\lra \Up_{j_A}$ with the monomorphism in
\eqref{unglued} gives the section
\begin{align*}
( \pi_\PP ^* \mpr_A^*\sigma^+_{d'_0},\overline{\sigma}_{ex}):  \cO_{\PP_{j_A}}\lra 
 \pi_\PP ^* \mpr_A^*F^+_{d'_0}\oplus \Up_{j_A, ex, \hat{0}}.
\end{align*}

The base of our deformation will be $\AAA^1$ with coordinate $t$. 
Denote $\varrho:\PP_{j_A}\times\AAA^1\lra \PP_{j_A}$ the projection. Define on $\PP_{j_A}\times\AAA^1$ the vector bundle $\ker$ via the exact sequence
\begin{align*}\xymatrixcolsep{3pc}\xymatrix{
0\ar[r] & \ker \ar[r] & \varrho^*(\pi_\PP ^* \mpr_A^*F^+_{d'_0}\oplus \Up_{j_A, ex, \hat{0}})
\ar[r]^-{\oplus_a (tr_a-r_a^{tail})} & \varrho^*\pi_{\PP}^*(\oplus_a\tF ^0) \ar[r] &0
}
\end{align*}
deforming \eqref{unglued}.
The section
$$\widetilde{\sigma}:=(\varrho^* \pi_\PP ^* \mpr_A^*\sigma^+_{d'_0},t\varrho^*\overline{\sigma}_{ex})$$
of $ \varrho^*(\pi_\PP ^* \mpr_A^*F^+_{d'_0}\oplus \Up_{j_A, ex, \hat{0}})$
factors through $\ker$, so we will view it from now on as a section 
of $\ker$. We have the identifications
$$(\ker|_{t=1},\widetilde{\sigma}|_{t=1})=(\Up_{j_A},\overline{\sigma})$$
and
$$(\ker|_{t=0},\widetilde{\sigma}|_{t=0})=( \pi_\PP ^* \mpr_A^*F^+_{d'_0}\oplus \Up_{j_A, ex}, ( \pi_\PP ^* \mpr_A^*\sigma^+_{d'_0},0)).$$
Let
$$\widetilde{Z}:=\widetilde{\sigma}^{-1}(0)\subset \PP_{j_A}\times\AAA^1$$
be the zero locus and observe that we have in fact
$$\widetilde{Z}\subset \PP_{j_A}|_{D_{X, A}}\times\AAA^1,$$
where $\PP_{j_A}|_{D_A}$ is the fibered product
$$
\xymatrix{ 
 \PP_{j_A}|_{D_{X, A}} \ar[r]^{\ \ \pi_{\PP}}  \ar@{^{(}->}[d]   & D_{X, A} \ar@{^{(}->}[d]    \ar[r]^{\mathrm{Pr}_A \ \ \ \ \ \ } &  Q^+_{g, k+A} (X, d_0) \ar[d]  \\
\PP_{j_A} \ar[r]_{\pi_{\PP}} &  D_{A}       \ar[r]_{{\mathrm{Pr}}_A \ \ \ }    & U^+_{k+A, d_0'}    .}
$$

The fibers of the $\AAA^1$-family $\widetilde{Z}$ at $t=1$ and at $t=0$ are 
$$\widetilde{Z}|_{t=1}=Z_{j_A}, \;\;\;\;\; \widetilde{Z}|_{t=0}=\PP_{j_A}|_{D_{X, A}}.$$
Notice that the normal cones satisfy
\begin{align}\label{cone at 0}
[\bC_{\widetilde{Z}/(\PP_{j_A}\times\AAA^1)}|_{t=0}]=[\bC_{(\PP_{j_A}|_{D_{X, A}})/\PP_{j_A}}]=
 \pi_\PP ^* \mpr_A^*[\bC_{Q^+_{g, k + A}(X, d_0)/U^+_{k+A, d'_0}}],
 \end{align}
 and
 \begin{align}\label{cone at 1}
 [\bC_{\widetilde{Z}/(\PP_{j_A}\times\AAA^1)}|_{t=1}]=[\widetilde{\bC_{j_A}}],
 \end{align}
 as desired.

The ``correct" obstruction bundle $\Up^0_{j_A}$ also deforms. Namely, if we repeat the construction
in this step, but with the bundles $\sP^{\pm}\oplus \sR^{\pm}$, $F^{\pm}_{d'}$ replaced by $\sQ^{\pm}$, $Q^{\pm}_{d'}:=\pi _* \sQ ^{\pm}_{ d'}$ respectively, 
we obtain an unglued
version of $\Up_{\sQ, j_A}:=\alpha_{j_A}^*\Up_{\sQ}|_{\Gamma_{\infty, j_A}}$ given as the extension
\begin{align*}
\xymatrixcolsep{3pc}\xymatrix{
0\ar[r] & \Up_{\sQ ,j_A}\ar[r] & \pi_\PP ^* \mpr_A^* Q^+_{d'_0}\oplus \Up_{\sQ,j_A, ex, \hat{0}}
\ar[r]^-{\oplus_a (r_a-r_a^{tail})} & \pi_{\PP}^*(\oplus_a \tF_{\sQ}^{0}) \ar[r] &0,
}
\end{align*}
and a vector bundle $ {\ker}_{\sQ}$ on $\PP_{j_A}\times\AAA^1$ defined via the deformation
\begin{align*}
\xymatrixcolsep{3pc}\xymatrix{
0\ar[r] & {\ker}_{\sQ}\ar[r] & \varrho^*(\pi_\PP ^* \mpr_A^* Q^+_{d'_0}\oplus \Up_{\sQ,j_A, ex, \hat{0}})
\ar[r]^-{\oplus_a (tr_a-r_a^{tail})} & \varrho^*\pi_{\PP}^*(\oplus_a \tF_{\sQ}^{0}) \ar[r] &0.
}
\end{align*}
Here $\tF_{\sQ}^{0}$ ``at the marking $a$" is the cokernel of $0\ra\tF ^0_{small}\ra\tF ^0$; alternatively,
$$
\tF_{\sQ}^{0} =\pr_{A} ^*(\pi_*(\mathscr{Q}^{+}_{d'_a} \ot  \cO_{p_a^{tail}}))
 =\mpr_A ^*(\pi_*(\mathscr{Q}^{+}_{d'_0}\ot  \cO_{p_a})).
$$
{\it After restricting to $\widetilde{Z}$}, there is a surjection 
$$\varrho^*(\pi_\PP ^* \mpr_A^*F^+_{d'_0}\oplus \Up_{j_A, ex, \hat{0}})\lra
\varrho^*(\pi_\PP ^* \mpr_A^* Q^+_{d'_0}\oplus \Up_{\sQ ,j_A, ex, \hat{0}})\lra0,$$
(just as in \S \ref{def of obs}), making the diagram
\begin{align*}
\xymatrixcolsep{4pc}\xymatrix{
\varrho^*(\pi_\PP ^* \mpr_A^*F^+_{d'_0}\oplus \Up_{j_A, ex, \hat{0}})\ar[d]
\ar[r]^-{\oplus_a (tr_a-r_a^{tail})} & \varrho^*\pi_{\PP}^*(\oplus_a\tF ^0) \ar[r]\ar[d] &0\\
\varrho^*(\pi_\PP ^* \mpr_A^* Q^+_{d'_0}\oplus \Up_{\sQ,j_A, ex, \hat{0}})\ar[d]
\ar[r]^-{\oplus_a (tr_a-r_a^{tail})} & \varrho^*\pi_{\PP}^*(\oplus_a \tF_{\sQ}^{0})\ar[d] \ar[r] &0\\
0 & 0 &
}
\end{align*}
commutative. We conclude that there is an induced map of vector bundles
$$ {\ker}\lra {\ker}_{\sQ},$$
which is easily seen to be surjective at all closed points, and hence surjective. Now define the correct
obstruction bundle $\widetilde{\Up}$ on $\widetilde{Z}$ as the kernel:
$$0\lra \widetilde{\Up} \lra  {\ker}\lra {\ker}_{\sQ} \lra 0.$$
At $t=1$ we have
\begin{align}\label{bundle at 1}
\widetilde{\Up}|_{t=1}=\Up ^0_{j_A},
\end{align}
while at $t=0$
\begin{align}\label{bundle at 0}
\widetilde{\Up}|_{t=0}= \pi_{\PP}^*\mpr_A^*E^+_{d_0}\oplus \Up^0_{j_A,ex}.
\end{align}
Here $\Up^0_{j_A,ex}$ on $\PP_{j_A}|_{D_{X, A}}$ is given by the same extension as in \eqref{excess 2}:
\begin{align}\label{excess 3}
0\ra \boxplus_{a\in A}(\cO_{\PP_{j_a}}(1)\ot \pi_{\PP}^* {\tF}^{j_a}_{small}) \ra  \Up^0_{j_A,ex} \ra \pi_{\PP}^*(\FF^{+,j_A}_{tail, d_a, small}\oplus \FF_{d_0, small}^{-,j_A} )\ra 0.
\end{align}

By a calculation similar to the one used to prove Lemma \ref{subcone}, one checks that 
the normal cone $\bC_{\widetilde{Z}/ (\PP_{j_A}\times\AAA^1)}$ is a subcone of $\widetilde{\Up}$.

Let $\iota:\widetilde{Z}\hookrightarrow \PP_{j_A}|_{D_{X, A}}\times\AAA^1 $ denote the inclusion and
consider the diagram
\begin{align*}
\xymatrix{        t \ar[d]_{t}     &  \ar[l]  \PP_{j_A}|_{D_{X, A}}  \ar[d]  & \ar[l]_{\ \iota _t}   \widetilde{Z}|_t \ar[r]\ar[d] & 
\bC_{\widetilde{Z} / (\PP_{j_A}\times\AAA^1)} |_t \ar[d]         \ar[r]  & t   \ar[d]^{t}  \\
     \AAA^1 &  \ar[l]   \PP_{j_A}|_{D_{X, A}}\times\AAA^1     & \ar[l]^-{\iota} \widetilde{Z}  \ar[r]  \ar[d] & 
     \bC_{\widetilde{Z} / (\PP_{j_A}\times\AAA^1)}  \ar[d] \ar[r] & \AAA^1  \\
                            &  &  \widetilde{Z}  \ar[r]^-{0}  &   \widetilde{\Upsilon} . & }
\end{align*}                            
The proof of Lemma \ref{degeneration for cones} shows the equality
\begin{align*}
(\iota_1)_*0^!_{\widetilde{\Up}|_{t=1}}([\bC_{\widetilde{Z}/(\PP_{j_A}\times\AAA^1)}|_{t=1}])=
0^!_{\widetilde{\Up}|_{t=0}}([\bC_{\widetilde{Z}/(\PP_{j_A}\times\AAA^1)}|_{t=0}])
\end{align*}
in the Chow group of $ \PP_{j_A}|_{D_{X, A}} $. By \eqref{cone at 0}, \eqref{cone at 1},  \eqref{bundle at 1},
\eqref{bundle at 0}, the Excess Intersection Formula (\cite[Theorem 6.3]{Fu}), 
the compatibility of Gysin maps with flat pull-back,
and Corollary \ref{concrete virtual class}, this can be rewritten as
\begin{align}\label{reduction}
(\iota_1)_*0^!_{\Up^0_{j_A}}([\widetilde{\bC_{j_A}}])=\mathbf{e} (\Up^0_{j_A,ex})\cap \pi_\PP^*\mpr_A^*
[Q^{+}_{g,k+A}(X,d_0)]^{\mathrm{vir}},
\end{align}
where $\mathbf{e}$ denotes the Euler class and $\pi_\PP^*$, $\mpr_A^*$ are the flat pull-backs.

{\it Step 4: Final calculation.} Recall the diagram from \eqref{Main Diagram}
\begin{align*}
\xymatrixcolsep{4pc}
\xymatrix{  & \PP_{j_A}|_{D_{X, A}}  \ar[d]^{\pi_\PP}  &     \\
Z_{j_A}\ar[r]^-{\pi_{\PP}}\ar[ur]^-{\iota_1} & \ar[d]_{\pr _{A}} D_{X, A}\ar[r]^-{\mpr  _{A} } 
 &   \ar[d]^-{(({ev}_a,h_a^+))_{a\in A}} Q^+_{g, k+A}(X,d_0)  
    \\
 & Q_{tail, A}^+
 \ar[r]_-{\prod  (\hat{ev}_a\times\id)}  & (X\times\PP(\CC^N)) ^A      } 
 \end{align*}
and that we want to compute \eqref{main correction}.
From \eqref{reduction} this is the same as computing
\begin{align}\label{main correction 2}
\sum_{j_A}m_{j_A}(\mpr_A)_*(\pi_{\PP})_*\left ( \mathbf{e} (\Up^0_{j_A,ex})\cap \pi_\PP^*\mpr_A^*
[Q^{+}_{g,k+A}(X,d_0)]^{\mathrm{vir}}\right ).
\end{align}
By \eqref{excess 3}, 
$$\mathbf{e} (\Up^0_{j_A,ex})= \mathbf{e}(\boxplus_{a\in A}(\cO_{\PP_{j_a}}(1)\ot \pi_{\PP}^* {\tF}_{small}^{j_a})) \mathbf{e}(\pi_{\PP}^*(\FF^{+,j_A}_{tail, d_a, small}))\mathbf{e}((\pi_\PP)^*( \FF_{d_0, small}^{-,j_A} )).
$$

Set $\alpha:= \mathbf{e}(\FF^{+,j_A}_{tail, d_a, small})\mathbf{e}(\FF_{d_0, small}^{-,j_A} )\cap 
\mpr_A^*[Q^{+}_{g,k+A}(X,d_0)]^{\mathrm{vir}}$. Then \eqref{main correction 2} can be successively rewritten as
\begin{align*}
&\sum_{j_A}m_{j_A}(\mpr_A)_*\left \{ (\pi_\PP)_*\left ( \mathbf{e}(\boxplus_{a\in A}(\cO_{\PP_{j_a}}(1)\ot \pi_{\PP}^* {\tF}_{small}^{j_a}) ) \cap \pi_\PP^*\alpha \right)\right\}\\
&=\sum_{j_A}m_{j_A}(\mpr_A)_*\prod_{a\in A} (\pi_{\PP})_*\left(\sum_{m=0}^{\rk({\tF}_{small}^{j_a})}
c_1(\cO_{\PP_{j_a}}(1))^m\cap \pi_{\PP}^* \left(c_{\rk({\tF}_{small}^{j_a})-m}({\tF}_{small}^{j_a})\cap\alpha\right)\right)\\
&=\sum_{j_A}m_{j_A}(\mpr_A)_*\prod_{a\in A} \left(\sum_{m=0}^{\rk({\tF}_{small}^{j_a})}
s_{m-1}\left(\pr_A^*\cO(j_a\psi_a^{tail})\oplus\mpr_A^*\cO(-j_a\psi_a)\right)c_{\rk({\tF}_{small}^{j_a})-m}({\tF}_{small}^{j_a})\cap\alpha \right),
\end{align*}
where $s_{m-1}$ denote the Segre classes.

The Chow cohomology class
\begin{equation*}
\sum_{m=0}^{\rk({\tF}_{small}^{j_a})}
s_{m-1}\left(\pr_A^*\cO(j_a\psi_a^{tail})\oplus\mpr_A^*\cO(-j_a\psi_a)\right)c_{\rk({\tF}_{small}^{j_a})-m}({\tF}_{small}^{j_a})
\end{equation*}
is a polynomial in $\mpr_A^* \psi_a$, of the form
\begin{equation*}
\sum_{b} \pr_A^*((\hat{ev}_a\times\id)^*\delta_b(z)|_{z=\psi_a^{tail}}) \mpr_A^* \psi_a^b,
\end{equation*}
where the $\delta_b$'s are themselves polynomials with coefficients given by 
universal expressions in Chern classes of various tautological bundles $\cO_X(l)$ on $X$, and $\cO_{ \PP(\CC^N)}(m)$ and $Q$ on 
$\PP(\CC^N)$. 
Further, by \eqref{F plus}, \eqref{F minus}, the Euler classes $\mathbf{e}(\FF^{+,j_A}_{tail, d_a, small})$ and $\mathbf{e}(\FF_{d_0, small}^{-,j_A} )$ appearing in 
$\alpha$ are given respectively by the universal expressions $\prod_a\pr_A^*(\hat{ev}_a\times\id)^*f^{+,j_a}_{d_a}(\psi_a^{tail})$ and 
$\prod_a\mpr_A^*(ev_a,h_a^+)^* f_{d_a}^{-,j_a}(-\psi_a)$.

Setting 
$$\gamma_b:= (\hat{ev}_a\times\id)^*(\delta_bf^{+,j_a}_{d_a})(\psi_a^{tail})\in A^*(Q^+_{tail,a})_{\QQ}
$$
and
recalling that $m_{j_A}=\prod_{a\in A} j_a$, we conclude that \eqref{main correction 2} 
has the form
\begin{align}\label{almost there}
\prod_{a\in A}\left( \sum_{j_a=1}^{\max_i\{l_id_i\}} j_a (\mpr_A)_*\left\{ \sum_b\pr_A^*(\gamma_b)
\mpr_A^*(\psi_a^b(ev_a,h_a^+)^* f_{d_a}^{-,j_a}(-\psi_a))\right\} \left([Q^{+}_{g,k+A}(X,d_0)]^{\mathrm{vir}}
\right)
\right).
\end{align}
Here $(\mpr_A)_*:A^*(D_{X, A})_\QQ\lra A^*(Q^+_{g, k + A}(X, d_0 ))_\QQ$ denotes the Gysin map induced by the
bivariant class $[\mpr_A]$ corresponding to the canonical orientation of the flat proper morphism $\mpr_A$,
see equation $(\mathrm{G}_2)$ in \cite[\S 17.4]{Fu}. Applying
\cite[Example 17.4.1(b)]{Fu} to the cartesian square \eqref{cartesian diag} and using the projection formula for bivariant classes, equation \eqref{almost there} proves Theorem \ref{Degeneration2}, with
\begin{align}\label{explicit mu}
\mu_{d_a}^N(z):= \sum_{j_a=1}^{\max_i\{l_id_i\}} j_a \sum_b  (-z)^b
f_{d_a}^{-,j_a}(z)(\hat{ev}_a\times\id)_*(\gamma_b)\in A^*(X\times\PP(\CC^N))_\QQ[z].
\end{align}
\end{proof}
We stress again that our argument shows that the formula \eqref{explicit mu} for the correcting class $\mu_d^N$
is universal in the following sense: it depends on $g$ and $k$ only through the dependence on $N$ of
the polynomials $f_{d_a}^{+,j_a}(z), f_{d_a}^{-,j_a}(z), \delta_b(z)\in A^*(X\times\PP(\CC^N))_\QQ[z]$.
This will be used in the next subsection.

\subsection{Identification of the correcting class}\label{identification}
In this subsection we finish the proof of Theorem \ref{Main} (for $(g,k)\neq (1,0)$) by showing that the class
\eqref{explicit mu} satisfies
\begin{align}\label{first identification}
\mu_{d_a}^N(z)={\mathrm{ coefficient \; of }}\; q^{d_a}\; {\mathrm {in\;}} z(J^{\ke_-}_{sm}(z)-J^{\ke_+}_{sm}(z))\ot \one_{\PP(\CC^N)}.
\end{align}
Indeed, assuming \eqref{first identification}, it follows first that the coefficient of $q^{d_a}$ in
$z(J^{\ke_-}_{sm}(z)-J^{\ke_+}_{sm}(z))$
is a polynomial in $z$ (because the left-hand side is such) and then by the general asymptotic properties of the small
$J^\ke$-functions it coincides with the coefficient of $q^{d_a}$ in $[zI_{sm}(q,z)-z]_+$.
Second, \eqref{first identification} also shows that
the class
$(ev_a,h_a^+)^*\mu_{d_a}^N(z)$ is independent of $N$, so that we may replace it by $ev_a^*\mu_{d_a}(z)$ in the formula \eqref{weak formula}.
Hence Theorem
\ref{Degeneration2} together with \eqref{first identification} imply Theorem \ref{Main}.

To prove \eqref{first identification}, we take $d=d_a$ (so that $d_0=0$) and consider the graph spaces 
$QG_{0, 0, d_a}^\pm(X)$. These are the moduli stacks of $\ke_{\pm}$-stable quasimaps of degree $d_a$ to $X$, whose domains are genus zero unpointed curves with a component which is a parametrized 
$\PP^1$, see \cite{CKM, CKg0}. Similarly, we have the moduli stacks $QG_{0, 0, d_a}^\pm(\PP(V))$ and
$QG_{0, 0, d_a}^\pm(\PP(V\ot\CC^N))$, which are smooth. The $\ke_-$-stability condition implies that the domain curve must be an irreducible parametrized $\PP^1$, 
while $\ke_+$-stability allows in addition quasimaps with domain consisting of one rational tail and the parametrized $\PP^1$. These quasimaps have degree $d_a$ on the rational tail and are constant maps on the parametrized $\PP^1$. In particular, there are identifications
$$QG_{0, 0, d_a}^-(\PP(V) ) \cong \PP(\Sym^{d_a}(\CC^2)\ot V),$$
$$QG_{0, 0, d_a}^-(\PP(V\ot\CC^N))\cong \PP(\Sym^{d_a}(\CC^2)\ot V\ot \CC^N).$$

Recall that we have the embeddings
$$X\times\PP(\CC^N)\hookrightarrow \PP(V)\times\PP(\CC^N)\hookrightarrow\PP(V\ot\CC^N)\times
\PP(\CC^N),$$
whose composition is the map $i_{Seg}$ from \eqref{Segre}. The induced embeddings of graph spaces commute with the contraction maps:
\begin{equation}\label{obvious embeddings}
\xymatrix{ QG^+_{0, 0, d_a}(X)\times\PP(\CC^N) \ar[d]_{c\times\id}  \ar@{^{(}->}[r]   & 
QG^+_{0, 0, d_a}(\PP(V\ot\CC^N))\times\PP(\CC^N)\ar[d]^{c\times\id}\\
QG^-_{0, 0, d_a}(X)\times\PP(\CC^N)\ar@{^{(}->}[r]  & QG^-_{0, 0, d_a}(\PP(V\ot\CC^N))\times\PP(\CC^N).}
\end{equation}
The right contraction map $c\ti \id$ is an isomorphism outside the boundary divisor
\begin{align*}
  D_{a} & \cong (Q^+_{0,\{a\} }(\PP(V\ot\CC^N),d_a)\times\PP(\CC^N))\times_{\PP(V\ot\CC^N)\times\PP(\CC^N)} (QG_{0, \{a\},0}^+(\PP(V\ot\CC^N))\times\PP(\CC^N))\\
 & \cong (Q^+_{0,\{a\} }(\PP(V\ot\CC^N),d_a)\times\PP(\CC^N))\times_{\PP(V\ot\CC^N)\times\PP(\CC^N)} 
 (\PP(V\ot\CC^N)\times\PP^1 \times\PP(\CC^N)),
\end{align*}
where $QG_{0, \{a\},0}^+(\PP(V\ot\CC^N)) \cong \PP(V\ot\CC^N)\times\PP^1$ is the moduli stack of 
$\ke_{+}$-stable quasimaps of degree $0$ to $\PP(V\ot\CC^N)$, whose domains are genus zero {\em one-pointed} curves with a component which is a parametrized 
$\PP^1$, see \cite{CKM, CKg0}.
Let $\sL_\pm$ denote the universal line bundles of degree $d_a$ on the fibers of the universal curves
over the various $QG^\pm\times\PP(\CC^N)$. 
Let $\sM$ denote the pull-back of $\cO_{\PP(\CC^N)}(1)$ to $QG^\pm\times\PP(\CC^N)$, with the basis $\{t_1,\dots,t_N\}$ of global sections, and set $\sL'_\pm=\sL_\pm\ot\sM$.
With these notations (which are justified, since the line bundles are compatible with the above embeddings), the construction of \S\ref{POT} 
produces the obstruction theory \eqref{perfect} of $QG_{0, 0, d_a}^\pm(X)\times\PP(\CC^N)$ relative to the smooth, pure dimensional stack $\BunG^{\PP^1}\times\PP(\CC^N)$. 
Here $\BunG^{\PP^1}\lra\widetilde{\PP^1[0]}$ is the relative Picard stack over the Fulton-MacPherson stack $\widetilde{\PP^1[0]}$ of unpointed rational curves with one parametrized component. The corresponding virtual class is $[QG_{0, 0, d_a}^\pm(X)]^{\mathrm{vir}}\times[\PP(\CC^N)]$. 
Note that for all universal curves, the map $h$ to $\PP(\CC^N)$ is just the projection.

Further, if we put 
$$U^\pm:= QG_{0, 0, d_a}^\pm(\PP(V\ot\CC^N))\times\PP(\CC^N),$$
then the construction of \S\ref{U-embedding} also applies to produce the vector bundles $F^\pm$ on $U^\pm$,
with sections $\sigma^\pm$ such that $(\sigma^\pm)^{-1}(0)\cong QG^+_{0, 0, d_a}(X)\times\PP(\CC^N)$. This embedding of $QG^+_{0, 0, d_a}(X)\times\PP(\CC^N)$ in $U^\pm$ is precisely the one in \eqref{obvious embeddings}. The diagram \eqref{embedded obstruction} holds as well, hence we have the concrete
description 
$$[QG_{0, 0, d_a}^\pm(X)]^{\mathrm{vir}}\times[\PP(\CC^N)]=0^!_{E^\pm}(\bC _{QG_{0, 0, d_a}^\pm(X)\times\PP(\CC^N)/U^\pm})
$$
as in Corollary \ref{concrete virtual class}. 

From the degeneration analysis in \S \ref{sec:boundary} -- \S \ref{correcting class},
it follows that Theorem \ref{Degeneration2} holds in the situation considered in this section, giving the equality
\begin{align}\label{full}
&[QG_{0, 0, d_a}^-(X)]^{\mathrm{vir}}\times[\PP(\CC^N)]-(c\times\id)_*([QG_{0, 0, d_a}^+(X)]^{\mathrm{vir}}\times[\PP(\CC^N)])=\\
\nonumber &
(b_{a}\times\id)_*((ev_a,h_a^+)^*\mu_{d_a}^N(-\psi_a)\cap ([QG_{0, \{a\},0}^+(X)]^{\mathrm{vir}}\times[\PP(\CC^N)]),
\end{align}
with $\mu_{d_a}^N$ the universal class in \eqref{explicit mu}.
Notice that the one-pointed, {\it degree zero} graph space is identified with $X\times\PP^1$, with virtual class the usual fundamental class (for any stability parameter $\ke$), while the maps 
$$ev_a :X\times\PP^1\times\PP(\CC^N)\lra X,\;\;\; h_a^+ : X\times\PP^1\times\PP(\CC^N)\lra \PP(\CC^N)$$
are respectively the first and third projections. The class $\psi_a$ is the pull-back of $c_1(\omega_{\PP^1})$ via
the second projection.

Now recall that graph spaces carry a $\CC^*$-action (induced by the standard action on the parametrized domain component) for which the maps $c$ and $b_a$ are equivariant. 
It is customary to denote by $z$ the equivariant parameter for this action.
In each graph space there is a distinguished part of the $\CC^*$-fixed locus corresponding to quasimaps for which the entire nontrivial data is concentrated over the point $0$ in the parametrized domain component. The restrictions of the maps $c$ and $b_a$ to the fixed point locus respect
the decomposition into distinguished and non-distinguished parts. 
It follows that if we apply the virtual localization
formula of \cite{GP} to \eqref{full} (using the trivial action on the $\PP(\CC^N)$ factors) and discard from both sides the localization residues at all 
non-distinguished fixed-point loci, we still have an equality between the remaining distinguished residues.

In our particular case, the distinguished fixed locus in $QG^-_{0, 0, d_a}(X)\times\PP(\CC^N)$ is identified with $X\times\PP(\CC^N)$, the distinguished fixed locus in 
$QG^+_{0, 0, d_a}(X)\times\PP(\CC^N)$ 
is identified with $Q^+_{0,1}(X,d_a)\times\PP(\CC^N)$, and the distinguished fixed locus in
$QG_{0, \{a\},0}^+(X)\times \PP(\CC^N) = X\times\PP^1\times\PP(\CC^N)$ is $X\times\{0\}\times\PP(\CC^N)$. Moreover, the restriction 
of $c\times\id$ to the distinguished fixed locus is $ev_1\times\id$, while $b_a\times\id$,
$(ev_a, h_a^+)$ are the identity map on the distinguished fixed locus. The equality of distinguished residues of 
\eqref{full} becomes
\begin{align}\label{full_2}
{\mathrm{ coefficient \; of }}\; q^{d_a}\; {\mathrm {in\;}} (J^{\ke_-}_{sm}(z)-J^{\ke_+}_{sm}(z))\ot \one_{\PP(\CC^N)}=\frac{\mu_{d_a}^N(z)}{z}
\end{align}
in $A^*(X \ti \PP(\CC ^N))_{\QQ}[z, z^{-1}]$,
proving \eqref{first identification}. Indeed, the left-hand side is as stated by the very definition of the
small $J$-functions in (5.1.1) of \cite{CKg0}, while for the right-hand side we used that, in the right-hand side of \eqref{full},
$\psi_a|_{X\times\{0\}\times\PP(\CC^N)}=-z$, and that the equivariant normal bundle of $\{0\}\subset\PP^1$ has first Chern class $z$, i.e., the denominator $z$
in the right-hand side of \eqref{full_2} so that 
$\frac{1}{z}$ is the distinguished residue of $[QG^+_{0, \{ a \}, 0} (X)]^{\vir} \ti [\PP(\CC ^N)]$.

\subsection{The unpointed genus \texorpdfstring{$1$}{Lg} case}\label{elliptic}
Since $\overline{M}_{1,0}$ is empty, we do not have the twisting line bundles 
$\sM$ satisfying Lemma \ref{easy lemma} and which are all compatible.
However, it turns out that an appropriate modification of
the set-up in \S\ref{virtual classes} allows for an application of the arguments in \S\ref{proof of Main Thm} to establish Theorem \ref{Main} in this case as well. 

\subsubsection{Set-up}
By an unpointed semistable genus $1$ curve we mean an unpointed prestable genus $1$ curve with no rational tails. Let $\fM_{1,0}^{ss}$ denote the moduli stack  of semistable genus $1$ curves.

Fix positive integers $d$ and $e$. Let $M_N$ denote the moduli stack of degree $e$ unpointed genus $1$ {\em stable maps} to $\PP (\CC ^N)$ 
with {\em semistable} domain curves. Since all line bundles of degree $e$ on semistable genus $1$ curves are non-special, $M_N$ is a smooth (non-proper) Deligne-Mumford stack. Denote by $\fC _{1,0}^{ss}\lra M_N$ the universal curve and by $$h:\fC _{1,0}^{ss}\lra\PP(\CC^N)$$ the universal map.

Let $d' = d + e$ and
let $Q_{1,0}^{\ke, unob}(\PP (V\ot \CC^N), d')$ be the open substack of $Q_{1,0}^{\ke}(\PP(V\ot \CC^N), d')$ consisting of $\ke$-stable quasimaps
$(C, L', u')$ with vanishing $H^1(C, L')$.
Define $U_{d'}^{\ke, N}$ as the fiber product $$Q_{1,0}^{\ke, unob}(\PP (V\ot \CC^N), d') \times _{\fM _{1,0}^{ss}} M_N . $$
Here the morphism $Q_{1,0}^{\ke, unob}(\PP (V\ot \CC^N), d') \ra \fM _{1,0}^{ss}$ is the composite 
of the contraction map $Q_{1,0}^{\ke, unob}(\PP  (V\ot \CC^N),  d')\ra Q_{1,0}^{0+}(\PP (V\ot \CC^N), d')$ and the forgetful map 
$Q_{1,0}^{0+}(\PP (V\ot \CC^N), d')\ra \fM_{1,0}^{ss}$.

 Since $M_N$ is smooth over $\fM_{1,0}^{ss}$ and $Q_{1,0}^{\ke, unob}(\PP  (V\ot \CC^N), d')$ is smooth over $\fBun ^{1,0}$,
 the stack  $U_{d'}^{\ke, N}$ is smooth over $\fBun ^{1,0}$.

 The universal curve $\fC ^{\ke}_{1, 0, d'}$ over $U_{d'}^{\ke, N}$ has a semistabilization morphism $ss_{\ke}: \fC ^{\ke}_{1, 0, d'} \ra \fC _{1,0}^{ss}$ (the contraction of rational tails of universal curves), fitting into the commuting diagram 
 \begin{equation*}
 \xymatrix{\fC ^{\ke}_{1, 0, d'}\ar[d]^{\pi}\ar[r]_{ss_{\ke}} &  \fC ^{ss}_{1,0} \ar[r]_{h} \ar[d] & \PP(\CC^N) \\
 U^{\ke, N}_{d'} \ar[r]_{\text{proj}}  & M_N .&
 }
 \end{equation*}
 We set $h_\ke=h\circ ss_{\ke} : \fC ^{\ke}_{1, 0, d'} \ra \PP(\CC^N)$ and $\sM_\ke=h_\ke^*\cO_{\PP(\CC^N)}(1)$. Further, the sections $t_j$ of $\cO_{\PP(\CC^N)}(1)$ associated to the homogeneous coordinates of $\PP (\CC ^N)$ give the
sections $s_j:=h_\ke^* t_j\in H^0(\fC ^{\ke}_{1, 0, d'},\sM_\ke)$, $j=1, ..., N$.

\subsubsection{Obstruction theory for  $Q^{\ke}_{1,0}(X, d) \times_{\fM _{1,0}^{ss}}M_N $ relative to $\fBun ^{1,0}$}

 Denote by $\sL ' _{\ke}$ the universal line bundle on the universal curve $\fC  ^{\ke}_{1,0, d'}$ of $U_{d'}^{\ke, N}$ and put $\sL_{\ke} := \sL_{\ke} ' \ot \sM_\ke^{-1}$.
 
 Consider the diagram of vector bundles and $\cO _{\fC  ^{\ke}_{1,0, d'}}$-linear maps, corresponding to \eqref{twisting},
 \begin{equation*} \xymatrix{ 0 \ar[r]   &  \mathscr{L}_{\ke} \ot V   \oplus h_{\ke}^*T_{\PP (\CC ^N)}  \ar[r]^{(\oplus _j s_j, \mathrm{id} ) \;\;\;\;}   
 &  \oplus _{j=1}^N  \mathscr{L}'_{\ke} \ot V  \oplus h_{\ke}^*T_{\PP (\CC ^N)}   \ar[r]\ar[d]^{(\oplus_j(\oplus_id\varphi_i), 0)}        & \mathscr{P}^{\ke}_{ d'}  \ar[r]                 & 0 \\
                      0 \ar[r]  & \oplus _{i=1}^r  \sL_{\ke}  ^{l_i}  \ar[r]^{\oplus _{i, j} s_j^{l_i}}       
& \oplus _{i, j}   (\sL '_{\ke}) ^{l_i} \ar[r]                              & \mathscr{Q}^{\ke}_{ d'}  \ar[r]              & 0 .} \end{equation*}

  Let $Q^{\ke}_X:=Q^{\ke}_{1,0}(X, d) $.
 As before, there is a vector bundle $$P^{\ke}_{d'} \oplus R^{\ke}_{d'} := \pi _*\sP^{\ke}_{d'} \oplus \pi _* (\oplus _{i, j}   (\sL '_{\ke}) ^{l_i})$$
 on $U_{d'}^{\ke, N}$, with a section $\sigma ^{\ke}$ whose zero locus
is naturally isomorphic to the product stack $$Q^{\ke}_{X} \times _{\fM _{1,0}^{ss}} M_N . $$

On the universal curve $\fC ^{\ke}_{X}$ over $Q^{\ke}_X \times _{\fM ^{ss}_{1,0}} M_N$ (associated to the universal curve of $Q^{\ke}_X$), 
we may complete the diagram above to a homomorphism of short exact sequences.  In particular, we obtain 
a natural homomorphism  $$\mathscr{L}_{\ke} \ot V  \oplus h_{\ke}^*T_{\PP(\CC^N)} \ra   \oplus _{i=1}^r  \sL_{\ke}  ^{l_i}  $$
and an exact sequence $$0 \ra \mathscr{E}^{\ke}_{ d} \ra  \sP^{\ke}_{d'} \oplus   (\oplus _{i, j}   (\sL '_{\ke}) ^{l_i}  )   \ra   \sQ^{\ke}_{ d'}  \ra 0,$$
defining a vector bundle $\mathscr{E}^{\ke}_{ d}$ on $\fC^{\ke}_{X}$, with $\pi_*\mathscr{E}^{\ke}_{ d} $ also locally-free.
 
        Denote by $\mathbf{C}_{\sigma ^{\ke}}$ the normal cone to
$Q^{\ke}_{X} \times _{\fM _{1,0}^{ss}} M_N$ in  $    U_{d'}^{\ke, N}   $. As before, $\mathbf{C}_{\sigma ^{\ke}}$ 
is a closed subcone of the vector bundle $\pi_*\mathscr{E}^{\ke}_{ d} $, with the embedding induced by a surjection $\pi_*\mathscr{E}^{\ke}_{ d}\twoheadrightarrow 
\mathscr{I}/\mathscr{I}^2$, where $\mathscr{I}$ is the ideal sheaf of the closed substack $Q_X^{\ke} \times_{\fM _{1,0}^{ss}}M_N$.

 Consider the following commuting diagram
  \[ \xymatrix{ Q_X^{\ke} \times_{\fM _{1,0}^{ss}}M_N  \ar@{^{(}->}[r]_{\text{closed}} &  U_{d'}^{\ke, N} \ar[r]\ar[d]  & M_N \ar[d]^{\text{smooth}}  \\
                          & Q_{1,0}^{\ke, unob}(\PP(V\ot \CC^N), d') \ar[r] \ar[d]^{\text{smooth}} & \fM_{1,0}^{ss} \\
                          & \fBun ^{1,0}  &  } \]
and {\it define} a perfect obstruction theory $\mathbb{E}$ for $Q_X^{\ke} \times_{\fM _{1,0}^{ss}}M_N$ relative to $\fBun ^{1,0}$
by \begin{align*} &
 \left[ R^{\bullet}   \pi _* (    \mathscr{L}_{\ke}  \ot V  \oplus h_{\ke}^*T_{\PP (\CC ^N)}  \ra   \oplus _{i=1}^r  \sL_{\ke}  ^{l_i}   )\right]^{\vee} \\
 \stackrel{\text{qiso}}{\sim} &  
 \left[ (\pi _* \mathscr{E}^{\ke}_{ d} ) ^{\vee} \ra  ( \oplus _{j=1}^N  \pi_*    \mathscr{L}'_{\ke} \ot V  \oplus \pi_* h_{\ke}^*T_{\PP (\CC ^N)}  )^{\vee} \right] = : \mathbb{E} \\
 & \quad \downarrow    \quad \quad \quad   \quad   \quad \quad   \downarrow { \cong}   \\
& \left[  \mathscr{I}/\mathscr{I}^2  \ra \Omega_{U^{\ke, N}_{ d'} /\fBun ^{1,0}} |_{Q^{\ke}_{X} \times _{\fM _{1,0}^{ss}} M_N}  \right] .
\end{align*} 
     The associated virtual class is, by definition,
     \[ [  Q_X^{\ke} \times_{\fM _{1,0}^{ss}}M_N    ]^{\vir} := 0^!_{  \pi _* \sW _{\ke, d}  }[\mathbf{C}_{\sigma ^{\ke}}] . \]

\subsubsection{Wall-crossing}
       
       We will compare the virtual classes $[  Q^{\pm}_{X} \times _{\fM _{1,0}^{ss}} M_N]^{\vir}  $ under the contraction map $c: Q^{+}_{X} \times _{\fM _{1,0}^{ss}} M_N \ra 
       Q^{-}_{X} \times _{\fM _{1,0}^{ss}} M_N$, where the contraction map does not do anything on the $M_N$ factor.
       
       The comparison can be carried out as before. Similar to \eqref{univ curves}, there is a commuting diagram 
 \begin{equation*}
\xymatrix{\fC ^{+}_{1, 0, d'} \ar[rd]_{\pi} \ar[r]\ar@/^1pc/[rr]^{\tilde{c}}  \ar@/^3pc/[rrr]_{ss_{+}} \ar@/^4pc/[rrrr]^{h_{+}} 
& c^*\fC ^{-}_{1, 0, d'}\ar[d]\ar[r]
&\fC ^{-}_{1, 0, d'}\ar[d]^{\pi} \ar@/^1pc/[rr]^{h_{-}} \ar[r]_{ss_{-}} &  \fC ^{ss}_{1,0} \ar[r]_{h} \ar[d] & \PP(\CC^N) \\
& U^{+, N}_{d'}  \ar[r]_{c}    & U^{-, N}_{d'} \ar[r]_{\text{proj}}  & M_N .& 
}
\end{equation*}

       First use the homomorphism $\Phi  :  P^{+}_{d'} \oplus R ^{+}_{d'} \ra c^*P^{-}_{d'} \oplus  c^*R ^{-}_{d'} $ 
       induced from the contraction map to perform the MacPherson graph construction.
       Second, deform the obstruction normal cone of $c^{-1} (Q^-_X\times _{\fM ^{ss}_{1,0}} M_N)$ in $U^{+, N}_{ d'} $
       using the induced section of the universal quotient bundle of $\mathbf{Gr}(P^{+}_{d'} \oplus R ^{+}_{d'} \oplus  c^*P^{-}_{d'} \oplus c^*R ^{-}_{d'} )$.
       
     Repeating word for word the arguments of \S3.3-3.6, we obtain the following analogue of Theorem \ref{Degeneration2}.
      Let $z$ be a formal variable. Let the Chow cohomology class $\mu_{d_a}^N(z)\in A^*(X\times\PP(\CC^N))_\QQ[z]$ be given
      by the universal formula \eqref{explicit mu}. The equality 
       \begin{align}\label{degen2_(1,0)}
&[Q^{-}_{1, 0} (X, d) \times _{\fM _{1,0}^{ss}} M_N]^{\mathrm{vir}} - c_*[Q^{+}_{1, 0}(X, d) \times _{\fM _{1,0}^{ss}} M_N]^{\mathrm{vir}} =\\
& \nonumber\sum_A \frac{1}{|A|!}
(b_A)_* (c_A)_* \left(\prod_{a\in A} (ev_a, h_a^+)^*\mu_{d_a}^N(z)|_{z=-\psi_a} \cap [Q^{+}_{1,A}(X,d_0^A)  \times _{\fM _{1,0}^{ss}} M_N ]^{\mathrm{vir}}\right )
\end{align}
holds in the Chow group $A_*(Q^{-}_{1,0}(X, d) \times _{\fM _{1,0}^{ss}} M_N)_\QQ$,
where 
\begin{itemize}
\item $c_A$ is the contraction map 
$$Q^{+}_{1,A}(X,d_0^A)  \times _{\fM _{1,0}^{ss}} M_N\ra Q^{-}_{1,A}(X,d_0^A)  \times _{\fM _{1,0}^{ss}} M_N,$$
\item
$b_A$ is the morphism $$Q^{-}_{1,A}(X,d_0^A)  \times _{\fM _{1,0}^{ss}} M_N \ra Q^{-}_{1,0}(X,d)  \times _{\fM _{1,0}^{ss}} M_N$$ 
which trades the markings $A$ for base 
points of length $d_a$, 
\item
the morphism $h_a^{+}:      Q^{+}_{1,A}(X,d_0^A)  \times _{\fM _{1,0}^{ss}} M_N   \ra  \PP (\CC ^N)$ 
is the composite of the contraction $$Q^{+}_{1,A}(X,d_0^A)\times _{\fM _{1,0}^{ss}} M_N\ra Q^{-}_{1,A}(X,d_0^A)\times _{\fM _{1,0}^{ss}} M_N,$$
the marking section $$\Sigma _a: Q^{-}_{1,A}(X,d_0^A)  \times _{\fM _{1,0}^{ss}} M_N \ra \fC ^{-}_{A, X}$$ of the universal
curve over $Q^{-}_{1,A}(X,d_0^A)  \times _{\fM _{1,0}^{ss}} M_N$ (associated to the universal curve of $Q^{-}_{1,A}(X,d_0^A)$), 
the morphism $$ \fC ^{-}_{A, X} \ra \fC ^{-}_{X}$$ induced from $b_A$, and finally $h_{-}|_{\fC ^{-}_X} : \fC^{-}_X \ra \PP (\CC ^N)$.
\end{itemize}

 \subsubsection{Relation between $[Q_X^{\ke} \times_{\fM _{1,0}^{ss}}M_N ]^{\vir}$ and $[Q_X^{\ke}]^{\vir}$}

By a result of Cooper, \cite{Cooper}, the stack $Q^{0+}_{1,0} (\PP (V), d)$ has projective coarse moduli
and hence there is a morphism from the universal curve of $Q^{0+}_{1, 0} (\PP(V) , d)$ to
$\PP (\CC ^N)$ for some $N$ such that the morphism does not contract any irreducible component of any fiber of the universal curve. 
Fix such a morphism $\phi$ and let $e$ be the degree of a fiber curve under $\phi$. The degree
$e$ is independent of the choice of fiber since $Q^{0+}_{1,0} (\PP(V), d)$  is connected.  (In fact, $Q^{0+}_{1,0} (\PP(V), d)$ is irreducible; this follows from
the connectedness of $\overline{M}_{1,0}(\PP (V), d)$ (see \cite{KP}), the surjectivity of the contraction map $\overline{M}_{1,0}(\PP (V), d) \ra Q^{0+}_{1,0} (\PP (V), d)$, and the
smoothness of $Q^{0+}_{1,0}(\PP (V), d)$ (see \cite{MOP}).) From now on {\it we work with the stack $M_N$ corresponding to these particular choices of 
$N$ and $e$}.

By the universal property of $M_N$, upon restricting $\phi$ to the universal curve over $Q^{0+}_X $,
we obtain a morphism $\underline{h}_{1,0}: Q^{0+}_X\ra M_N $ fitting in the diagram with the cartesian square
\[ \xymatrix{ \fC _{1, 0}^{0+} \ar[rr] \ar[d] \ar@/^1pc/[rrr]^{\phi}& & \fC_{1,0}^{ss} \ar[d] \ar[r]_{h} & \PP(\CC ^N)  \\
Q^{0+}_X \ar@/_1pc/[rr]_{\underline{h}_{1,0}  }  \ar[r] & Q^{0+}_{1,0} (\PP(V), d) \ar[r]&   M_N . &  }  \]
We also let $$\underline{h}^{\ke}_{1,0} : Q^{\ke}_X \ra Q^{0+}_X \stackrel{\underline{h}_{1,0}}{\longrightarrow}  M_N $$ denote 
the composition of $\underline{h}_{1,0} $ and the contraction
$Q^{\ke}_X \ra Q^{0+}_X$.

 One checks directly that there is a natural cartesian square
 \[    \xymatrix{ Q^{\ke}_{X}   \ar[r]^{(\mathrm{id},  \underline{h}^{\ke}_{1,0} ) \ \ \  \ \  } \ar[d]_{\underline{h}^{\ke}_{1,0} }  
        &  Q^{\ke}_{X} \ar[d]  \times _{\fM _{1,0}^{ss}} M_N
                \ar[d]^{(\underline{h}^{\ke}_{1,0} , \mathrm{id})} \\               
       M_N  \ar[r]^{\Delta \ \ \ \  \ }    &           M_N \times _{\fM _{1,0}^{ss}} M_N. } \]

  In the derived category of coherent sheaves on $Q^{\ke}_X$ there is a commuting diagram 
  \[ \xymatrix{  (\underline{h}_{1,0}^{\ke})^*  (L_{\Delta}[-1]   \cong (\pi_* h^*T_{\PP (\CC ^N)})^{\vee} )  \ar[r] \ar[d] 
  &   (\mathrm{id},  \underline{h}^{\ke}_{1,0} )^* \mathbb{E} \ar[d]   \\
                         \mathbb{L}_{   Q^{\ke}_{X}/  Q^{\ke}_X\times _{\fM ^{ss}_{1,0}} M_N   } [-1] \ar[r]
                          &  (\mathrm{id},  \underline{h}^{\ke}_{1,0} )^* \mathbb{L}_{Q^{\ke}_X\times _{\fM ^{ss}_{1,0}} M_N / \fBun ^{1,0}} 
                         } \]
      whose mapping cone  is the obstruction theory for  $Q^{\ke}_{X} $ relative to  $\fBun ^{1,0}$, as in \S \ref{POT}. 
      The functoriality result of \cite[Proposition 5.10]{BF} implies the relation
       \begin{align}\label{Delta^!} \Delta ^{!} [  Q^{\ke}_{X} \times _{\fM _{1,0}^{ss}} M_N]^{\vir}  = [ Q^{\ke}_X]^{\vir} . \end{align}

Now apply $\Delta ^{!}$ to \eqref{degen2_(1,0)}. Using
the compatibility of the Gysin homomorphism for proper push-forward, the commutativity of Chern
classes with Gysin homomorphism, the relation \eqref{Delta^!}, and the identification of $\mu_{d_a}^N(z)$ from \S\ref{identification},
we conclude the proof of Theorem  \ref{Main} in the remaining case $(g,k)=(1,0)$.

\end{document}